\documentclass[reqno]{amsart}

\usepackage{cite}
\usepackage[pagebackref]{hyperref}
\usepackage[all]{xy}
\usepackage{amsmath}
\usepackage{hyperref}
\usepackage{amsfonts,graphics,amsthm,amsfonts,amscd,latexsym}
\usepackage{epsfig}
\usepackage{flafter}
\usepackage{mathtools}
\usepackage{comment}
\usepackage{stmaryrd}
\usepackage{tabularx}
\usepackage{pst-node}
\usepackage{array, longtable}
\usepackage{pdfsync}
\usepackage{tikz-cd}
\hypersetup{
    colorlinks=true,    
    linkcolor=blue,          
    citecolor=blue,      
    filecolor=blue,      
    urlcolor=blue           
}

\usepackage{pgfplots}
\pgfplotsset{compat=1.10}
\usepackage{mathrsfs}
\usepackage{tikz}
\usetikzlibrary{graphs,positioning,arrows,shapes.misc,decorations.pathmorphing}

\tikzset{
    >=stealth,
    every picture/.style={thick},
    graphs/every graph/.style={empty nodes},
}

\tikzstyle{vertex}=[
    draw,
    circle,
    fill=black,
    inner sep=1pt,
    minimum width=5pt,
]
\usepackage[position=top]{subfig}
\usepackage{amssymb}
\usepackage{color}
\usepackage{bm}

\calclayout

\usetikzlibrary{decorations.pathmorphing}
\tikzstyle{printersafe}=[decoration={snake,amplitude=0pt}]

\newcommand{\msp}{\mathsf p}
\newcommand{\msq}{\mathsf q}
\newcommand{\dd}{\mathop{}\!\mathrm{d}}
\newcommand{\ZZ}{{\mathbb Z}}

\newcommand{\lct}{\operatorname{lct}}
\newcommand{\Proj}{\operatorname{Proj}}
\newcommand{\PP}{{\mathbb P}}
\newcommand{\wt}{\operatorname{wt}}
\newcommand{\QQ}{{\mathbb Q}}
\newcommand{\CC}{{\mathbb C}}
\newcommand{\mult}{\operatorname{mult}}

\newcommand{\bA}{{\mathbb A}}
\newcommand{\bZ}{{\mathbb Z}}
\newcommand{\bN}{{\mathbb N}}

\newcommand{\bQ}{{\mathbb Q}}
\newcommand{\bP}{{\mathbb P}}
\newcommand{\bR}{{\mathbb R}}

\newcommand{\bG}{\mathbb{G}}

\newcommand{\cF}{\mathcal{F}}
\newcommand{\cE}{\mathcal{E}}
\newcommand{\cX}{\mathcal{X}}
\newcommand{\cD}{\mathcal{D}}
\newcommand{\cH}{\mathcal{H}}
\newcommand{\cO}{\mathcal{O}}
\newcommand{\cP}{\mathcal{P}}
\newcommand{\cW}{\mathcal{W}}
\newcommand{\GL}{\mathrm{GL}}
\newcommand{\pr}{\mathrm{pr}}
\newcommand{\bmu}{\bm{\mu}}
\newcommand{\Aut}{\mathrm{Aut}}
\newcommand{\Fut}{\mathrm{Fut}}
\newcommand{\rank}{\mathrm{rank}\,}

\newtheorem{theorem}{Theorem}[section]
\newtheorem{lemma}[theorem]{Lemma}
\newtheorem{proposition}[theorem]{Proposition}
\newtheorem{definition}[theorem]{Definition}
\newtheorem{example}[theorem]{Example}
\newtheorem{corollary}[theorem]{Corollary}
\newtheorem{remark}[theorem]{Remark}

\numberwithin{equation}{section}

\DeclareMathOperator{\ord}{ord}

\DeclareMathOperator{\vol}{vol}
\def\Q{\mathbb{Q}}

\theoremstyle{remark}

\numberwithin{equation}{section}


\subjclass[]{}

\begin{document}

\title[Wall-crossing for K-MODULI SPACES of WEIGHTED PROJECTIVE HYPERSURFACES]{Wall-crossing for K-MODULI SPACES of certain families of WEIGHTED PROJECTIVE HYPERSURFACES}

\begin{abstract}
We describe the K-moduli spaces of  weighted hypersurfaces of degree $2(n+3)$ in $\mathbb{P}(1,2,n+2,n+3)$.  We show that the K-polystable limits of these weighted hypersurfaces are also weighted hypersurfaces of the same degree in the same weighted projective space. This is achieved by an explicit study of the wall crossing for K-moduli spaces $M_w$ of certain log Fano pairs with coefficient $w$ whose double cover gives the weighted hypersurface. Moreover, we show that the wall crossing of $M_w$ coincides with variation of GIT except at the last K-moduli wall which gives a divisorial contraction. Our K-moduli spaces provide new birational models for some natural loci in the moduli space of marked hyperelliptic curves.
\end{abstract}

\author{In-Kyun Kim}
\address{June E Huh Center for Mathematical Challenges, Korea Institute for Advanced Study, 85 Hoegiro Dongdaemun-gu, Seoul 02455, Republic of Korea.} \email{soulcraw@kias.re.kr}

\author{Yuchen Liu}
\address{Department of Mathematics, Northwestern University, 2033 Sheridan Rd, Evanston, IL 60208, USA} \email{yuchenl@northwestern.edu}

\author{Chengxi Wang}
\address{Yau Mathematical Sciences Center,
Jingzhai, Tsinghua University, Haidian District,
Beijing, China, 100084} \email{chxwang@tsinghua.edu.cn}

\maketitle

\setcounter{tocdepth}{1}
\tableofcontents

\section{Introduction}
The notion of K-stability of Fano manifolds is an important concept in algebraic geometry and differential geometry. In \cite{Tia97, Don02} it was introduced to study the existence of K\"ahler-Einstein metrics on Fano manifolds, which is an important problem in differential geometry. Moreover, K-stability of Fano manifolds has attracted significant attention due to its connections with several important problems in algebraic geometry and related fields.

In algebraic geometry, to construct moduli spaces, which parametrize families of algebraic varieties with certain properties, is an important problem. In the case of Fano varieties, the challenging problem of constructing a well-behaved moduli space has been solved as an outcome of the recent advances in the theory of K-stability. 
The algebraic construction for K-moduli stacks and spaces of  log Fano pairs is completed thanks to many people's work \cite{ABHLX20,BHLLX21,BLX22,BX19,CP21,Jia20,LWX21,LXZ22,Xu20,XZ20,XZ21}. Although the foundational theory of K-moduli spaces have been established, it remains an important question to understand K-moduli spaces of explicit Fano varieties. There has been much work on this topic, mostly focusing on K-moduli spaces coming from compactifications of smooth Fano varieties. See \cite{MM93, OSS16, SS17, LX19, Liu22, ADL22, ACD+, CDGF, DJKHQ24, LZ24} for a partial list of works along this direction. 

In this paper, we consider K-moduli spaces for certain families of weighted projective hypersurfaces that are not necessarily $\bQ$-Gorenstein smoothable. More precisely, we consider well-formed weighted projective hypersurfaces $H_{2(n+3)}$ of degree $2(n+3)$ in $\mathbb{P}(1,2,n+2,n+3)$. By \cite[Theorem 1.0.6]{KVNW}, such a hypersurface $H_{2(n+3)}$  is K-stable if it is quasi-smooth. Our main result gives a description of the  K-moduli space $\mathcal{F}_n$ parameterizing these K-stable quasi-smooth hypersurfaces $H_{2(n+3)}$ and their K-polystable limits. 

\begin{theorem}\label{thm:main}
Let $n$ be a positive integer. Then every K-polystable member of $\mathcal{F}_n$ is a weighted hypersurface of degree $2(n+3)$ in $\bP(1,2,n+2,n+3)_{x,y,z,t}$. Moreover, we have the following detailed description. Let \begin{equation*}
  \tilde{f}= t^2 + z^2y + azx^{n+4} + a_0x^{2n+6} + a_1x^{2n+4}y + \cdots + a_{n+3}y^{n+3}.
\end{equation*}

\begin{enumerate}
    \item If $n$ is odd, then the K-polystable members in the K-moduli space $\mathcal{F}_n$ are precisely all the hypersurfaces $H_{2(n+3)}:\tilde{f}=0$ 
    such that $\{a,a_0,a_1,\cdots,a_{\frac{n+3}{2}}\}$ are not all zero and $\{a_{\frac{n+3}{2}+1},\cdots,a_{n+3}\}$ are not all zero.
    \item If $n=2l$ is even, then the K-polystable members in the K-moduli space $\mathcal{F}_n$ are precisely all the hypersurfaces $H_{2(n+3)}:\tilde{f}=0$ such that one of the following holds:\begin{itemize}
         \item $\{a,a_0,\cdots,a_{l+1}\}$ are not all zero 
         and $\{a_{l+3},\cdots,a_{2l+3}\}$ are not all zero;
         \item $a=a_0=\cdots=a_{l+1}=0$, $a_{l+2}\neq0$ and $a_{l+3}=\cdots=a_{2l+3}=0$.
     \end{itemize}
   
\end{enumerate}

\end{theorem}

As Theorem \ref{thm:main} suggests, the K-moduli spaces $\cF_n$  parameterizes hypersurfaces of the same degree in the same  weighted projective space. There are other families of Fano (weighted) hypersurfaces or complete intersections whose K-moduli compactifications exhibit a similar behavior, see e.g. \cite{SS17, LX19, Liu22, LP20}. On the other hand, this phenomenon is far from true in general, see e.g. \cite{OSS16, ADL22, ACKLP}.

To prove Theorem \ref{thm:main}, we first reduce to K-moduli of log Fano pairs by double cover constructions. It is easy to observe that the projection $H_{2(n+3)}\to \bP(1,2,n+2)$ by $[x,y,z,t]\mapsto [x,y,z]$ shows that $H_{2(n+3)}$ is a double cover of $W = \bP(1,2,n+2)$ branched along a divisor $D$ of degree $2n+6$.
This description works well when $n$ is odd. If $n=2l$ is even, then the weighted projective space $\bP(1,2,n+2)$ is no longer well-formed, so we need to modify the setup. We may treat $W = (W', \frac{1}{2}H)$ as an orbifold, where $W' = \bP(1,1,l+1)_{u, y,z}$ and $H = (u=0)$. Then the map $H_{2(n+2)}\to W'$ by $[x,y,z,t]\mapsto [x^2, y, z]$ is a double cover branched along the divisor $H+D$ where $D$ has degree $2l+3$.
Since K-polystability is preserved under finite Galois covers \cite{LiuZhu, Zhu21}, we know that K-polystability of $H_{2(n+3)}$ is equivalent to that of $(W, \frac{1}{2}D)$ (resp. of $(W', \frac{1}{2}(H + D))$ when $n$ is odd (resp. even). Then without too much effort we can generalize this double cover construction to families and get an isomorphism between $\cF_n$ and the K-moduli space $M_{\frac{1}{2}}$ of pairs $(W, \frac{1}{2}D)$ or $(W', \frac{1}{2}(H+D))$ depending on the parity of $n$.

Next, we use the wall crossing framework for K-moduli of log Fano pairs developed in \cite{ADL19, Zh23, Zh23c} to study the K-moduli spaces $M_w$ of $(W, wD)$ (resp. $(W', \frac{1}{2}H+ wD)$) when $n$ is odd (resp. even) as the coefficient $w$ varies. This framework was initially introduced in \cite{ADL19} for $\bQ$-Gorenstein smoothable log Fano pairs $(X, cD)$ where $D$ is proportional to $-K_X$ based on \cite{LWX19} with analytic inputs from \cite{CDS15, Tia15}. Later, in \cite{Zh23} the $\bQ$-Gorenstein smoothable assumption was removed thanks to the algebraic proof of the K-moduli theorem. Recently, in \cite{Zh23c} this framework was generalized to multiple divisors such that each component is proportional to $-K_X$. Since the pairs we consider are not necessarily $\bQ$-Gorenstein smoothable and may contain two components, we will follow the setup of \cite{Zh23, Zh23c}. See also \cite{Fuj17, GMGS18, ADL21, ADL22, Zha22, Pap22, Zho23, PSW23} for examples on wall crossing for K-moduli spaces.

In our setting, the advantage of working with wall crossing of $M_w$ rather than focusing on the single K-moduli $M_{\frac{1}{2}}$ is that usually $M_w$ is simple when $w$ is small, and as $w$ increases to $\frac{1}{2}$ we can break down the complexity in $M_{\frac{1}{2}}$ into elementary transformations given by a sequence of wall-crossing diagrams. As it turns out, for a majority of the coefficients $w$, our K-moduli space $M_w$ is isomorphic to the corresponding variation of GIT moduli space; see Theorems \ref{isomorphism,GITandK} and \ref{isomorphism,GITandK,even}. Then after the last wall $\xi_n$, the K-moduli wall crossing extracts a divisor that parameterizes pairs of a different form; see Theorem \ref{intro,K-moduliparameterize,i}(6) and Theorem \ref{intro,K-moduliparameterize,even,i}(6).

We first describe the K-moduli spaces $M_w$ when $n$ is odd.  

\begin{theorem}
\label{intro,K-moduliparameterize,i} \
  Let $n$ be an odd positive integer. For $w\in (0, \frac{n+5}{2n+6})\cap \bQ$,  let $M_w$ be the K-moduli spaces  parameterizing K-polystable log Fano pairs $(W=\bP(1,2,n+2)_{x,y,z}, wD)$ and their K-polystable limits, where $D= (f=0)$ and 
\[
   f:= z^2y + azx^{n+4} + a_0x^{2n+6} + a_1x^{2n+4}y + \cdots + a_{n+3}y^{n+3}. 
\]
Then the walls for $M_w$ are given by 
  \[
  w_{i}:=\frac{(n+2)^2 - (2n+1)i}{(n+2)(2n+6)-(4n+6)i}, \quad \textrm{ where }0\leq i\leq \frac{n+3}{2},
  \]
  and 
\[
\xi_n:=\frac{n^3 + 11n^2 + 31n + 23}{2n^3 + 18n^2 + 50n + 42}.
\]
  Moreover, we have the following explicit description of $M_w$.
 \begin{enumerate}
    \item When $w\in (0, w_{0}=\frac{n+2}{2n+6})$, the K-moduli space $M_w$ is empty.
    \item The K-moduli space $M_{w_0}$ consists a single point parameterizing the K-polystable pair $(W,w_0 D)$ such that $a=a_0=\cdots=a_{n+2}=0$ and $a_{n+3}\neq 0$, i.e., $D:z^2y+a_{n+3}y^{n+3}=0$.
     \item The K-moduli space $M_{w_i}$ parameterizes all pairs $(W,w_i D)$ from one of the following two cases:
     \begin{itemize}
         \item $\{a,a_0,\cdots,a_{n+2-i}\}$ are not all zero, $a_{n+3-i}$ is arbitrary and $\{a_{n+4-i},\cdots,a_{n+3}\}$ are not all zero;
         \item $a=a_0=\cdots=a_{n+2-i}=0$, $a_{n+3-i}\neq0$, and $a_{n+4-i}=\cdots=a_{n+3}=0$. 
     \end{itemize}
     \item If $w\in (w_{i},w_{i+1})$ where $0\leq i\leq \frac{n-1}{2}$, then $M_w$ parameterizing all pairs $(W,wD)$ such that $\{a,a_0,\cdots,a_{n+2-i}\}$ are not all zero, and $\{a_{n+3-i},\cdots,a_{n+3}\}$ are not all zero.
     \item If $w\in (w_{\frac{n+1}{2}},\xi_n)$,
     then $M_w$ parameterizing all pairs $(W,w D)$ such that $\{a,a_0,\cdots,a_{\frac{n+3}{2}}\}$ are not all zero, and $\{a_{\frac{n+3}{2}+1},\cdots,a_{n+3}\}$ are not all zero.
     \item We describe the wall-crossing at $\xi_n$. Assume $D_{\ss}$ satisfies $a\neq 0$, $a_{n+3}\neq 0$ and has a type $A_{n+3}$-singularity (see \eqref{eq:Dss} for its equation). 
     Let $W_0=\PP(1,n+2,\frac{(n+3)^2}{2})_{x_0,x_1,x_2}$, $D_0$ be a curve defined by $x_2^2-x_0x_1^{n+4}=0$, and $D_1$ be a curve defined by \begin{equation*}
x^2_2-x_0x_1^{n+4}+b_{n+2}x_0^{2n+5}x_1^{n+2}+\cdots+b_{1}x_0^{(n+3)(n+2)+1}x_1+b_{0}x_0^{(n+4)(n+2)+1}=0,
\end{equation*} where coefficients $b_0,\cdots,b_{n+2}$ are not all zero.
Then there is a wall crossing diagram
\[
M_{\xi_n -\epsilon }\xrightarrow[\cong]{\phi^-} M_{\xi_n}\xleftarrow{\phi^+}M_{\xi_n +\epsilon}
\]
for $0<\epsilon \ll 1$ such that the following hold.
\begin{enumerate}
    \item $\phi^-$ is an isomorphism that only replaces $(W, D_{\ss})$ by $(W_0, D_0)$.
    \item $\phi^+$ is a divisorial contraction with only one exceptional divisor $E^+$. Moreover, we have $\phi^+(E^+)= \{[(W_0, D_0)]\}$, and $E^+$ parameterizes pairs $(W_0, D_1)$. 
\end{enumerate}
Moreover, $M_w \cong M_{\xi_n+\epsilon}$ for $w\in (\xi_n, \frac{n+5}{2n+6})$.
 \end{enumerate}      
\end{theorem}

Next, we describe the K-moduli spaces $M_w$ when $n$ is even.

\begin{theorem}\label{intro,K-moduliparameterize,even,i}
  Let $n=2l$ be an even positive integer. For $w\in (0, \frac{2l+5}{4l+6})\cap \bQ$,  let $M_w$ be the K-moduli spaces  parameterizing K-polystable log Fano pairs $(W'=\bP(1,1,l+1)_{u,y,z}, \frac{1}{2}H +wD)$ and their K-polystable limits, where $H=(u=0)$, $D= (f=0)$ and 
\[
f:= z^2y + azu^{l+2} + a_0u^{2l+3} + a_1u^{2l+2}y + \cdots + a_{2l+3}y^{2l+3}.
\]
Then the walls for $M_w$ are given by 
  \[
  w_{i}:=\frac{(n+2)^2 - (2n+1)i}{(n+2)(2n+6)-(4n+6)i}, \quad \textrm{ where }0\leq i\leq l+1,
  \]
  and
\[
\xi_{2l}:=\frac{2l^2 + 8l + 3}{4l^2 + 12l + 6}.
\]  
  Moreover, we have the following explicit description of $M_w$.
 \begin{enumerate}
     \item When $w\in (0, w_{0}=\frac{l+1}{2l+3})$, the K-moduli space $M_w$ is empty.
     \item The K-moduli space $M_{w_0}$ consists of a single point parameterizing the K-polystable pair $(W',\frac{1}{2}H +w_0 D)$ such that $a=a_0=\cdots=a_{2l+2}=0$ and $a_{2l+3}\neq 0$, i.e., $D:z^2y+a_{2l+3}y^{2l+3}=0$.
     \item  The K-moduli space $M_{w_i}$ parameterizes all pairs $(W',\frac{1}{2}H +w_i D)$ from one of the following two cases:
     \begin{itemize}
         \item $\{a,a_0,\cdots,a_{2l+2-i}\}$ are not all zero, $a_{2l+3-i}$ is arbitrary, and $\{a_{2l+4-i},\cdots,a_{2l+3}\}$ are not all zero;
         \item $a=a_0=\cdots=a_{2l+2-i}=0$, $a_{2l+3-i}\neq0$ and $a_{{2l+4-i}}=\cdots=a_{2l+3}=0$.
     \end{itemize}
     \item If $w\in (w_{i},w_{i+1})$ where $0\leq i\leq l$, then $M_w$ parameterizing all pairs $(W',\frac{1}{2}H +wD)$ such that $\{a,a_0,\cdots,a_{2l+2-i}\}$ are not all zero, and $\{a_{2l+3-i},\cdots,a_{2l+3}\}$ are not all zero.
     \item If $w\in (w_{l+1}, \xi_{2l})$,
     then $M_w$ parameterizing all pairs $(W',\frac{1}{2}H+wD)$ such that $\{a,a_0,\cdots,a_{l+1}\}$ are not all zero, and $\{a_{l+2},\cdots,a_{2l+3}\}$ are not all zero.
     \item We describe the wall-crossing at $\xi_{2l}$. Assume $D_{\ss}$ satisfies $a\neq 0$, $a_{2l+3}\neq 0$ and has a type $A_{2l+3}$-singularity (see \eqref{eq:Dss-even} for its equation). 
     Let $H_0:y=0$ and $D_1$ be a curve defined by
     \begin{equation*}
z^2 y - z u^{l+2} + b_1 u^{2l+2} y + b_2 u^{2l+1} y^2 + \cdots + b_{2l+3} y^{2l+3} = 0,
\end{equation*}
where coefficients $b_1,\cdots,b_{2l+3}$ are not all zero. 
Then there is a wall crossing diagram
\[
M_{\xi_{2l} -\epsilon }\xrightarrow[\cong]{\phi^-} M_{\xi_{2l}}\xleftarrow{\phi^+}M_{\xi_{2l} +\epsilon}
\]
for $0<\epsilon \ll 1$ such that the following hold.
\begin{enumerate}
    \item $\phi^-$ is an isomorphism that only replaces $(W', \frac{1}{2}H+D_{\ss})$ by $(W', \frac{1}{2}H_0+D_0)$.
    \item $\phi^+$ is a divisorial contraction with only one exceptional divisor $E^+$. Moreover, we have $\phi^+(E^+)= \{[(W', \frac{1}{2}H_0+D_0)]\}$, and $E^+$ parameterizes pairs $(W', \frac{1}{2}H_0+D_1)$. 
\end{enumerate}
Moreover, we have $M_w\cong M_{\xi_{2l}+\epsilon}$ for $w\in (\xi_{2l},\frac{2l+5}{4l+6})$.
     \end{enumerate}    
     \end{theorem}

Finally, we note that our K-moduli spaces provide new birational models for some natural loci in the moduli space of marked hyperelliptic curves. See Section \ref{sec:hyperelliptic} for details.

\subsection*{Organization} Our paper is organized as follows. In Section \ref{Prelimi}, we review the background, necessary definitions and set-ups about K-stability of log Fano pairs, and K-moduli spaces of log Fano pairs and wall-crossing. In Section \ref{VGIT}, we describe our hypersurfaces from the perspective of variation of GIT. Section \ref{section5:section number} contains the computation of log canonical thresholds for the log pairs. The results will be used in the proofs in Section \ref{section,nodd} and Section \ref{section,even}. In Section \ref{section,nodd}, we give description of K-moduli spaces and wall-crossing before $\xi_n$ when $n$ is odd and prove Theorem \ref{intro,K-moduliparameterize,i}(1)-(5). In Section \ref{section,even}, we 
investigate the K-moduli spaces and wall-crossing before $\xi_n$ when $n$ is even and prove Theorem \ref{intro,K-moduliparameterize,even,i}(1)-(5). In Section \ref{moduliafterxi_n}, we describe the K-moduli spaces after $\xi_n$ and prove Theorem \ref{intro,K-moduliparameterize,i}(6) and Theorem \ref{intro,K-moduliparameterize,even,i}(6). We also discuss the relation of our K-moduli spaces with moduli of marked hyperelliptic curves.
 In Section \ref{NonCYcase}, we focus on proving K-semistability of $(W, wD_{\ss})$ and $(W', \frac{1}{2}H + wD_{\ss})$ for $\frac{1}{2}\leq w\leq \xi_n$. The proof is mainly based on computations using the Abban-Zhuang method \cite{AZ22} (see also Section \ref{prelimi,Abban-Zhuang}).

\subsection*{Acknowledgement} We thank Martin Bishop and  Burt Totaro for some helpful conversations. IK is supported by the National Research Foundation of Korea (NRF-2023R1A2C1003390). YL is partially supported by NSF CAREER Grant DMS-2237139 and the Alfred P. Sloan Foundation.

\section{Preliminaries}\label{Prelimi}

We work over the complex numbers.
In this section, we give a brief overview of the K-stability and the K-moduli spaces of log Fano pairs. We also introduce some computation methods about K-stability of a pair of a singular del Pezzo surface and a curve on it.
The singular del Pezzo surfaces considered in the paper are weighted projective planes.

\subsection{Foundation}
In this paper we study log pairs consisting of weighted projective spaces and $\QQ$-divisors. We define the weighted projective space.  Let $N$ be a positive integer. For positive integers $a_0,\ldots , a_N$, let
\begin{equation*}
    R[a_0, \ldots, a_N]\coloneqq \CC[x_0,\ldots, x_N]
\end{equation*}
be the graded ring where the variable $x_i$ has weight $\wt(x_i) = a_i$. We define
\begin{equation*}
    \PP(a_0,\ldots, a_N) = \Proj R[a_0, \ldots, a_N]
\end{equation*}
and call it the weighted projective space with homogeneous coordinates $x_0,\ldots , x_N$. We sometimes denote
\begin{equation*}
    \PP(a_0,\ldots, a_N)_{x_0,\ldots , x_N}
\end{equation*}
in order to make it clear the homogeneous coordinates $x_0,\ldots , x_N$. For $i=0,\ldots, N$ we denoted by
\begin{equation*}
    \msp_{x_i}=[0:\cdots :1:\cdots : 0]\in \PP(a_0,\ldots , a_N)
\end{equation*}
the coordinate point at which only the coordinate $x_i$ does not vanish.
For $i=0,\ldots , N$ we define
\begin{equation*}
    H_{x_i}\coloneqq (x_i = 0)\subset \PP(a_0,\ldots, a_N)
\end{equation*}
and 
\begin{equation*}
    U_{x_i}\coloneqq \PP(a_0,\ldots, a_N)\setminus H_{x_i}.
\end{equation*}
We call $U_{x_i}$ the standard affine open subset of $\PP(a_0,\ldots , a_N)$ containing the point $\msp_{x_i}$.
\begin{definition}
    Let $(X,\Delta)$ be a pair, D an effective $\QQ$-divisor on $X$, and let $\msp\in X$ be a point. Assume that $(X,\Delta)$ has at most log canonical singularities. We define the log canonical threshold of $D$ at $\msp$ with respect to the log pair $(X, \Delta)$ and the log canonical threshold of $D$ with respect to the log pair $(X, \Delta)$ to be the numbers
    \begin{equation*}
        \begin{split}
            \lct_{\msp}(X, \Delta;D) &= \sup\{c\in\QQ~|~(X,\Delta+cD) \textrm{~is log canonical at~}\msp\},\\
            \lct(X, \Delta;D) &= \sup\{c\in\QQ~|~(X,\Delta+cD) \textrm{~is log canonical}\},
        \end{split}
    \end{equation*}
    respectively. We set $\lct_{\msp} (X, D) = \lct_{\msp} (X, 0; D)$ and $\lct_{\msp} (X, D) = \lct(X, \Delta; D)$ when $\Delta = 0$.
\end{definition}

Let $S$ be a surface with at most cyclic quotient singularities and $\msp$ a cyclic quotient singular point of type $\frac{1}{r}(a,b)$ on $S$, where $a,b,r$ are pairwise coprime integers. Then there is an orbifold chart $\phi\colon \widetilde{U}\to U$ for some open set $\msp\in U$ on $S$ such that $\widetilde{U}$ is smooth and $\phi$ is a cyclic cover of degree $r$ branched over $\msp$. For any $\QQ$-divisor $B$ on $S$ we write $\mult_{\msp}(B)$ for $\mult_{\msq}(\phi^*(B|_U))$ where $\msq\in \widetilde{U}$ is a point such that $\phi(\msq) = \msp$.

\begin{lemma}\label{Lemma:mult}
    Notation as above. Let $D$ be an effective $\QQ$-divisor on $S$. Suppose that log pair $(S, D)$ is not log canonical at the point $\msp$. Let $C$ be an integral curve on $S$ that passes through the point $\msp$. Then
    \begin{enumerate}
        \item the inequality $\mult_{\msp}(D)>1$ holds. 
        \item If $C$ is not contained in the support of $D$, then the inequality $C\cdot D > \frac{1}{r}$ holds.
    \end{enumerate}
     
\end{lemma} 

\begin{proof}
 (1)    See \cite[Proposition~9.5.13]{La04}, for instance.

 (2)   This follows from \cite[Proposition~3.16]{Kollar97}, Lemma \ref{Lemma:mult}, and \cite[Lemma~2.2]{CPS10}.
\end{proof}

\subsection{K-stability of log Fano pairs}
\begin{definition}
 Let $(X,\Delta)$ be a log pair, that is, $X$ is a normal variety and $\Delta$ is an effective $\QQ$-divisor on $X$ such that $K_X+\Delta$ is $\QQ$-Cartier. We call $(X,\Delta)$ a log Fano pair if additionally $X$ is projective and $-(K_X+\Delta)$ is a $\mathbb{Q}$-Cartier ample $\bQ$-divisor. A normal projective variety $X$ is called $\mathbb{Q}$-Fano if $(X,0)$ is a klt log Fano pair.
\end{definition}

\begin{definition}
Let $E$ be a prime divisor over a normal variety $X$, i.e. there is a proper birational morphism $\sigma:Y \rightarrow X$ from a normal variety $Y$ and $E$ is a divisor on $Y$. For a log pair $(X,\Delta)$  the log discrepancy $A_{X, \Delta}(E)$ of $(X,\Delta)$ along $E$ is defined to be
\begin{equation*}
    A_{X, \Delta}(E) = \ord_E(K_Y - \sigma^*(K_X + \Delta)) + 1.
\end{equation*}
We define the S-invariant $S_{X, \Delta}(E)$ of the log pair $(X, \Delta)$ along $E$ as
\begin{equation}\label{pre:S-invariant}
    S_{X,\Delta}(E)=\frac{1}{\vol(-(K_X+\Delta))}\int^{\tau}_0\vol(\sigma^*(-(K_X+\Delta))-tE) \dd t
\end{equation}
where $\tau=\tau(E)=\mathrm{sup}\{t\in \Q_{\geq 0} \ | \ \sigma^*(-(K_X+\Delta))-tE \text{ is big }\}$ is the pseudo-effective threshold of $E$ with respect to $-(K_X+\Delta)$. 
We set
\begin{equation*}
    \beta_{X,\Delta}(E)=A_{X,\Delta}(E)-S_{X,\Delta}(E).
\end{equation*}
We call $\beta_{X,\Delta}(E)$ the $\beta$-invariant of the log pair $(X, \Delta)$ along $E$.
\end{definition}

\begin{lemma}\label{Sinv,k}
Let $X=\PP(a,b,c)$ be a well-form weighted projective plane, $E=\mathcal{O}(k)$ for some $k\in\mathbb{Z}_{>0}$ be a $\mathbb{Q}$-divisors on $X$, and $\Delta$ be a divisor of degree $r$ on $X$. Then $S_{X,\Delta}(E)=\frac{r}{3k}$.
\end{lemma}

\begin{proof}
    By \eqref{pre:S-invariant}, we have 
    \[S_{X,\mathcal{O}(1)}(E)=\frac{1}{(\mathcal{O}(1))^2}\int^{\frac{1}{k}}_{0}\mathrm{vol}(\mathcal{O}(1)-tE)dt=\int^{\frac{1}{k}}_{0}(1-tk)^2dt=\frac{1}{3k}.\]
    Since $\Delta\in|\mathcal{O}(r)|$, we get $S_{X,\Delta}(E)=rS_{X,\mathcal{O}(1)}(E)$. Hence the result is true.
\end{proof}

K-stability of a log Fano pair was originally defined using test configurations (see e.g. \cite[Definition 2.6.]{Xu21}). The following theorem gives the Fujita--Li's valuation criteria of K-stability which we also take as a definition.

\begin{theorem}[\cite{Fuj19b,Li17,BX19}]
A log Fano pair $(X,\Delta)$ is K-semistable (resp. K-stable) if and only if $\beta_{X,\Delta}(E) \geq 0 $ (resp. $>0$) for any prime divisors $E$ over $X$.
\end{theorem}

We refer reader for to \cite[Definition 2.3]{ADL19} for the definition of test configuration and specially degeneration. The following theorem is a direct consequence of  the openness of K-semistability \cite{BLX22, Xu20}.
\begin{theorem}[\cite{BLX22, Xu20}]\label{openness}
 Assume $(X,\Delta)$ is a klt log Fano pair which specially degenerates to a log Fano pair $(X_0, \Delta_0)$. If $(X_0, \Delta_0)$ is K-semistable, then $(X,\Delta)$ is also K-semistable.
\end{theorem}

We consider a klt log Fano pair $(X,\Delta)$ with an action by the maximal torus $\mathbb{T}$ in $\mathrm{Aut}(X,\Delta)$.

\begin{definition}[{\cite[Definition 1.3.3.]{9p}}]
We call the divisor $F$ on $X$ vertical if a maximal $\mathbb{T}$-orbit in $F$ has the same dimension as the torus $\mathbb{T}$. Otherwise, the divisor $F$ is said to be horizontal. 
\end{definition}

The complexity of the $\mathbb{T}$-action on $(X,\Delta)$ is the number $\mathrm{dim}(X)-\mathrm{dim}(\mathbb{T})$.
When the complexity of the $\mathbb{T}$-action on $(X,\Delta)$ is one, we call $(X,\Delta)$ a complexity one $\mathbb{T}$-pair. For a complexity one $\mathbb{T}$-pair, the following theorem gives a criterion for K-polystability (cf. \cite{IS17}, \cite[Theorem 1.3.9]{9p},\cite[Theorem 3.2]{L23}). In our paper, we use the criterion for log Fano surface pairs $(X,\Delta)$.

\begin{theorem}[\cite{IS17,9p,L23}]\label{Tone}
Let $(X,\Delta)$ be a log Fano pair with an algebra torus $\mathbb{T}$-action of complexity one. Assume in addition that $(X,\Delta)$ is not toric.  Then $(X,\Delta)$ is K-polystable (resp. K-semistable) if and only if all of the following statements are true.
\begin{enumerate}
    \item $\beta_{X,\Delta}(F)>0$ (resp. $\beta_{X,\Delta}(F)\geq 0$)  for each vertical $\mathbb{T}$-equivariant prime divisor $F$ on $X$;
    \item $\beta_{X,\Delta}(F)=0$ for each horizontal $\mathbb{T}$-equivariant prime divisor $F$ on $X$;
    \item $\mathrm{Fut}_{X,\Delta}=0$ on the cocharacter lattice of $\mathbb{T}$.
\end{enumerate}
\end{theorem}

\begin{remark}\label{remarkTone}
For a klt log Fano surface pair $(X,\Delta)$, the complexity of the $\mathbb{T}$-action on $(X,\Delta)$ being one means $\mathbb{T}=\mathbb{G}_m$. The divisor $F$ on the surface $X$ is horizontal if and only if all $\mathbb{G}_m$-orbits in $F$ are just points, i.e. $\mathbb{G}_m$ acts identity on the divisor.
For a log Fano surface pair $(X,\Delta)$, the condition (3) in Theorem \ref{Tone} is equivalent to  $\beta_{X,\Delta}(v)=0$ for the valuation $v$ induced by $\lambda$.
\end{remark}

 We introduce another invariant of log Fano pairs which we call $\delta$-invariant. It serves a criterion for K-stability.

\begin{definition}[\cite{FO18,BJ17}]
Let $(X,\Delta)$ be a klt log Fano pair, we define the stability threshold to be 
\begin{equation*}
    \delta(X,\Delta)\coloneqq \inf_{E} \frac{A_{X,\Delta}(E)}{S_{X,\Delta}(E)},
\end{equation*}
where the infimum is taken over all prime divisors $E$ over $X$.
\end{definition}

\begin{theorem}[\cite{Fuj19b, Li17, FO18,BJ17}]
A klt log Fano pair $(X,\Delta)$ is K-semistable if and only if $\delta(X,\Delta)\geq 1$.
\end{theorem}

In this paper, we use both local and global methods to check K-stability. We now state local analogues of the stability threshold.

\begin{definition}[{\cite[Definition 2.5]{AZ22}}]
Let $(X,\Delta)$ be a log Fano pair and $\msp$ be a point in $X$, we define the $\delta$-invariant at a point $\msp$ with respect to $(X,\Delta)$ to be
\[
\delta_{\msp}(X,\Delta)\coloneqq\inf_{\msp\in C(E)}\frac{A_{X,\Delta}(E)}{S_{X,\Delta}(E)}
\]
where infimum is taken all prime divisors $E$ over $X$ whose centers on $X$ contain $\msp$.  
\end{definition}

\subsection{Abban-Zhuang method}\label{prelimi,Abban-Zhuang}
Let $(X,\Delta)$ be a klt log del Pezzo surface pair. We now explain how to apply Abban-Zhuang method to estimate $\delta(X,\Delta)$. Let $E$ be a prime divisor over $X$. Then there is a birational morphism $\phi\colon Y\to X$ where $E$ is a prime divisor on $Y$. \cite{AZ22} gives the definition of the number $S(W^E_{\bullet,\bullet};q)$ for $q\in E$ which can be calculated by the following theorem. 

\begin{theorem}[{\cite[Equation (1.7.13)]{9p}, \cite[Lemma 3.16]{F21}}]
Let $P(u)$ and $N(u)$ be the positive and negative parts respectively of the Zariski Decomposition of the divisor $\pi^*(-K_X-\Delta)-uE$. For each point $q$ in $E$, we denote
\begin{equation}\label{pre:function h(u)}
    h(u):=(P(u)\cdot E)\cdot \mathrm{ord}_q(N(u)|_E)+\int_0^{\infty}\vol(P(u)|_E-vq)\dd v.
\end{equation}
Then we have 
    \begin{equation}\label{pre:restricted volume part of AZ}
        S(W^E_{\bullet,\bullet};q)=\frac{2}{\vol(-(K_X+\Delta))}\int_0^{\tau(E)}h(u) \dd u.
    \end{equation}
\end{theorem}

We denote the difference of $\Tilde{\Delta}$ on $E$ by $\Phi$ such that $(K_Y+\Tilde{\Delta}+E)|_E=K_E+\Phi$, where $\Tilde{\Delta}$ is the strict transform of $\Delta$. From \cite[Theorems 3.2 and 3.4]{AZ22} we have the following result.
 
\begin{theorem}[\cite{AZ22}]\label{AbbanZ}
We have  \begin{equation*}
    \delta_{\msp}(W, \Delta)\geq \min\left\{\frac{A_{W, \Delta}(E)}{S_{W, \Delta}(E)}, \frac{A_{E, \Phi}(\msq)}{S(W^E_{\bullet,\bullet};\msq)}\right\}.
\end{equation*}   
\end{theorem}

The following lemma allows us to compute the $S$-invariant of toric valuations. 
\begin{lemma}\label{S-inv.}
Let $\PP(m_0,m_1,m_2)$ be a well-formed weighted projective plane with coordinates $x_0, x_1, x_2$, and $v_{(a,b)}$ be the toric monomial valuation of weight $(a,b)$ in the affine coordinate $(x_j, x_k)$ and  $x_i\neq 0$, where $a$ and $b$ are two positive rational numbers and $\{i,j,k\}=\{0,1,2\}$. Then we have 
\[
S_{\mathcal{O}(1)}(v_{(a,b)}) = \frac{1}{3}\left(\frac{a}{m_j}+\frac{b}{m_k}\right).
\]
\end{lemma}

\begin{proof}
Let $s_1, \cdots, s_{N_m}$ be a basis of $H^0(\mathcal{O}(m))$. Denote by $v=v_{(a,b)}$. We have
$S_{m,\mathcal{O}(1)}(v)=\mathrm{sup}\frac{1}{mN_m}\sum_{\alpha=1}^{N_m}v(s_\alpha)$, where the supremum is taken over all the bases of $H^0(\mathcal{O}(m))$. The supremum is achieved by taking a basis consisting of monomials as $v$ is a toric valuation.
Let $s_\alpha=x_0^{p_\alpha}x_1^{q_\alpha}x_2^{r_\alpha}$ such that $m_0p_\alpha+m_1q_\alpha+m_2q_\alpha=m$ for $\alpha=1,\cdots,N_m$ be a monomial basis in $H^0(\bP(m_0, m_1, m_2),\mathcal{O}(m))$. We have $S_{m,\mathcal{O}(1)}(v)=\frac{1}{mN_m}\sum_{\alpha=1}^{N_m}v(x_0^{p_\alpha}x_1^{q_\alpha}x_2^{r_\alpha})$. Without loss of generality, we take $(i,j,k)=(2,0,1)$ and consider in the affine chart $x_2=1$. Let $\Delta_{m}\subset\bR^2$ be the triangle with vertices $(0,0),(\frac{m}{m_0} ,0),(\frac{m}{m_1} ,0)$ and $\Delta^{\mathbb{Z}}_{m}=\{p,q\in \mathbb{Z}_{\geq 0}\colon m_0p+m_1q\leq m \text{ and } m_2| (m_0p+m_1q-m)\}$ be a set of a part of lattice points in $\Delta_{m}$. We get $$S_{m,\mathcal{O}(1)}(v)=(\frac{1}{mN_m}\sum_{p,q\in \Delta^{\mathbb{Z}}_{m}}(p,q) )\cdot (a,b).$$
Since $\frac{1}{mN_m}\sum_{p,q\in \Delta^{\mathbb{Z}}_{m}}(p,q)$ approaches to the centroid of $\frac{1}{m}\Delta_m = \Delta_1$ when $m\to \infty$ (see e.g. \cite[Theorem 2.1]{BJ17}), we get $$S_{\mathcal{O}(1)}(v)=\lim\limits_{m\to\infty}S_{m,\mathcal{O}(1)}(v)=( \frac{1}{3m_0},\frac{1}{3m_1} ) \cdot ( a,b)=\frac{a}{3m_0}+\frac{b}{3m_1},$$ where $( \frac{1}{3m_0},\frac{1}{3m_1})$ is the centroid of $\Delta_1$.
\end{proof}

\subsection{K-moduli of log Fano pairs and wall crossing}\label{backgroudmultiple boundaries}
Now we go over some background about the K-moduli space of log Fano pairs where the boundary divisors are proportional to the anti-canonical divisor. The wall crossing for K-moduli spaces of such pairs with a single boundary divisor was first introduced in \cite{ADL19} for $\bQ$-Gorenstein smoothable case. Later, using the advances from the algebraic K-stability theory, \cite{Zh23} generalized \cite{ADL19} to allow arbitrary klt singularities. Recently, \cite{Zh23c} generalizes the description of wall crossing in \cite[Theorem 1.2.]{Zh23} to log pairs with multiple boundaries. The definition of log pairs with multiple boundaries is stated in \cite{LZ23}. 
In our paper, we consider log pairs with two boundary divisors.
Now we state the definition of log pairs used in the paper.

\begin{definition}[\cite{LZ23,Zh23c}]\label{definelogpairswith2}
Fix $n\in \bN$, $v\in \bQ_{>0}$, $h,d\in \bQ_{>0}$, let $\theta\in [0,\max\{1, h\})\cap \bQ$ be a fixed number, we denote $\mathcal{E}:=\mathcal{E}(n, v, h, d,\theta, w)$ to be the set consisting of log pairs $(X, H+D)$ which satisfies the following conditions:
 \begin{enumerate}
\item $X$ is a $n$-dimensional $\bQ$-Fano variety of volume $(-K_X)^n = v$;
\item $H$ and $D$ are two $\bQ$-Cartier Weil divisors on $X$ such that $-K_X\sim_{\bQ} h^{-1}H \sim_{\bQ} d^{-1} D$;
\item there exists a number $w\in \Delta:=\{w\ |\ \textit{$w\in [0,1)\cap \QQ$ and $h\theta + dw < 1$}\}$ such that $(X, \theta H+wD)$ is K-semistable.
\end{enumerate}
The K-semistable domain of $(X,  H+D)\in \mathcal{E}$ is defined to be:
$$Kss(X, \theta H+wD):=\overline{\{w\in \Delta\ |\ \textit{$(X, \theta H+wD)$ is K-semistable}\}}, $$
where the overline means taking the closure.
\end{definition}

From the works of \cite{Jia20, Xu20, BLX22}, for $w\in \Delta$, there is an Artin stack $\mathcal{M}^K_{n, v, h, d,\theta, w}$ of finite type, parametrizing the log pairs $(X,  H+D)$ such that they satisfies the conditions (1), (2) in Definition \ref{definelogpairswith2}, and $(X,\theta H+wD)$ is a K-semistable log Fano pair.

\medskip

 Furthermore, the Artin stack $\mathcal{M}^K_{n, v, h, d,\theta, w}$ descends to a projective scheme $M^K_{n, v, h, d,\theta, w}$ as a good moduli space which parametrizes K-polystable log Fano pairs $(X, \theta H + wD)$ where $h\theta + dw < 1$ \cite{ABHLX20, BHLLX21, CP21, XZ20, LXZ22}. We call $\mathcal{M}^K_{n, v, h, d,\theta, w}$ (resp. $M^K_{n, v, h, d,\theta, w}$) a K-moduli stack (resp. K-moduli space). 

\medskip

In this paper, the description of wall crossing is based on the principle in the following theorem \cite[Theorem 1.2.]{Zh23}, \cite[Theorem 1.6.]{Zh23c}.

\begin{theorem}\label{wallcrossing[0,1]}
We have a list of rational numbers 
$$0=w_0<w_1<w_2<\cdots<w_k<w_{k+1}=1$$ such that 
\begin{enumerate}
\item For any $w\in (w_i,w_{i+1})$ and $0\leq i\leq k$, the Artin stacks $\mathcal{M}^{K}_{n, v, h, d,\theta, w}$ keep the same.
Also there are open immersions
$$\mathcal{M}^{K}_{n, v, h, d,\theta, w_i-\epsilon}\hookrightarrow \mathcal{M}^{K}_{n, v, h, d,\theta, w_i}\hookleftarrow \mathcal{M}^{K}_{n, v, h, d,\theta, w_i+\epsilon}$$
for very small positive rational number $\epsilon\ll1$ and $1\leq i\leq k$.
\item For any $w\in (w_i,w_{i+1})$ and $0\leq i\leq k$, the good moduli space $M^{K}_{n, v, h, d,\theta, w}$ keep the same. Also there are the following projective morphisms among good moduli spaces 
$$M^{K}_{n, v, h, d,\theta, w_i-\epsilon}\stackrel{\phi^-_i}{\rightarrow} M^{K}_{n, v, h, d,\theta, w_i}\stackrel{\phi^+_i}{\leftarrow}M^{K}_{n, v, h, d,\theta, w_i+\epsilon}$$
for very small positive rational number $\epsilon\ll1$ and $1\leq i\leq k$.
which are induced by the following diagram:
\begin{center}
	\begin{tikzcd}[column sep = 2em, row sep = 2em]
	 \mathcal{M}_{n, v, h, d,\theta, w_i-\epsilon}^K \arrow[d,"",swap] \arrow[rr,""]&& \mathcal{M}_{n, v, h, d,\theta, w_i}^K\arrow[d,"",swap] &&\mathcal{M}_{n, v, h, d,\theta, w_i+\epsilon}^K\arrow[d,""]\arrow[ll,""]\\
	 M_{n, v, h, d,\theta, w_i-\epsilon}^K\arrow[rr,"\phi_i^-"]&& M_{n, v, h, d,\theta, w_i}^K&&M_{n, v, h, d,\theta, w_i+\epsilon}^K\arrow[ll,"\phi_i^+", swap].
	 	 	 	\end{tikzcd}
       \end{center}
\end{enumerate}

\end{theorem}

The following lemma displays interpolation of K-stability which is a helpful tool for checking K-stability. See \cite[Proposition 2.13]{ADL19} and \cite[Lemma 2.6]{Der16} for related results on interpolation.

\begin{lemma}\label{theta,w,interpolation,polystable}
    Let $(X,H+D)\in \cE$. Let $\theta\geq 0$ and $0<w_1< w_2 $ be rational numbers such that $h\theta + d w_2\leq 1$. Assume that both $(X, \theta H + w_1 D)$ and $(X,  \theta H + w_2 D)$ are K-semistable. Then $(X, \theta H+ wD)$ is K-semistable for any $w\in [w_1, w_2]$. Moreover, if one of the following conditions holds, then $(X, \theta H+ wD)$ is K-stable for any $w\in (w_1, w_2)$.
    \begin{enumerate}
        \item $(X, \theta H + w_i D)$ is K-stable for some $i\in \{1,2\}$.
        \item $(X,  \theta H + w_1 D)$ has a K-polystable degeneration $(X_1, \theta H_1 + w_1 D_1)$ such that $\mathrm{rank}\, \Aut(X_1, \theta H_1 + D_1)=1$ and $(X_1,  \theta H_1 + w_2 D_1)$ is not K-semistable.
    \end{enumerate}
\end{lemma}

Here we use the following terminology for log Calabi-Yau pairs according to \cite{Oda13}: it is K-stable if it is klt, and K-semistable if it is log canonical.

\begin{proof}
By interpolation \cite[Proposition 2.13]{ADL19}, we know that $(X, \theta H + wD)$ is K-semistable for any $w\in [w_1, w_2]$. Moreover, (1) also follows from \cite[Proposition 2.13]{ADL19}. We shall focus on (2). 

By Theorem \ref{wallcrossing[0,1]}, there exists $\epsilon >0$ such that the K-polystable degeneration $(X_2, \theta H_2 + w D_2)$ of $(X, \theta H + w D)$ is independent of the choice of $w\in (w_1, w_1+\epsilon)$. Denote the test configuration associated to this degeneration by $(\cX, \theta \cH + w\cD)/\bA^1$. Since $(X_2, \theta H_2 + w D_2)$ is K-polystable for every $w\in (w_1, w_1+\epsilon)$, we know that $\Fut(\cX, \theta \cH + w\cD) = 0$. By linearity of Futaki invariants, we have $\Fut(\cX, \theta \cH + w\cD) = 0$ for every $w\in [w_1, w_2]$. Thus $(X_2, \theta H_2 + w D_2)$ is K-semistable for any $w\in [w_1, w_2]$ by \cite[Lemma 3.1]{LWX21}. 

Since K-semistable is a closed condition for coefficients, we know that $(X_2, \theta H_2 + w_1 D_2)$ is a K-semistable degeneration of $(X, \theta H + w_1 D)$. By \cite[Theorem 1.3]{LWX21} we know that the K-polystable degeneration of $(X_2, \theta H_2 + w_1 D_2)$  is also $(X_1, \theta H_1 + w_1 D_1)$. Thus
\begin{equation}\label{eq:rankaut}
    1 = \rank\Aut(X_1, \theta H_1 + w_1 D_1)\geq \rank\Aut(X_2, \theta H_2 + w_1 D_2).
\end{equation}
If $(X_1, \theta H_1 + D_1)\cong (X_2, \theta H_2 + D_2)$, then $(X_1, \theta H_1 + w_2 D_1)$ is K-semistable, a contradiction to our assumption. Thus  $(X_1, \theta H_1 + D_1)$ is not isomorphic to $(X_2, \theta H_2 + D_2)$, which implies that the inequality in \eqref{eq:rankaut} is strict by \cite[Theorem 1.4]{LWX21}. Thus $\Aut(X_2, \theta H_2 + w D_2)$ is finite for $w\in (w_1, w_2)$ by its rank being zero and reductivity \cite{ABHLX20}. As a result, $(X_2, \theta H_2 + w D_2)$ is K-stable for any $w\in (w_1, w_2)$ and so is $(X, \theta H + wD)$ by \cite{BX19}.
\end{proof}

\section{Variation of GIT}\label{VGIT}
Before depicting the K-moduli spaces of the log Fano pairs in Section \ref{section,nodd} and Section \ref{section,even}, we will describe the corresponding variation of GIT moduli spaces for curves $D$ on $W$ and $W'$. Note that we always consider GIT for reductive groups (indeed $\bG_m^2$) as we use the unipotent subgroup of $\Aut(W)$ or $\Aut(W', H)$ to transform the curve in $W$ or $W'$ to the standard form of $D$ (see \eqref{GIT,even,equationD} and \eqref{GIT,odd,equationD}). Throughout this section, we assume $w\in (0, \frac{n+5}{2n+6})\cap \bQ$.

\subsection{Set-up of $w$-GIT when $n$ is odd}\label{set_up_GIT_odd_affine}
Let $D$ be a curve in $W=\PP(1,2,n+2)_{x,y,z}$ defined by \begin{equation}\label{GIT,even,equationD}
    z^2y + azx^{n+4} + a_0x^{2n+6} + a_1x^{2n+4}y + \cdots + a_{n+3}y^{n+3} = 0.
\end{equation}
We consider an affine space $\mathbb{A}^{n+5}$ with coordinates $a,a_0,\cdots,a_{n+3}$ which parametrizes all curves $D$ of the form \eqref{GIT,even,equationD}.
Consider the action of torus $(\mathbb{G}_m)_{s_1,s_2}^2$ on $W$ given by 
\[
(s_1, s_2) \cdot [x:y:z]:= [x:s_1^{-1} y: s_2^{-1} z].
\]
Then this induces an action of $(\mathbb{G}_m)_{s_1,s_2}^2$ on $\mathbb{A}^{n+5}$ by pushforward of $D$ along the torus action. In particular, we have 
\[
(s_1, s_2) \cdot (a, a_j) = (s_1^{-1}s_2^{-1} a, s_1^{j-1} s_2^{-2} a_j). 
\]

We have a universal family $\phi:(\mathcal{W},w\mathcal{D})\rightarrow \mathbb{A}^{n+5}$ and the fiber over each point $(a,a_0,\cdots,a_{n+3})$ is $(W,wD)$, where $D$ is given by Equation \eqref{GIT,even,equationD}. The CM $\bQ$-line bundle on $(\mathcal{W},w\mathcal{D})$ induces the trivial $\bQ$-line bundle $L_w$ on $\mathbb{A}^{n+5}$ with a $\bG_m^2$-linearization. 

\begin{definition}
    We define the $w$-GIT (semi/poly)stability on $\bA^{n+5}$ by the GIT of the $\bG_m^2$-action on $\bA^{n+5}$ with respect to the linearization $L_{w}$. 
\end{definition}
    
For each $(d,b)\in \bZ^2$ we have a one-parameter subgroup $\lambda_{(d,b)}: \mathbb{G}_m \rightarrow \bG_m^2$ given by $\lambda(t) = (t^d, t^b)$. Moreover, $\lambda$ acts on $\bA^{n+5}$ as 
\begin{equation}\label{act_on_p,odd}
    \lambda(t)\cdot a=t^{-d-b}a \text{ and }
    \lambda(t)\cdot a_j=t^{(j-1)d-2b}a_j \text{ for } j=0,\cdots,n+3.
\end{equation}
Let 
\[
\beta^{w}_{(d,b)}:=\frac{1-n+2nw}{6}d+\frac{2n+1-w(4n+6)}{3(n+2)}b.
\]

\begin{proposition}\label{PropGIT stable_odd}
A point $p=(a,a_0,\cdots,a_{n+3})$ is $w$-GIT semistable if and only if for any one-parameter subgroup $\lambda=\lambda_{(d,b)}$, 
the existence of $\lim\limits_{t \to 0}\lambda(t)\cdot p$  implies $\beta^w_{(d,b)}\geq 0$.
Moreover, a $w$-GIT semistable point $p$ is $w$-GIT polystable if and only if for any one-parameter subgroup $\lambda=\lambda_{(d,b)}$ satisfying that $\lim\limits_{t \to 0}\lambda(t)\cdot p$ exists and $\beta^w_{(d,b)}=0$, the limit point $\lim\limits_{t \to 0}\lambda(t)\cdot p$ and $p$ lie in the same orbit of $\mathbb{G}_m^2$. 
\end{proposition}

\begin{proof}
By Hilbert-Mumford criterion in the affine case (see e.g. \cite[Proposition 2.5]{King1994}), it suffices to show that $\beta_{(d,b)}^w$ coincides with the GIT weight $\mu^{L_w}(q, \lambda)$ of $\lambda=\lambda_{(d,b)}$ acting on the fiber of $L_w$ at the limit point $q:=\lim_{t\to 0}\lambda(t)\cdot p$ provided the existence of $q$. Since the $\bG_m^2$-action on $\bA^{n+5}$ is linear, we know that the fixed locus of every $\lambda$ is an affine subspace of $\bA^{n+5}$ and hence is connected. Therefore, we have $\mu^{L_w}(q, \lambda) = \mu^{L_w}(0, \lambda)$. Since the weight of the CM $\bQ$-line bundle coincides with the Futaki invariant, which also coincides with the $\beta$-invariant for product test configurations, we have that 
\[
\mu^{L_w}(q, \lambda_{(d,b)}) = \mu^{L_w}(0, \lambda_{(d,b)}) = \beta_{(W,wD_0)}(v_{(d,b)}),
\]
where $v_{(d,b)}$ is the monomial valuation on $W$ where in the affine chart $x=1$ we have $v_{(d,b)}(y) = d$ and $v_{(d,b)}(z) = b$. 

Let us assume $d,b\geq 0$ for a moment. Then we can compute $\beta_{(W,wD_0)}(v_{(d,b)})$ as follows.  By Lemma \ref{S-inv.}, we get the $S$-invariant for $v=v_{(d,b)}$:
\begin{equation*}
        S_{\mathcal{O}(1)}(v)=\frac{1}{3}\big(\frac{d}{2}+\frac{b}{n+2}\big).
\end{equation*}
We have $-K_{W}-wD_0=\mathcal{O}(n+5-w(2n+6))=r\mathcal{O}(1)$ and $S_{(W,wD_0)}(v)=\mathrm{deg}(-K_{W}-wD_0)\cdot S_{\mathcal{O}(1)}(v)=rS_{\mathcal{O}(1)}(v)$, where $r=n+5-w(2n+6)$. The log discrepancy $A_{W}(v)=d+b$ and $A_{(W,wD_0)}(v)=A_{W}(v)-wv(D_0)=d+b-w(d+2b)$. Then we have $\beta$-invariant 
\begin{equation}\label{betav,GIT,odd}
    \beta_{(W,wD_0)}(v_{(d,b)})=A_{(W,wD_0)}(v)-S_{(W,wD_0)}(v)=\frac{1-n+2nw}{6}d+\frac{2n+1-w(4n+6)}{3(n+2)}b.
\end{equation}
Thus this implies that $\mu^{L_w}(0, \lambda_{(d,b)})  = \beta^w_{(d,b)}$ whenever $d,b\geq 0$. 
By the standard representation theory of torus we know that $\mu^{L_w}(0, \lambda_{(d,b)}) $ is linear in $(d,b)$. Since $\beta^w_{(d,b)}$ is linear in $(d,b)$ by definition, we conclude that 
\[
\mu^{L_w}(q, \lambda_{(d,b)}) = \mu^{L_w}(0, \lambda_{(d,b)}) = \beta^w_{(d,b)}
\]
for every $(d,b)\in \bZ^2$. Thus the proof is finished. 
\end{proof}

\subsection{Set-up of $w$-GIT when $n$ is even}\label{set_up_GIT_even_affine} Let $n=2l$, where $l$ is a positive integer.
Let $D$ be a curve in $W'=\PP(1,1,l+1)_{u,y,z}$ defined by \begin{equation}\label{GIT,odd,equationD}
    z^2y + azu^{l+2} + a_0u^{2l+3} + a_1u^{2l+2}y + \cdots + a_{2l+3}y^{2l+3} = 0.
\end{equation}
We consider an affine space $\mathbb{A}^{2l+5}$ with coordinates $a,a_0,\cdots,a_{2l+3}$ which parametrizes all curves $D$ of the form \eqref{GIT,odd,equationD}. We consider the action of torus $(\mathbb{G}_m)_{s_1,s_2}^2$ on $W'$ given by 
\[
(s_1, s_2) \cdot [u:y:z]:= [u:s_1^{-1} y: s_2^{-1} z].
\]Then this induces an action of $(\mathbb{G}_m)_{s_1,s_2}^2$ on $\mathbb{A}^{2l+5}$ by pushforward of $D$ along the torus action. In particular, we have 
\[
(s_1, s_2) \cdot (a, a_j) = (s_1^{-1}s_2^{-1} a, s_1^{j-1} s_2^{-2} a_j)
\] for $j=0,1,\cdots,2l+3.$


We have a universal family $\phi:(\mathcal{W}',\frac{1}{2}\mathcal{H}+w\mathcal{D})\rightarrow \mathbb{A}^{2l+5}$ and the fiber over each point $(a,a_0,\cdots,a_{2l+3})$ is $(W',\frac{1}{2}H+wD)$, where $D$ is given by Equation \eqref{GIT,odd,equationD}. The CM $\bQ$-line bundle on $(\mathcal{W}',\frac{1}{2}\mathcal{H}+w\mathcal{D})$ induces the trivial $\bQ$-line bundle $L_w$ on $\mathbb{A}^{2l+5}$ with a $\mathbb{G}_m^2$-linearization.

\begin{definition}
   We define the $w$-GIT (semi/poly)stability on $\mathbb{A}^{2l+5}$ by the GIT of the $\mathbb{G}_m^2$-action on $\mathbb{A}^{2l+5}$ with respect to the linearization $L_w$.
\end{definition}

For each $(d,b)\in \bZ^2$ we have a one-parameter subgroup $\lambda_{(d,b)}: \mathbb{G}_m \rightarrow \bG_m^2$ given by $\lambda(t) = (t^d, t^b)$. Moreover, $\lambda$ acts on $\bA^{2l+5}$ as 
\begin{equation}\label{act_on_p,even}
    \lambda(t)\cdot a=t^{-d-b}a \text{ and }
    \lambda(t)\cdot a_j=t^{(j-1)d-2b}a_j \text{ for } j=0,\cdots,2l+3.
\end{equation}
Let \[\beta_{(d,b)}^w:=\frac{1-2l+4lw}{6}d+\frac{4l+1-w(8l+6)}{6(l+1)}b.\]

\begin{proposition}\label{PropGIT stable_even}
    A point $p=(a,a_0,\cdots,a_{2l+3})$ is $w$-GIT semistable if and only if for any one-parameter subgroup $\lambda=\lambda_{(d,b)}$, 
the existence of $\lim\limits_{t \to 0}\lambda(t)\cdot p$  implies $\beta^w_{(d,b)}\geq 0$. Moreover, a $w$-GIT semistable point $p$ is $w$-GIT polystable if and only if for any one-parameter subgroup $\lambda=\lambda_{(d,b)}$ satisfying that $\lim\limits_{t \to 0}\lambda(t)\cdot p$ exists and $\beta^w_{(d,b)}=0$, the limit point $\lim\limits_{t \to 0}\lambda(t)\cdot p$ and $p$ lie in the same orbit of $\mathbb{G}_m^2$. 
\end{proposition}

\begin{proof}
    By Hilbert-Mumford criterion in the affine case (see e.g. \cite[Proposition 2.5]{King1994}), it suffices to show that $\beta_{(d,b)}^w$ coincides with the GIT weight $\mu^{L_w}(q, \lambda)$ of $\lambda=\lambda_{(d,b)}$ acting on the fiber of $L_w$ at the limit point $q:=\lim_{t\to 0}\lambda(t)\cdot p$ provided the existence of $q$. Since the $\bG_m^2$-action on $\bA^{2l+5}$ is linear, we know that the fixed locus of every $\lambda$ is an affine subspace of $\bA^{2l+5}$ and hence is connected. Therefore, we have $\mu^{L_w}(q, \lambda) = \mu^{L_w}(0, \lambda)$. Since the weight of the CM $\bQ$-line bundle coincides with the Futaki invariant, which also coincides with the $\beta$-invariant for product test configurations, we have that 
\[
\mu^{L_w}(q, \lambda_{(d,b)}) = \mu^{L_w}(0, \lambda_{(d,b)}) = \beta_{(W',\frac{1}{2}H+wD_0)}(v_{(d,b)}),
\]
where $D_0$ is given by $z^2y=0$, $v_{(d,b)}$ is the monomial valuation on $W'$ where in the affine chart $u=1$ we have $v_{(d,b)}(y) = d$ and $v_{(d,b)}(z) = b$. 

Let us assume $d,b\geq 0$ for a moment. Then we can compute $\beta_{(W',\frac{1}{2}H+wD_0)}(v_{(d,b)})$ as follows.
By Lemma \ref{S-inv.}, we get the $S$-invariant for $v=v_{(d,b)}$:
\begin{equation*}
        S_{\mathcal{O}(1)}(v)=\frac{1}{3}\big(d+\frac{b}{l+1}\big).
\end{equation*}
We have $-K_{W'}-\frac{1}{2}H-wD_0=r\mathcal{O}(1)$ and $S_{(W',\frac{1}{2}H+wD_0)}(v)=rS_{\mathcal{O}(1)}(v)$, where $r=l+3-\frac{1}{2}-w(2l+3)$. The log discrepancy $A_{W'}(v)=d+b$ and $A_{(W',\frac{1}{2}H+wD_0)}(v)=A_{W'}(v)-v(\frac{1}{2}H+wD_0)=d+b-w(d+2b)$ since $v(H)=0$. Then we have $\beta$-invariant $\beta_{(W',\frac{1}{2}H+wD_0)}(v_{(d,b)})$ equals
\begin{equation}\label{betav,GIT,even}
    A_{(W',\frac{1}{2}H+wD_0)}(v)-S_{(W',\frac{1}{2}H+wD_0)}(v)=\frac{1-2l+4lw}{6}d+\frac{4l+1-w(8l+6)}{6(l+1)}b.
\end{equation}
Thus this implies that $\mu^{L_w}(0, \lambda_{(d,b)})  = \beta^w_{(d,b)}$ whenever $d,b\geq 0$. 
By the standard representation theory of torus we know that $\mu^{L_w}(0, \lambda_{(d,b)}) $ is linear in $(d,b)$. Since $\beta^w_{(d,b)}$ is linear in $(d,b)$ by definition, we conclude that 
\[
\mu^{L_w}(q, \lambda_{(d,b)}) = \mu^{L_w}(0, \lambda_{(d,b)}) = \beta^w_{(d,b)}
\]
for every $(d,b)\in \bZ^2$. Thus the proof is finished. 
\end{proof}

\subsection{$w$-GIT moduli spaces}



In Sections \ref{set_up_GIT_odd_affine} and \ref{set_up_GIT_even_affine}, we set up $w$-GIT in the affine case. The goal of this section is to describe the $w$-GIT moduli spaces $M_w^{GIT}$. We first give an explicit criterion for $w$-GIT stability. 

Let $(d,b)$, $u=( \frac{1-n+2nw}{6},\frac{2n+1-w(4n+6)}{3(n+2)})$, $v=(-1,-1)$, $v_j=(j-1,-2)$ for $j=0,\cdots,n+3$ be vectors in $\mathbb{R}^2$.
For a point $p=(a,a_0,\cdots,a_{n+3})$, we denote $T_p$ as a subset of $\{v,v_0,\cdots,v_{n+3}\}$ such that $v\in T_p$ if and only if $a\neq 0$, and $v_j\in T_p$ if and only if $a_j\neq 0$ for $j=0,\cdots,n+3$. That is to say, 
\[
T_p:=\begin{cases}
\{v_j\mid a_j\neq 0 \text{ for } j=0,\cdots,n+3\} & \textrm{ if } a=0;\\
\{v\}\cup\{v_j\mid a_j\neq 0 \text{ for } j=0,\cdots,n+3\} & \textrm{ if }a\neq 0.
\end{cases}
\]
%
%
Let $\Gamma_p$ be the cone generated by all the vectors in $T_p$.
We denote the dual cone of $\Gamma_p$ to be $$\Gamma_p^{\vee}=\{r\in\mathbb{R}^2\mid  r\cdot r'\geq 0 \text{ for any } r'\in \Gamma_p\}.$$ 

Now we can reinterpret the conditions for $w$-GIT stability as follows.

\begin{lemma}\label{GITstable,condition,odd}
    Let $n$ be a positive integer.
    A point $p=(a,a_0,\cdots,a_{n+3})$ is $w$-GIT semistable if and only if $u\in \Gamma_p$.
Moreover, if $\Gamma_p$ has dimension one, then a $w$-GIT semistable point $p$ is $w$-GIT polystable. 
If $\Gamma_p$ has dimension two, then 
 $w$-GIT semistable point $p$ is $w$-GIT polystable if for any one-parameter subgroup $\lambda=\lambda_{(d,b)}$ such that $\beta^w_{(d,b)}=0$, then $u$ is in the interior of $\Gamma_p$.
\end{lemma}

\begin{proof}
Let $J=\{j\in \{0,\cdots,n+3\} \mid a_j\neq 0\}$. Let $\lambda=\lambda_{(d,b)}$ be a one-parameter subgroup.
Note \eqref{act_on_p,odd} and \eqref{act_on_p,even}. When $a=0$, we have $\lim\limits_{t \to 0}\lambda(t)p$ exists if and only if \begin{equation*}
    v_j\cdot (d,b)=(j-1)d-2b\geq 0\text{ for } j\in J.
    \end{equation*}
When $a\neq0$, we have $\lim\limits_{t \to 0}\lambda(t)p$ exists if and only if \begin{equation*}
    v\cdot (d,b) =-d-b\geq 0\text{ and }
     v_j\cdot (d,b)=(j-1)d-2b\geq 0\text{ for } j\in J.
    \end{equation*}
Then the existence of $\lim\limits_{t \to 0}\lambda(t)p$ is equivalent to that the vector $(d,b) \in \Gamma_p^{\vee}$. 
By Propositions \ref{PropGIT stable_odd} and \ref{PropGIT stable_even}, if $(d,b)\in \Gamma_p^{\vee}$, then $u\cdot (d,b)=\beta^w_{(d,b)} \geq 0$, 
which is equivalent to $u\in \Gamma_p$.

\smallskip

\smallskip

Let $p$ be a $w$-GIT semistable point. If $\Gamma_p$ has dimension one, then $T_p$ contains only one vector which is denoted by $r$, i.e., $p$ has only one nonzero coordinate corresponding to $r$. Since $p$ is $w$-GIT semistable, we have $u\in \Gamma_p$ is proportional to $r$. By $(d,b) \cdot u=\beta^w_{(d,b)}=0$, we have $ (d,b) \cdot r=0$. Thus $\lim\limits_{t \to 0}\lambda(t)p$
is in the same orbit as $p$. Hence $p$ is $w$-GIT polystable.

\smallskip

Assume $\Gamma_p$ has dimension two, that is to say, $p$ has at least two nonzero coordinates. By Propositions \ref{PropGIT stable_odd} and  \ref{PropGIT stable_even}, the point $p$ is $w$-GIT polystable if $u$ is in the interior of the cone $\Gamma_p$. Otherwise, we suppose $u$ is on a bound of $\Gamma_p$. Then $u$ is proportional to a vector $r\in T_p$ generating the bound. By $( d,b) \cdot u=\beta^w_{(d,b)}=0$, we have $(d,b) \cdot r=0$ and $ (d,b) \cdot r'>0$ for all $r'\in T_p$ such that $r'\neq r$. Under the action of the one-parameter group $\lambda_{(d,b)}$, as $t$ goes to zero, the coordinate corresponding to $r$ goes to a nonzero number, and other nonzero coordinates go to zero. Thus $\lim\limits_{t \to 0}\lambda(t)p$ is not in the same orbit as $p$, which leads to a contradiction. 
\end{proof}

Now we describe the \emph{$w$-GIT moduli spaces} $M^{GIT}_w:= \bA^{n+5}\sslash_{L_w} \bG_m^2$ parameterizing  $w$-GIT polystable points. Recall that we always assume $w\in (0, \frac{n+5}{2n+6})\cap \bQ$. 

\begin{theorem}\label{describe,GITmoduli,odd}
 Let $n$ be an odd positive integer. 
Then the walls for $M^{GIT}_w$ are given by 
  \[
  w_{i}:=\frac{(n+2)^2 - (2n+1)i}{(n+2)(2n+6)-(4n+6)i}, \quad \textrm{ where }0\leq i\leq \frac{n+1}{2}.
  \]
Moreover, we have the following explicit description of $M^{GIT}_w$.
 \begin{enumerate}
    \item When $w\in (0, w_{0}=\frac{n+2}{2n+6})$, the $w$-GIT moduli space $M^{GIT}_w$ is empty.
    \item The $w$-GIT moduli space $M^{GIT}_{w_0}$ consists a single point parameterizing the $w$-GIT polystable points $(a,a_0,\cdots,a_{n+3})$ such that $a=a_0=\cdots=a_{n+2}=0$ and $a_{n+3}\neq 0$.
     \item The $w$-GIT moduli space $M^{GIT}_{w_i}$ parameterizes all points $(a,a_0,\cdots,a_{n+3})$ from one of the following two cases:
     \begin{itemize}
         \item $\{a,a_0,\cdots,a_{n+2-i}\}$ are not all zero, $a_{n+3-i}$ is arbitrary and $\{a_{n+4-i},\cdots,a_{n+3}\}$ are not all zero;
         \item $a=a_0=\cdots=a_{n+2-i}=0$, $a_{n+3-i}\neq0$, and $a_{n+4-i}=\cdots=a_{n+3}=0$. 
     \end{itemize}
     \item If $w\in (w_{i},w_{i+1})$ where $0\leq i\leq \frac{n+1}{2}$, then $M^{GIT}_w$ parameterizing all points $(a,a_0,\cdots,a_{n+3})$ such that $\{a,a_0,\cdots,a_{n+2-i}\}$ are not all zero, and $\{a_{n+3-i},\cdots,a_{n+3}\}$ are not all zero. Note that $w_{\frac{n+3}{2}}=\frac{n+5}{2n+6}$.
 \end{enumerate}      
\end{theorem}

\begin{proof} 
    If $w\in (0,\frac{2n+1}{4n+6}]$, then $u=(\frac{1-n+2nw}{6},\text{ }\frac{2n+1-w(4n+6)}{3(n+2)})$ is not in the lower half plane of $\mathbb{R}^2$. For any nonzero point $p=(a,a_0,\cdots,a_{n+3})$, cone $\Gamma_p$ is in the lower half plane of $\mathbb{R}^2$. So $u$ is not in $\Gamma_p$, which implies $p$ is not $w$-GIT semistable by Lemma \ref{GITstable,condition,odd}. Now we consider $w>\frac{2n+1}{4n+6}$.
    Multiplying by $-2\frac{3(n+2)}{2n+1-w(4n+6)}>0$, the vector $u$ becomes $$(\frac{(n+2)(n-1)-2n(n+2)w}{2n+1-(4n+6)w}, \text{ } -2)$$ which we still denote by $u$. For $0\leq i\leq \frac{n+1}{2}$, we have
    \[\frac{(n+2)(n-1)-2n(n+2)w}{2n+1-(4n+6)w}=n+2-i \text{ if and only if } w=w_i.\]
When $w\in (\frac{2n+1}{4n+6},w_0)$, we have $\frac{(n+2)(n-1)-2n(n+2)w}{2n+1-(4n+6)w}>n+2$. Thus $u$ is not in $\Gamma_p$ for any nonzero point $p=(a,a_0,\cdots,a_{n+3})$. Hence (1) holds.
    
    If $w=w_0$, then $u=(n+2, -2)=v_{n+3}$.
    By (2) of Lemma \ref{GITstable,condition,odd}, 
    a point $p=(a,a_0,\cdots,a_{n+3})$ is $w$-GIT semistable if $T_p$ contains $v_{n+3}$. 
   If $T_p=\{v_{n+3}\}$, then $\Gamma_p$ has dimension one and contains $u$. So $p$ is $w$-GIT polystable. 
    If $T_p$ contains a vector more than $v_{n+3}$, then $u$ is not in the interior of $T_p$. Thus $p$ is not $w$-GIT polystable. Hence $p$ is $w$-GIT polystable iff $T_p=\{v_{n+3}\}$, i.e., $a=a_0=\cdots=a_{n+2}=0$ and $a_{n+3}\neq 0$. Therefore (2) holds.

    If $w=w_i$ for $1\leq i\leq \frac{n+1}{2}$, then $u=(n+2-i, -2)=v_{n+3-i}$. By (2) of Lemma \ref{GITstable,condition,odd}, a point $p=(a,a_0,\cdots,a_{n+3})$ is $w$-GIT semistable if $T_p$ contains $v_{n+3}$, i.e., $\{a,a_0,\cdots,a_{n+3-i}\}$ are not all zero and $\{a_{n+3-i},\cdots,a_{n+3}\}$ are not all zero. If $p$ is polystable, then $v_{n+3}$ cannot lie on a bound of $\Gamma_p$, i.e., $T_p$ must contains a vector in $\{v,v_0,\cdots,v_{n+2-i}\}$ and a vector in $\{v_{n+4-i},\cdots,v_{n+3}\}$. Thus (3) holds.

If $w\in(w_{i},w_{i+1})$ for $0\leq i\leq \frac{n+1}{2}$, then $u$ is between vector $v_{n+3-i}$ and $v_{n+2-i}$. By Lemma \ref{GITstable,condition,odd}, a point $p=(a,a_0,\cdots,a_{n+3})$ is $w$-GIT semistable if $T_p$ contains a vector in $\{v,v_0,\cdots,v_{n+2-i}\}$ and a vector in $\{v_{n+3-i},\cdots,v_{n+3}\}$. We also have that $u$ is in the interior of $\Gamma_p$. So $p$ is $w$-GIT polystable. Hence (4) holds.
\end{proof}

With similar arguments, we can describe the $w$-GIT moduli spaces $M^{GIT}_w$ when $n$ is an even positive integer as follows.
\begin{theorem}\label{describe,GITmoduli,even}
  Let $n=2l$ be an even positive integer for $l\leq 1$. 
Then the walls for $M^{GIT}_w$ are given by 
  \[
  w_{i}:=\frac{(n+2)^2 - (2n+1)i}{(n+2)(2n+6)-(4n+6)i}, \quad \textrm{ where }0\leq i\leq l+1.
  \]
  Moreover, we have the following explicit description of $M^{GIT}_w$.
 \begin{enumerate}
     \item When $w\in (0, w_{0}=\frac{l+1}{2l+3})$, the $w$-GIT moduli space $M^{GIT}_w$ is empty.
     \item The $w$-GIT moduli space $M^{GIT}_{w_0}$ consists a single point parameterizing the $w$-GIT polystable points $(a,a_0,\cdots,a_{2l+3})$ such that $a=a_0=\cdots=a_{2l+2}=0$ and $a_{2l+3}\neq 0$.
     \item  The $w$-GIT moduli space $M^{GIT}_{w_i}$ parameterizes all points $(a,a_0,\cdots,a_{2l+3})$ from one of the following two cases:
     \begin{itemize}
         \item $\{a,a_0,\cdots,a_{2l+2-i}\}$ are not all zero, $a_{2l+3-i}$ is arbitrary, and $\{a_{2l+4-i},\cdots,a_{2l+3}\}$ are not all zero;
         \item $a=a_0=\cdots=a_{2l+2-i}=0$, $a_{2l+3-i}\neq0$ and $a_{{2l+4-i}}=\cdots=a_{2l+3}=0$.
     \end{itemize}
     \item If $w\in (w_{i},w_{i+1})$, where $0\leq i\leq l$, then the $w$-GIT moduli space $M^{GIT}_{w}$ parameterizes all points $(a,a_0,\cdots,a_{2l+3})$ such that $\{a,a_0,\cdots,a_{2l+2-i}\}$ are not all zero, and $\{a_{2l+3-i},\cdots,a_{2l+3}\}$ are not all zero. Note that $w_{l+1}=\frac{1}{2}$.
     \item If $w\in (w_{l+1},\frac{2l+5}{4l+6})$, then the $w$-GIT moduli space $M^{GIT}_{w}$ parameterizes all points $(a,a_0,\cdots,a_{2l+3})$ such that $\{a,a_0,\cdots,a_{l+1}\}$ are not all zero, and $\{a_{l+2},\cdots,a_{2l+3}\}$ are not all zero. 
     \end{enumerate}    
\end{theorem}

\begin{proof}
The proofs for (1)-(4) are the same as the proofs for (1)-(4) in Theorem \ref{describe,GITmoduli,odd}.

Now we consider (5). 
When $w\in (w_{l+1},\frac{2l+5}{4l+6})$, we have $w>\frac{2n+1}{4n+6}$. Similar as the proof for Theorem \ref{describe,GITmoduli,odd},
    multiplying by $-2\frac{3(n+2)}{2n+1-w(4n+6)}>0$, the vector $u$ becomes $$(\frac{(n+2)(n-1)-2n(n+2)w}{2n+1-(4n+6)w}, \text{ } -2)$$ which we still denote by $u$. Let $h(w)= \frac{(n+2)(n-1)-2n(n+2)w}{2n+1-(4n+6)w}$. 
    Since $h(w_{l+1})=l+1$, $h(\frac{2l+5}{4l+6})=l+\frac{1}{2}$ and $h(w)$ decreases when $w\in (w_{l+1},\frac{2l+5}{4l+6})$, we have $h(w)\in(l+\frac{1}{2},l+1)$ for $w\in (w_{l+1},\frac{2l+5}{4l+6})$. Thus $u$ is between vectors $v_{l+1}=(l,-2)$ and $v_{l+2}=(l+1,-2)$. By Lemma \ref{GITstable,condition,odd}, a point $p=(a,a_0,\cdots,a_{n+3})$ is $w$-GIT semistable if $T_p$ contains a vector in $\{v,v_0,\cdots,v_{l+1}\}$ and a vector in $\{v_{l+2},\cdots,v_{n+3}\}$. We also have $u$ is in the interior of $\Gamma_p$, so $p$ is $w$-GIT polystable. Hence (5) holds.
\end{proof}


To conclude this section, we introduce the relevant projective GIT following similar approach as \cite[Definition 5.7 and Theorem 5.8]{ADL19}. Let $n$ be a positive integer. Let $\lambda_{(0,-1)}$ be the one-parameter subgroup of $\mathbb{G}_m^2$ acting on coordinates $a,a_1,\cdots,a_{n+3}$ with positive weights $1,2,\cdots,2$. Taking quotient of $\mathbb{A}^{n+5}\setminus \{0\}$ by the action of $\lambda_{(0,-1)}$, we get the weight projective stack $\cP(1,2^{n+4}):= [(\mathbb{A}^{n+5}\setminus \{0\})/\mathbb{G}_m]$ whose coarse moduli space $\bP(1,2^{n+4})$ is isomorphic to $\bP^{n+4}$ via $[a:a_0:\cdots:a_{n+3}]\mapsto [a^2:a_0:\cdots:a_{n+3}]$. Recall that $\phi: (\cW, w\cD)\to \bA^{n+5}$ (resp. $\phi: (\cW', \frac{1}{2}\cH + w\cD)\to \bA^{n+5}$) is the universal family of pairs when $n$ is odd (resp. even). After quotienting out the diagonal action of $\lambda_{(0,-1)}$ we can descend $\phi$ to a universal family over $\cP(1, 2^{n+4})$, where the CM $\bQ$-line bundle $L_w$ over $\bA^{n+5}$ also descends to a CM $\bQ$-line bundle $\Lambda_w$ over $\bP(1, 2^{n+4})$. By standard GIT we know that the degree of $\Lambda_w$ is equal to the GIT weight $\mu^{L_w}(0, \lambda_{(0,-1)})$ which coincides with $\beta_{(0,-1)}^w$ by the proof of Propositions \ref{PropGIT stable_odd} and  \ref{PropGIT stable_even}. Thus we know that $\Lambda_w$ is ample if and only if $\beta_{(0,-1)}^w > 0$ which is equivalent to $w> \frac{2n+1}{4n+6}$. Thus we can define the projective GIT for the $\bG_m^2/\langle\lambda_{(0,-1)}\rangle$-action on $\bP(1,2^{n+4})$ with respect to $\Lambda_w$.

\begin{proposition}\label{prop:projGIT}
Let $w\in (\frac{2n+1}{4n+6}, \frac{n+5}{2n+6})\cap \bQ$. Then a point $p = (a,a_0, \cdots, a_{n+5})\in \bA^{n+5}$ is $w$-GIT (semi/poly)stable if and only if $p\neq 0$ and $[p]$ is GIT (semi/poly)stable for the $\bG_m^2/\langle\lambda_{(0,-1)}\rangle$-action on $\bP(1,2^{n+4})$ with respect to the ample linearization $\Lambda_w$. In particular, we have 
\[
M_w^{GIT}\cong \bP(1,2^{n+4})\sslash_{\Lambda_w} (\bG_m^2/\langle\lambda_{(0,-1)}\rangle)
\]
is a normal projective variety.
\end{proposition}

\begin{proof}
This follows directly from the standard GIT theory relating affine GIT and projective GIT; see e.g. \cite{MFK94, King1994}. The normality of the GIT quotient follows from the smoothness of $\bP(1,2^{n+4})\cong \bP^{n+4}$. 
\end{proof}

\begin{remark}
We note that  an approach of GIT for non-reductive groups has been developed to construct moduli spaces of various kinds of objects (including weighted projective hypersurfaces); see e.g. \cite{BDHK16, BDHK18, BJK18, Bun24}. It is an interesting question to compare this approach to our more traditional approach of  reductive GIT.    
\end{remark}

\section{Log canonical thresholds}\label{section5:section number}
We consider pairs consisting of a hypersurface defined by a homogeneous polynomial and its ambient weighted projective plane. 
In this section, we compute the log canonical threshold of the pairs for different cases of coefficients in the polynomial. This will be used to investigate K-semistability and K-polystability in Section \ref{section,nodd} and Section \ref{section,even}.

\medskip

Let $(X, L)$ be a normal polarized variety, $B$ an effective boundary such that $-(K_X + B) \equiv \lambda L$ where $\lambda$ is a constant. The log canonicity of the log pair $(X,B)$ is related to the K-semistability of the log pair $(X, B;L)$ (see \cite[Section 9]{BHJ17} and \cite{Oda12}). Thus we need to calculate the log canonical threshold of the log pair $(X, B)$.

\subsection{Odd cases}
Let $W= \PP(1,2,n+2)_{x,y,z}$ be the weighted projective plane where $n$ is a positive odd integer. And let $D$ be an effective divisor on $W$ defined by quasi-homogeneous polynomial
\begin{equation}\label{lct-D}
    z^2y + azx^{n+4} + a_0x^{2n+6} + a_1x^{2n+4}y + \cdots + a_{n+3}y^{n+3} = 0
\end{equation}
of degree $2n+6$, where $a$ is a constant and $a_0,\ldots, a_{n+3}$ are not all zero constants. We set $t = \max\{i\in \ZZ\mid a_i \neq 0\}$. Then the quasi-homogeneous polynomial of the curve $D$ is represented by
\begin{equation*}
    z^2y + azx^{n+4} + a_0x^{2n+6} + a_1x^{2n+4}y + \cdots + a_tx^{2n+6-2t}y^t = 0.
\end{equation*}
To calculate the log canonical threshold of the log pair $(W, D)$ we need to divide it into two cases whether the constant $a$ is zero or not. We first consider the case that $a$ is nonzero. We have $H_y\cap D = \{\msp_z, [1:0:-\frac{a_0}{a}]\}$. On the affine open set $U_z$ of $W$ given by $z=1$, the divisor $D$ is locally defined by the equation
\begin{equation*}
    y + ax^{n+4} + a_0x^{2n+6} + a_1x^{2n+4}y + \cdots + a_tx^{2n+6-2t}y^t = 0.
\end{equation*}
The curve in $\mathbb{C}^2 = \mathrm{Spec}(\mathbb{C}[x,y])$ defined by this equation is smooth at the origin. Since there is an orbifold chart $\mathbb{C}^2\to U_z$ on $W$, it follows that the log pair $(W, D)$ is log canonical at the point $\msp_z$. Similarly, on the affine open set $U_x$ of $W$ defined by $x=1$, the local equation of $D$ is given by
\begin{equation*}
    z^2y + az + a_0 + a_1y + \cdots + a_ty^t = 0.
\end{equation*}
Since the curve in $\mathbb{C}^2=\mathrm{Spec}(\mathbb{C}[y, z])$ defined by this equation is smooth at the point $(y,z) = (0,-\frac{a_0}{a})$, it follows that the log pair $(W, D)$ is log canonical at the corresponding point $[1:0:-\frac{a_0}{a}]$. Therefore, the log pair $(W, D)$ is log canonical along $H_y$.

On the affine open set $U_y$ of $W$ defined by $y=1$ the curve $D$ is defined by
\begin{equation*}
    \begin{split}
        & \left(z+\frac{a}{2}x^{n+4}\right)^2 - \frac{a^2}{4}x^{2n+8} + a_0x^{2n+6} + a_1x^{2n+4} + \cdots + a_{t}x^{2n+6-2t}\\
        =& \left(z+\frac{a}{2}x^{n+4}\right)^2 - \frac{a^2}{4} x^{2n+6-2t}(x^2 - b_1)^{m_1}\cdots (x^2 - b_k)^{m_k} = 0\\
    \end{split}
\end{equation*}
where $b_1,\ldots , b_k$ are nonzero constants. We set 
\begin{equation}\label{section5:higher multiplicity}
    m = \max\{m_1,\ldots , m_k\}.
\end{equation}
By the definition of the log canonical threshold and \cite[Theorem 1.1]{Ku99}, we have
\begin{equation}\label{eq:lct-computation}
    \lct(W, D) = 
    \begin{dcases}
        \min\left\{\frac{m+2}{2m}, 1\right\} &\text{ for $t\geq n+2$}\\
        \min\left\{\frac{n+4-t}{2n+6-2t}, \frac{m+2}{2m}, 1\right\} &\text{ for $t<n+2$}
    \end{dcases}
\end{equation}

\begin{definition}\label{def:Dss}
Let $D_{\ss}$ be the curve in $W$ defined by 
\begin{equation}\label{eq:Dss}
D_{\ss}: z^2 y - 2 z x^{n+4} + \sum_{i=0}^{n+3} (-1)^{i} \binom{n+4}{i+1} x^{2n+6-2i} y^i =0.
\end{equation}
Then in the affine chart $U_y = \{y=1\}$ the curve $D_{\ss}$ can be expressed as 
$$\left(z-x^{n+4}\right)^2 - (x^2 -1)^{n+4} =0.$$ 
\end{definition}

We choose the notation $D_{\ss}$ based on the observation that the pair $(W, D_{\ss})$ will be replaced (at the last wall) by pairs coming from  Weierstra{\ss} double marked hyperelliptic curves (see Section \ref{sec:hyperelliptic}).

From the above discussions we obtain the following lemma.

\begin{lemma}\label{section 5:a is nonzero Lemma}
    Assume that $a\neq 0$. The log pair $(W, \frac{n+5}{2n+6}D)$ is log canonical if and only if $\frac{n+3}{2}\leq t\leq n+3$ and $m\leq n+3$. Moreover, if $m \geq n+4$ then $t=n+3$, $m=n+4$, $D$ has an $A_{n+3}$-singularity at a non-toric point and is projectively equivalent to $D_{\ss}$, and $\lct(W,D) = \frac{n+6}{2n+8}$.
\end{lemma}

\begin{proof}
From the above discussion we see that $m\leq \sum_{j=1}^k m_j = t+1 \leq n+4$. The first statement follows from the lct computations \eqref{eq:lct-computation}. For the second statement, if $m\geq n+4$ then combining with the previous inequality we get $m = n+4$, $k = 1$, and $t = n+3$. In particular, in the affine open chart $U_y$ one has 
\[
D: \left(z+\frac{a}{2}x^{n+4}\right)^2 - \frac{a^2}{4} (x^2 - b_1)^{n+4} = 0,
\]
where $a$ and $b_1$ are both non-zero. Thus by a rescaling of the coordinates we may assume $ a= -2$ and $b_1=1$ which shows that $D$ and $D_{\ss}$ are projectively equivalent. Clearly $D_{\ss}$ has an $A_{n+3}$-singularity at the non-toric point $[1, 1 ,  1]$. The last statement on lct again follows from \eqref{eq:lct-computation}.
\end{proof}

We next consider the case that $a$ is zero. Then the curve $D$ is defined by
\begin{equation*}
    z^2y + a_0x^{2n+6} + a_1x^{2n+4}y + \cdots + a_{t}x^{2n+6-2t}y^{t} = 0.
\end{equation*}
Under this condition we first consider the case that $a_0$ is not zero. Then $D$ is irreducible. We have $H_y\cap D = \{\msp_z\}$. Since the equation of the curve $D$ has the monomial $z^2y$, the log pair $(W, D)$ is log canonical at the point $\msp_z$. On the affine open set $U_y$ the the curve $D$ is defined by
\begin{equation*}
    z^2 + a_0x^{2n+6-2t}(x^2 - b_1)^{m_1}\cdots (x^2 - b_k)^{m_k} = 0
\end{equation*}
where $b_1,\ldots, b_k$ are distinct nonzero constants. We set $m' = \max\{m_1,\ldots, m_k\}$. Then we have
\begin{equation*}
    \lct(W, D) = 
    \begin{dcases}
        \min\left\{\frac{m' + 2}{2m'}, 1\right\} &\text{ for $t\geq n+2$}\\
        \min\left\{\frac{n+4-t}{2n+6-2t}, \frac{m' + 2}{2m'}, 1\right\} &\text{ for $t<n+2$}
    \end{dcases}
\end{equation*}

\begin{lemma}\label{section5:a zero a_0 nonzero}
    Assume that $a=0$ and $a_0\neq 0$ ($D$ is irreducible). Then the log pair $(W, \frac{n+5}{2n+6}D)$ is log canonical if and only if $\frac{n+3}{2}\leq t\leq n+3$.
\end{lemma}

Finally we consider the case $a = a_0 = 0$. Then there is a nonzero nonnegative integer $m_0$ such that $a_{m_0+1}\neq 0$ and $a_0=\ldots  =a_{m_0}=0$. The curve $D$ is defined by
\begin{equation*}
    \begin{split}
        &y(z^2 + x^{2n+6-2t}y^{m_0}(a_{m_0+1}x^{2t-2m_0-2}+a_{m_0+2}x^{2t-2m_0-4}y+\cdots + a_ty^{t-m_0})) \\
        =&y(z^2 + a_t x^{2n+6-2t}y^{m_0}(y - b_1x^2)^{m_1}\cdots (y - b_kx^2)^{m_k} = 0.
    \end{split}
\end{equation*}
Since the equation of $D$ is divided by $y$ the curve $D$ is reducible. We set $m'' = \max\{m_1,\ldots , m_k\}$. Then we have
\begin{equation*}
    \lct(W, D) = 
    \begin{dcases}
        \min\left\{\frac{m_0 + 2}{2m_0+2}, \frac{m'' + 2}{2m''}, 1\right\} &\text{ for $t\geq n+2$}\\
        \min\left\{\frac{m_0 + 2}{2m_0+2}, \frac{n+4-t}{2n+6-2t}, \frac{m'' + 2}{2m''}, 1\right\} &\text{ for $t<n+2$}
    \end{dcases}
\end{equation*}
Then we have the following lemma:
\begin{lemma}\label{section5:a a_0 zeros}
    Assume that $a=a_0=0$  ($D$ is reducible). The log pair $(W, \frac{n+5}{2n+6}D)$ is log canonical if and only if $\frac{n+3}{2}\leq t\leq n+3$ and $m_0\leq\frac{n+1}{2}$.
\end{lemma}

\subsection{Even cases}
Let $W' = \PP(1,1,l+1)_{u,y,z}$ be the weighted projective plane where $l$ is a positive integer. And Let $H$ and $D$ be the effective divisors on $W'$ defined by $u=0$ and $z^2y + azu^{l+2} + a_0u^{2l+3} + a_1u^{2l+2}y + \cdots + a_{2l+3}y^{2l+3} = 0$, respectively, where $a$ is a constant and $a_0,\ldots , a_{2l+3}$ are not all zero constants. We set $t = \max\{i\in \ZZ\mid a_i \neq 0\}$. Then the quasi-homogeneous polynomial of the curve $D$ is represented by
\begin{equation}\label{section 5:even case equation}
    z^2y + azu^{l+2} + a_0u^{2l+3} + a_1u^{2l+2}y + \cdots + a_tu^{2l+3-t}y^t = 0.
\end{equation}

In this subsection we calculate the log canonical threshold of the divisor $D$ with respect to the log pair $(W', \frac{1}{2}H)$. We first consider the locus $H\cap D$ which is defined by $u = z^2y + a_{2l+3}y^{2l+3} = 0$ on $W'$. The point $\msp_z$ is contained in $H\cap D$. Since the equation of the effective divisor $H+D$ has the monomial $uz^2y$, the log pair $(W', H+D)$ is log canonical at $\msp_z$. It implies that the log pair $(W', \frac{1}{2}H+D)$ is log canonical at $\msp_z$. We next consider two cases whether $a_{2l+3}$ is zero or not. If $a_{2l+3}$ is not zero then we can assume that $a_{2l+3} = 1$. It implies that $H\cap D = \{\msp_z, \msp_1, \msp_2\}$ where $\msp_1 = [0:1:i]$ and $\msp_2 = [0:1:-i]$.  Since the intersection number $H\cdot D$ is $\frac{2l+3}{l+1}$, we have the local intersection numbers $(H\cdot D)_{\msp_1} = (H\cdot D)_{\msp_2} = 1$ and $(H\cdot D)_{\msp_z} = \frac{1}{l+1}$. These imply that $H$ and $D$ intersect transversely. Thus the log pair $(W', \frac{1}{2}H+D)$ is log canonical along $H\cap D$. If $a_{2l+3}$ is zero then $H\cap D = \{\msp_y, \msp_z\}$. On the affine open set $U_y$ defined by $y=1$ the equation (\ref{section 5:even case equation}) is represented by
\begin{equation*}
    \left(z + \frac{a}{2}u^{l+2}\right)^2 + u^{2l+3-t}\left(-\frac{a^2}{4}u^{1+t} + a_0u^t + a_1u^{t-1} + \cdots + a_t\right) = 0.
\end{equation*}
We use the local coordinates $(u',z')=(u, z+ \frac{a}{2}u^{l+2})$ at $\msp_y$. Let $\phi\colon Y\to W'$ be the weighted blow-up at $\msp_y$ with weights $\wt(u') = 2$ and $\wt(z') = 2l+3-t$. Then we have the equation
\begin{equation*}
    K_Y + \frac{1}{2}\widetilde{H} + c\widetilde{D} \equiv \phi^*\left(K_W + \frac{1}{2}H + cD\right) + (2l+4-t - 1 - c(4l+6-2t))E,
\end{equation*}
where $\widetilde{H}$ and $\widetilde{D}$ are the strict transform of $H$ and $D$, respectively, $E$ is the $\phi$-exceptional and $c$ is a positive rational number. From this equation we have the following lemma:
\begin{lemma}\label{section 5 even:lct along H}
    The log canonical threshold of the curve $D$ with respect to the log pair $(W', \frac{1}{2}H)$ along $H$ is either $1$ for $t \geq 2l+2$ or $\frac{2l+4-t}{4l+6-2t}$ for $t < 2l+2$. 
\end{lemma}

Finding the log canonical threshold $\lct_{\msp}(W', \frac{1}{2}H ; D)$ where $\msp\in W'\setminus H$ is similar to the odd cases. Since $\lct_{\msp}(W', D) = \lct_{\msp}(W', \frac{1}{2}H ; D)$ we calculate the log canonical threshold $\lct_{\msp}(W', D)$ instead of $\lct_{\msp}(W', \frac{1}{2}H ; D)$. We first consider the case that $a$ is not zero. We have $H_y\cap D = \{\msp_z, [1:0:-\frac{a_0}{a}]\}$. Since the curve in $\mathbb{C}^2$ defined by the local equation of $D$ on the affine open set $U_z$ of  $W'$ defined by $z=1$ is smooth at the origin and there is an orbifold chart $\mathbb{C}^2\to U_z$ on $W'$, the log pair $(W', D)$ is log canonical at the point $\msp_z$. Similarly, the curve in $\mathbb{C}^2$ defined by the local equation of $D$ on the affine open set $U_u$ of $W'$ defined by $u=1$ is smooth at the point $(y,z) = (0, -a_0/a)$, the log pair $(W', D)$ is log canonical at the corresponding point $[1:0:-\frac{a_0}{a}]$. Therefore, $\lct_{\msp}(W', \frac{1}{2}H ; D) = 1$ where $\msp\in H_y\setminus H$. On the affine open set $U_y$ defined by $y=1$ the equation (\ref{section 5:even case equation}) of $D$ is represented by
\begin{equation*}
    \left(z + \frac{a}{2}u^{l+2}\right)^2 - \frac{a^2}{4}u^{2l+3-t}(u - b_1)^{m_1}\cdots (u - b_k)^{m_k} = 0
\end{equation*}
where $b_1\ldots , b_k$ are nonzero constants. We set
\begin{equation}\label{section 5 even:maximal m}
    m = \max\{m_1,\ldots , m_k\}\leq 2l+4.
\end{equation}
We can obtain the following:
\begin{equation*}
    \inf\left\{\lct_{\msp}(W', D) ~\middle|~ \msp\in W'\setminus H\right\} = \min\left\{\frac{m+2}{2m}, 1\right\}
\end{equation*}

\begin{definition}\label{def:Dss-even}
    Let $D_{\ss}$ be the curve in $W'$ defined by 
\begin{equation}\label{eq:Dss-even}
D_{\ss}: z^2 y - 2 z u^{l+2} + \sum_{i=0}^{2l+3} (-1)^{i} \binom{2l+4}{i+1} u^{2l+3-i} y^i =0.
\end{equation}
Then in the affine chart $U_y = \{y=1\}$ the curve $D_{\ss}$ can be expressed as 
$$\left(z-u^{l+2}\right)^2 - (u -1)^{2l+4} =0.$$
\end{definition}

\begin{lemma}\label{section5 even:a nonzero}
    Assume that $a\neq 0$. The log pair $(W', \frac{1}{2}H + \frac{2l+5}{4l+6}D)$ is log canonical if and only if $\frac{2l+3}{2}\leq t\leq 2l+3$ and $m\leq 2l+3$. Moreover, if $m \geq 2l +4$ then $t= 2l+3$, $m=2l+4$, $\frac{1}{2}H+D$ is projectively equivalent to $\frac{1}{2}H+D_{\ss}$, $D$ has an $A_{2l+3}$-singularity at a non-toric point, and $\lct(W', \frac{1}{2}H;D) = \frac{l+3}{2l+4}$.
\end{lemma}

\begin{proof}
The proof is very similar to Lemma \ref{section 5:a is nonzero Lemma} so we omit it.
\end{proof}

If $a=0$ and $a_0\neq 0$ then the equation (\ref{section 5:even case equation}) is presented by
\begin{equation*}
    z^2 + a_tu^{2l+3-t}(u - b_1)^{m_1}\cdots (u - b_k)^{m_k} = 0
\end{equation*}
on the affine open set $U_y$, where $b_1,\ldots , b_k$ are nonzero constants. We have the following:
\begin{equation*}
    \inf\left\{\lct_{\msp}(W', D) ~\middle|~ \msp\in W'\setminus H\right\}
    = \min\left\{\frac{m'+2}{2m'}, 1\right\}\geq \frac{2l+5}{4l+6}
\end{equation*}
where $m' = \max\{m_1,\ldots , m_k\}\leq 2l+3$.
\begin{lemma}\label{section5 even:a zero a_0 nonzero}
    Assume that $a=0$ and $a_0\neq 0$ ($D$ is irreducible). Then the log pair $(W', \frac{1}{2}H + \frac{2l+5}{4l+6}D)$ is log canonical if and only if $\frac{2l+3}{2}\leq t\leq 2l+3$.
\end{lemma}
Finally we consider the case $a=a_0=0$. Then there is a nonzero nonnegative integer $m_0$ such that $a_{m_0+1}\neq 0$ and $a_0=\ldots  =a_{m_0}=0$. The curve $D$ is defined by
\begin{equation*}
    \begin{split}
        &y(z^2 + u^{2l+3-t}y^{m_0}(a_{m_0+1}u^{t-m_0-1}+a_{m_0+2}x^{t-m_0-2}y+\cdots + a_ty^{t-m_0})) \\
        =&y(z^2 + a_t u^{2l+3-t}y^{m_0}(y + b_1u)^{m_1}\cdots (y + b_ku)^{m_k}) = 0,
    \end{split}
\end{equation*}
where $b_1,\ldots , b_k$ are nonzero constants. We have the following:
\begin{equation*}
    \inf\left\{\lct_{\msp}(W', D) ~\middle|~ \msp\in W'\setminus H\right\}
    = \min\left\{\frac{m_0+2}{2m_0 + 2}, \frac{m''+2}{2m''}, 1\right\},
\end{equation*}
where $m'' = \max\{m_1,\ldots , m_k\}\leq 2l+3$. Then we have the following lemma:
\begin{lemma}\label{section5 even:a and a_0 zero}
    Assume that $a=a_0=0$  ($D$ is reducible). The log pair $(W', \frac{1}{2}H + \frac{2l+5}{4l+6}D)$ is log canonical if and only if $\frac{2l+3}{2}\leq t\leq 2l+3$ and $m_0\leq\frac{2l+1}{2}$.
\end{lemma}

\section{K-moduli space when $n$ is odd}\label{section,nodd}
Let $\mathbb{P}:=\mathbb{P}(1,2,n+2,n+3)$
be a weighted projective space with coordinates $x,y,z,t$ of weights $1,2,n+2,n+3$ respectively and $H_{2(n+2)}$ be a hypersurface of degree $2(n+2)$ in  $\mathbb{P}$.
The group of automorphisms of $\mathbb{P}$ preserving $H_{2(n+2)}$ is generated by scalings of $x$, $y$ or $z$ and unipotent elements mapping $y$ to $y$ plus a degree $2$ polynomial of $x$, mapping $z$ to $z$ plus a degree $n+2$ polynomial of $x$ and $y$, or mapping $t$ to $t$ plus a degree $n+3$ polynomial of $x,y$ and $z$. 
Note that $\dim {H}^0(W,\mathcal{O}_W(2(n+3)))=2(n+3)+5$. Hence the moduli space parametrizing the degree $2(n+3)$ hypersurfaces in $\mathbb{P}$ has dimension $n+9$.

By Theorem 1.0.6 in \cite{KVNW}, a quasi-smooth, well-formed hypersurface of degree $2(n+3)$ in $\mathbb{P}$ is K-polystable.
We are interested in the K-moduli space $\mathcal{F}_n$ parameterizing these K-stable quasi-smooth hypersurfaces $H_{2(n+3)}$ and their K-polystable limits. 
By Lemma \ref{reduce1/2},  
it reduces to the K-moduli of log Fano pairs $(W:=\mathbb{P}(1,2,n+2), \frac{1}{2}D)$ by double cover constructions, where $D$ is given by $f:=z^2y + azx^{n+4} + \nu(x,y)$ and $\nu(x,y)$ is a polynomial of degree $2n+6$.

\begin{lemma}\label{reduce1/2}
Let $H_{2(n+3)}$ be the hypersurface of degree $2(n+3)$ 
given by a homogenous polynomial $$\Tilde{f}={t}^2+t xz+t g(x,y)+z^2x^2+z^2y+zh(x,y)+\nu(x,y),$$ where $g(x,y), h(x,y)$ and $\nu(x,y)$ are homogenous polynomials of degree $n+3$, $n+4$ and $2(n+3)$ respectively.
Then $H_{2(n+3)}$ is K-semistable iff $(W,\frac{1}{2}D)$ is K-semistable, where $W=\mathbb{P}(1,2,n+2)$ and $D$ is the curve given by $f=z^2y + zx^{n+4} + \nu(x,y)$ for some homogenous polynomial $\nu(x,y)$ of degree $2n+6$.
\end{lemma}
\begin{proof}
 Under automorphisms of $\mathbb{P}$ preserving $H_{2(n+3)}$, the polynomial $\Tilde{f}$ can be rewrite as ${t}^2+z^2y + zx^{n+4} + \nu(x,y)$.
We consider the double cover $$\sigma:H_{2(n+3)}\rightarrow W=\mathbb{P}(1,2,n+2)$$ mapping $[x;y;z;t]$ to $[x;y;z]$ with branch locus $D$ defined by $f:=z^2y + zx^{n+4} + \nu(x,y)=0$.
We have $K_{H}=\sigma^*(K_W+\frac{1}{2}D)$.
Then by Theorem 1.2. in \cite{LiuZhu}, the hypersurface $H_{2(n+3)}$ is K-semistable if and only if the pair $(W,\frac{1}{2}D)$ is K-semistable.
\end{proof}

\begin{definition}\label{defineK-moduli,odd}
Let $D$ be a curve in $W=\mathbb{P}(1,2,n+2)$ given by 
\begin{equation}\label{equation of D}
f=z^2y+azx^{n+4}+a_0x^{2n+6}+a_1x^{2n+4}y+a_2x^{2n+2}y^2+\cdots+a_{n+2}x^2y^{n+2}+a_{n+3}y^{n+3},\end{equation}
where $n\geq 1$. Let $r=\frac{2n+6}{n+5}$. We have 
 $D\in|\mathcal{O}(2n+6)|$ and $D=-rK_W$. By the set-up in Section \ref{backgroudmultiple boundaries}, we let $v=\frac{(n+5)^2}{2(n+2)}$ and $\theta=0$, then the parameter $h$ does not affect the moduli space. We have $M^{K}_{2,v,h,r,0,w}$ is the K-moduli space parametrizing all K-polystable log Fano pairs $(W, wD)$ for $0\leq w<\frac{n+5}{2n+6}$. Denote $M_w$ to be the closure of the locus in $M^{K}_{2,v,h,r,0,w}$ consisting of pairs of which the surfaces are isomorphic to $W$. We know $M_w$ is a closed subscheme of $M^{K}_{2,v,h,r,0,w}$.
\end{definition}

\begin{definition}\label{define w_n,i,replace c_n,e}
 For an integer $0\leq i \leq \frac{n+3}{2}$, we denote $$w_{i}:=\frac{(n+2+i)(n+5)-3(n+2)(1+i)}{(n+2+i)(2n+6)-6i(n+2)}$$
 and $\xi_n:= \frac{n^3 + 11n^2 + 31n + 23}{2n^3 + 18n^2 + 50n + 42}$.
 Note that $w_{0}=\frac{n+2}{2n+6}$, $w_{\frac{n+3}{2}}=\frac{n+5}{2n+6}>\xi_n>\frac{1}{2}$ and $w_{i}<\frac{1}{2}$ for $0\leq i \leq \frac{n+3}{2}-1$.
\end{definition}

Now we relate the GIT moduli spaces $M^{GIT}_w$ defined in section \ref{VGIT} with $M_w$. 
Theorem \ref{isomorphism,GITandK} implies that they are isomorphic as long as $w<\xi_n$.

\begin{theorem}\label{isomorphism,GITandK}
For $w\in (0,\xi_n)$, the GIT moduli spaces $M^{GIT}_w$ are isomorphic to the K-moduli spaces $M_w$ by the morphism 
\[
\begin{split}
\Phi_w: M^{GIT}_w &\longrightarrow M_w \\(a,a_0,\cdots,a_{n+3}) & \mapsto (W,wD),
\end{split}
\]
where $D$ is given by the equation
\begin{equation*}
    z^2y + azx^{n+4} + a_0x^{2n+6} + a_1x^{2n+4}y + \cdots + a_{n+3}y^{n+3} = 0.
\end{equation*}
\end{theorem}

In order to prove Theorem \ref{isomorphism,GITandK}, we need several Lemmas and Propositions as follows to describe an open subset of $M_w$.

We first state a relatively easy direction, that is, K-stability implies GIT as long as the surface does not change.

\begin{theorem}\label{thm:K-imply-GIT}
Let $D$ be a curve in $W$ given by the equation \eqref{equation of D}. Let $w\in (0, \frac{n+5}{2n+6})\cap \bQ$.  If $(W,wD)$ is K-semistable (resp. K-polystable), then $(a, a_0, \cdots, a_{n+3})$ is $w$-GIT semistable (resp. $w$-GIT polystable).
\end{theorem}

\begin{proof}
When $w> \frac{2n+1}{4n+6}$, this follows directly from the Paul--Tian criterion \cite{PT06} (see \cite[Theorem 2.22]{ADL19} for a precise version) applied to the universal family over $\cP(1, 2^{n+4})$ and Proposition \ref{prop:projGIT} as the linearization $\Lambda_w$ is the CM $\bQ$-line bundle for the universal family of $(W, wD)$. When $w< \frac{2n+1}{4n+6}$,  we know that $\lambda_{(0,-1)}$ induces a test configuration degenerating $(W,wD)$ to $(W,w D_0)$ whose Futaki invariant is equal to $\beta^w_{(0,-1)}<0$, thus $(W,wD)$ is always K-unstable. When $w = \frac{2n+1}{4n+6}$, the same argument implies that either $(W,wD)$ is K-unstable or it is K-semistable and $\lambda_{(0,-1)}$ has Futaki invariant  $\beta^w_{(0,-1)}=0$. Since $(W,wD)$ is K-semistable, we are in the latter case. Then by \cite[Lemma 3.1]{LWX21} we know that $(W, wD_0)$ is K-semistable which implies that $\beta_{(d,b)}^w = 0$ for any $(d,b)\in \bZ^2$. This is absurd as $1-n + 2nw \neq 0$. Thus the proof is finished.
\end{proof}

The pairs in Proposition \ref{evenpoly+polyanyn} with a $\bG_m$ action have complexity one. 

\begin{proposition}\label{evenpoly+polyanyn}
For any integer $\frac{n+3}{2}\leq e\leq n+3$, where $n\geq 1$, then $(W,wD_e)$ is K-polystable if and only if $w=w_{n+3-e}$, where $D_e$ is given by $z^2y+x^{2n+6-2e}y^e=0$ and $\frac{n+2}{2n+6}\leq w_{n+3-e}\leq \frac{n+5}{2n+6}$. Moreover, when  $w\neq w_{n+3-e}$, $(W,wD_e)$ is K-unstable.
\end{proposition}
\begin{proof}
We can check that $$w_{n+3-e}\leq \frac{n+5}{2n+6} \text{ iff } \frac{n+3}{2}\leq e, \text{ and } \frac{n+2}{2n+6}\leq w_{n+3-e} \text{ iff } e\leq n+3 .$$
In the affine chart $y=1$, we have $D_e$ is given by $z^2+x^{2n+6-2e}$.
We use Theorem \ref{Tone} and Remark \ref{remarkTone} to check K-polysablity. The pair $(W, wD_0)$ is a complexity one $\mathbb{T}$-pair with an action by $\mathbb{T}=\mathbb{G}_m$: $\lambda[x,y,z]=[\lambda x, y, \lambda^{n+3-e} z]$ for $\lambda \in \mathbb{G}_m$

The monomial valuation $v_{(1,n+3-e)}$ such that $v(x)=1$ and $v(z)=n+3-e$ is induced by the $\mathbb{G}_m$-action.  
We first compute the $\beta$-invariant for $v=v_{(1,n+3-e)}$. 
Then by Lemma \ref{S-inv.}, we get the S-invariant for $v$: $$
S_{\mathcal{O}(1)}(v)=\frac{1}{3}(1+\frac{n+3-e}{n+2})=\frac{2n+5-e}{3(n+2)}.$$ We also have $-K_W-wD_e=\mathcal{O}(n+5-(2n+6)w)$ and $S_{(W,wD_e)}(v)=\mathrm{deg}(-K_W-wD_e)\cdot S_{\mathcal{O}(1)}(v)=(n+5-(2n+6)w)\frac{2n+5-e}{3(n+2)}$.
Also, the log discrepancy $A_W(v)=n+4-e$ and $A_{(W,wD_e)}(v)=A_W(v)-wv(D_e)=n+4-e-w(2n+6-2e)$. 
Hence we have $\beta$-invariant
\begin{equation}\label{horizontal-c-D_e}
    \beta_{(W,wD_e)}(v)=A_{(W,wD_e)}(v)-S_{(W,wD_e)}(v)=0 \text{ if and only if }w=w_{n+3-e}.
\end{equation}

By Remark \ref{remarkTone}, there is no horizontal $\mathbb{T}$-invariant prime divisor.
If $e$ is odd, We have vertical
$\mathbb{T}$-invariant prime divisors 
$\{[x;0;z]\}, \{[0;y;z]\}, \{[x;y;0]\}, \{z+ax^{n+3-e}y^{\frac{e-1}{2}}=0,a=\pm i\}, \{z+ax^{n+3-e}y^{\frac{e-1}{2}}=0,a\neq \pm i\}$ under the $\bG_m$ action $\lambda[x,y,z]=[\lambda x, y, \lambda^{n+3-e} z]$, where $i^2=-1$.
If $e$ is even, the vertical
$\mathbb{T}$-invariant prime divisors are
$\{[x;0;z]\}, \{[0;y;z]\}, \{[x;y;0]\}, \{z^2+x^{2n-6-2e}y^{e-1}=0\}, \{z+ax^{2n-6-2e}y^{e-1}=0,a\neq 1\}$. 
Next, we check $\beta(E)> 0$ for every vertical $\mathbb{T}$-invariant prime divisor $E$. 
We now deal with the situation $e$ is odd. The same arguments work for the situation that $e$ is even.

Since $-K_W-wD_0=(n+5-w(2n+6))\mathcal{O}(1)$, by Lemma \ref{Sinv,k}, we have $S_{W,wD_0}(F)=\frac{n+5}{3k}-w\frac{2n+6}{3k}$ for $F\in|\mathcal{O}(k)|$.

If $E$ is $\{[0;y;z]\}$, then $E\in|\mathcal{O}(1)|$. We have $S_{(W,wD_e)}(E)=\frac{n+5}{3}-\frac{2n+6}{3}w$. Also $A_{(W,wD_e)}(E)=1-\mathrm{Coeff}_E(wD_e)=1$. Then $\beta_{(W,wD_e)}(E)>0$ iff $w>\frac{n+2}{2n+6}$, which is satisfied by $w_{n+3-e}$.

If $E$ is $\{[x;y;0]\}$, then $E\in|\mathcal{O}(n+2)|$. We have $S_{(W,wD_e)}(E)=\frac{n+5}{3(n+2)}-\frac{2n+6}{3(n+2)}w$. Also $A_{(W,wD_e)}(E)=1-\mathrm{Coeff}_E(wD_e)=1-0=1$. Then $\beta_{(W,wD_e)}(E)=A_{(W,wD_e)}(E)-S_{(W,wD_e)}(E)>0$ iff $1> \frac{n+5}{3(n+2)}-\frac{2n+6}{3(n+2)}w$ which is satisfied by any $w>0$, so is the constant $w_{n+3-e}$.

If $E$ is $\{[x;0;z]\}$, then $E\in|\mathcal{O}(2)|$. We have $S_{(W,wD_e)}(E)=\frac{n+5}{6}-\frac{2n+6}{6}w$. Since $D_e$ is given by $y(z^2+x^{2n+6-2e}y^{e-1})=0$, so $A_{(W,wD_e)}(E)=1-\mathrm{Coeff}_E(wD_e)=1-w$. Then $\beta_{(W,wD_e)}(E)=A_{(W,wD_e)}(E)-S_{(W,wD_e)}(E)>0$ iff $w> \frac{n-1}{2n}$.
We have $w_{n+3-e}>\frac{n+2}{2n+6}>\frac{n-1}{2n}$.
So the constant $w_{n+3-e}$ satisfies this inequality.

If $E$ is $\{z+ax^{n+3-e}y^{\frac{e-1}{2}}=0\}$, then $E\in|\mathcal{O}(n+2)|$. We have $S_{(W,wD_e)}(E)=\frac{n+5}{3(n+2)}-\frac{2n+6}{3(n+2)}w$. If $a=\pm i$, the log discrepancy $A_{(W,wD_e)}(E)$ is $1-w$. Then $\beta_{(W,wD_e)}(E)=A_{(W,wD_e)}(E)-S_{(W,wD_e)}(E)>0$ iff $\frac{2n+1}{n}>w$. Since $w_{n+3-e}\leq \frac{n+5}{2n+6}<\frac{2n+1}{n}$, the constant $w_{n+3-e}$ satisfies this inequality. If $a\neq\pm i$, the log discrepancy $A_{(W,wD_e)}(E)=1$. Then $\beta_{(W,wD_e)}(E)=A_{(W,wD_e)}(E)-S_{(W,wD_e)}(E)>0$ iff $1>\frac{n+5}{3(n+2)}-\frac{2n+6}{3(n+2)}w$ which is satisfied by any $w>0$, so is the constant $w_{n+3-e}$.

Therefore, the pair $(W, w_{n+3-e}D_e)$ is K-polystable by Theorem \ref{Tone}. 

Now we show $(W, wD_e)$ is K-semistable only when $w=w_{n+3-e}$. By Theorem \ref{Tone}, if $(W, wD_e)$ is K-semistable, we have $\beta_{(W,wD_e)}(v)=0$. By \eqref{horizontal-c-D_e}, $\beta_{(W,wD_e)}(v)=0$ is equivalent to $w=w_{n+3-e}$. Hence the conclusion holds.
\end{proof}

\begin{remark}\label{degenrates to K-poly}\
We will use the following degenerations to show K-semistability.
\begin{enumerate}
\item Let $\frac{n+3}{2}+1\leq e\leq n+3$. Assume $D$ is defined by Equation \ref{equation of D} such that $a_e\neq 0$, $a_{e+1}=\cdots=a_{n+3}=0$ and $\{a,a_0,\cdots,a_{e-1}\}$ are not all zero. Then the pair $(W,D)$ degenerate to $(W,\{z^2y+x^{2n+6-2e}y^{e}=0\})$ by $\bG_m$ action defined by $\lambda[x:y:z]\mapsto[ x:\lambda^{-1}y:\lambda^{\frac{1-e}{2}}z]$, as $\lambda \rightarrow 0$.
\item Let $\frac{n+3}{2}\leq e\leq n+2$. Assume $D$ is defined by Equation \ref{equation of D} such that $a=a_0=\cdots=a_{e-1}=0$, $a_e\neq 0$ and $\{a_{e+1},\cdots,a_{n+3}\}$ are not all zero. Then the pair $(W,D)$ degenerate to $(W,\{z^2y+x^{2n+6-2e}y^{e}=0\})$ by $\bG_m$ action
 defined by $\lambda[x:y:z]\mapsto[\lambda^{-1}x: y:\lambda^{-(n+3-e)}z]$, as $\lambda \rightarrow 0$.
\end{enumerate}  
\end{remark}

Let $n\geq 1$ be an odd integer.
We consider the weighted projective plane $W \coloneqq \PP(1,2,n+2)_{x,y,z}$. Let $D$ be an effective divisor on $W$ defined by a quasi-homogeneous polynomial
\begin{equation}\label{equation-D-odd}
    z^2y + azx^{n+4} + a_0x^{2n+6} + a_1x^{2n+4}y + \cdots + a_{n+3}y^{n+3} = 0
\end{equation}
of degree $2n+6$, where $a, a_0,\ldots , a_{n+3}$ are constants. When $a_0,\ldots , a_{n+3}$ are not all zero, let $t=\max\{i|a_i\neq 0\}$. 
When $a=a_0=0$ and $a_1,\ldots , a_{n+3}$ are not all zero, we denote $m_0$ such that $a_1,\ldots , a_{m_0}$ are zero and $a_{m_0+1}$ is not zero. That is to say
$m_0+1=\min\{i|a_i\neq 0\}$.
Then the polynomial for $D$ is \begin{equation*}
    z^2y + a_{m_0+1}x^{2n+6-2(m_0+1)}y^{m_0+1} + \cdots + a_{t}x^{2n+6-2t}y^{t} = 0.
\end{equation*}

\medskip

\begin{proposition}\label{t<=(n+3)/2}
Assume $a$ is any integer.
\begin{enumerate}
    \item If $a_0=\cdots=a_{n+3}=0$, then $(W,wD)$ is not K-semistable for $\frac{n+2}{2n+6}\leq w\leq \frac{n+5}{2n+6}$.
    \item If $0\leq t\leq\frac{n+3}{2}$, then $(W,wD)$ is not K-semistable for $\frac{n+2}{2n+6}\leq w< \frac{n+5}{2n+6}$.
\end{enumerate}
\end{proposition}

\begin{proof}
    The result is implied by Lemmas \ref{alla,a_izero,notK-s.s},  \ref{all,a_izero,notK-s.s} and  \ref{0<=t<=(n+3)/2notk-s.s.}.
\end{proof}

\begin{remark}
    The lemmas in Section \ref{NonCYcase} do not rely on the other sections of this paper.
\end{remark}

\begin{lemma}\label{notK-s.s.}
The log Fano pair $(W,wD)$ is not K-semistable if $w<\frac{n+2}{2n+6}$.  
\end{lemma}
\begin{proof}
Let $H_x$ be the prime divisior given by $x=0$. 
Since $-K_W-wD=(n+5-w(2n+6))\mathcal{O}(1)$, by Lemma \ref{Sinv,k}, we have $S_{(W,wD)}(H_x)=\frac{n+5}{3}-w\frac{2n+6}{3}$, where $D$ is given by $f=0$. Since $A_{(W,wD)}(H_x)=1$, we have $\beta_{(W,wD)}(H_x)=A_{(W,wD)}(H_x)-S_{(W,wD)}(H_x)\geq 0$ iff $w\geq \frac{n+2}{2n+6}$.
\end{proof}

We recall that $\xi_n:= \frac{n^3 + 11n^2 + 31n + 23}{2n^3 + 18n^2 + 50n + 42}$. Assume $a_{n+3}$ is not zero, i.e., $t=n+3$. By Lemma \ref{section 5:a is nonzero Lemma}, if $a$ is not zero, then the log canonical threshold of the log pair $(W, D)$ is greater than or equal to $\frac{n+5}{2n+6}$ except $m = n+4$ case, where $m$ is the number defined in (\ref{section5:higher multiplicity}). The curve $D$ has a type $A_{n+3}$-singularity iff $m = n+4$.

\begin{proposition}\label{ksemistable,poly-allcases}
The K-semistability (K-polystability) of $(W,wD)$ depends on whether the coefficients in the polynomial \eqref{equation-D-odd} are zero or not. We list all the cases as follows (Table \ref{table: oddcoefficicentlist}):
 \begin{enumerate}
 \item Assume $a_{n+3}\neq 0$ and $a\neq 0$. If the curve $D$ has a type $A_{n+3}$-singularity, then the log pair $(W, wD)$ is K-semistable if and only if $w\in[\frac{n+2}{2n+6}, \xi_n]$; is K-polystable if and only if $w\in(\frac{n+2}{2n+6}, \xi_n)$.
 \item  Assume $a_t\neq 0$, $a_{t+1}=\cdots=a_{n+3}=0$ and $\{a,a_0,\cdots,a_{m_0+1}\}$ are not all zero, where $\frac{n+3}{2}+1\leq t\leq n+3$ and $0\leq m_0\leq \frac{n+1}{2}$. When $t=n+3$ and $a\neq 0$, we assume that the curve $D$ does not has a type $A_{n+3}$-singularity. Then the log pair $(W, wD)$ is K-semistable if and only if $w\in[w_{n+3-t},\frac{n+5}{2n+6}]$; is K-polystable if and only if $w\in(w_{n+3-t},\frac{n+5}{2n+6})$. 
 \item Assume $a=a_0=\cdots=a_{m_0}=0$, $a_{m_0+1}\neq 0$, $a_t\neq 0$ and $a_{t+1}=\cdots=a_{n+3}=0$, where $\frac{n+3}{2}+1\leq t\leq n+3$ and $\frac{n+3}{2}\leq m_0\leq t-1$. Then the log pair $(W, wD)$ is K-semistable if and only if $w\in[w_{n+3-t},w_{n+3-(m_0+1)}]$; is K-polystable if and only if $w\in(w_{n+3-t},w_{n+3-(m_0+1)})$.
 \end{enumerate}
\end{proposition}

\begin{center}
    \begin{longtable}{|l|l|l|l|l|l|l|l|}
 			\caption{List of cases of coefficients}\label{table: oddcoefficicentlist}\\
		\hline
		Case & $t$ & $a$ & $a_0$ & $m_0$ & $w$ & Proof  \\
				\hline
(a.) & $n+3$  & $\neq 0$ &  & &  $(\frac{n+2}{2n+6}, \xi_n)$ & Proposition \ref{ksemistable,poly-allcases}(1) \\
\hline
(b.) & $n+3$ & $\neq 0$ & &  &  $(\frac{n+2}{2n+6}, \frac{n+5}{2n+6})$ & Proposition \ref{ksemistable,poly-allcases}(2)\\
		\hline
(c.) & $n+3$ &  $0$ & $\neq 0$ &  & $(\frac{n+2}{2n+6}, \frac{n+5}{2n+6})$ & Proposition \ref{ksemistable,poly-allcases}(2) \\
\hline
(d.) & $n+3$ & $0$ & $0$ & $[0,\frac{n+1}{2}]$  &$(\frac{n+2}{2n+6}, \frac{n+5}{2n+6})$ & Proposition \ref{ksemistable,poly-allcases}(2)\\
		\hline	
(e.) & $n+3$ & $0$ & $0$ & $(\frac{n+1}{2},n+2]$  &$(\frac{n+2}{2n+6}, w_{n+3-(m_0+1)})$ & Proposition \ref{ksemistable,poly-allcases}(3)\\
		\hline	
(f.) & $(\frac{n+3}{2}, n+2]$ & $\neq 0$ &  &   &$(w_{n+3-t}, \frac{n+5}{2n+6})$ & Proposition \ref{ksemistable,poly-allcases}(2)\\
		\hline	
(g.) & $(\frac{n+3}{2}, n+2]$ & $0$ & $\neq 0$  &  &$(w_{n+3-t}, \frac{n+5}{2n+6})$ & Proposition \ref{ksemistable,poly-allcases}(2)\\
		\hline	
(h.) & $(\frac{n+3}{2}, n+2]$ & $0$ & $0$  & $[0,\frac{n+1}{2}]$  &$(w_{n+3-t}, \frac{n+5}{2n+6})$ & Proposition \ref{ksemistable,poly-allcases}(2)\\
		\hline
(i.) & $(\frac{n+3}{2}, n+2]$ & $0$ & $0$  & $(\frac{n+1}{2},t-1]$  &$(w_{n+3-t}, w_{n+3-(m_0+1)})$ & Proposition \ref{ksemistable,poly-allcases}(3)\\
		\hline	
(j.) & $[\frac{n+3}{2}, n+3]$ & $0$ & $0$  & $t-1$  &$w_{n+3-t}$ & Proposition \ref{evenpoly+polyanyn}\\
		\hline	
\end{longtable}
\end{center}

\begin{proof}
    Under the $\bG_m$ action $\lambda[x:y:z]=[x:\lambda^{-1}y: \lambda^{\frac{1-(n+3)}{2}}z],$ the curve $D$ degenerates to $D_0$ as $\lambda \rightarrow 0,$ where $D_0$ is defined by $z^2y + y^{n+3} = 0$.
    By Proposition \ref{evenpoly+polyanyn}, the pair $(W,\frac{n+2}{2n+6}D_0)$ is K-polystable. Then $(W,\frac{n+2}{2n+6}D)$ is K-semistable by Theorem \ref{openness}.

We first show (1).
By Proposition \ref{section5:main lemma} we see that the log pair $(W, \xi_n D)$ is K-semistable.  Therefore, by the first claim of Lemma 
\ref{theta,w,interpolation,polystable}, the pair $(W,wD)$ is K-semistable iff $w\in[\frac{n+2}{2n+6}, \xi_n]$. We have rank $\mathrm{Aut}(W,D_0)=1$ since the pair $(W,D_0)$ has complexity one. By Proposition \ref{evenpoly+polyanyn}, the pair $(W,\xi_nD_0)$ is not K-semistable. Then we have $(W,wD)$ is K-polystable iff $w\in(\frac{n+2}{2n+6}, \xi_n)$ by (2) in Lemma \ref{theta,w,interpolation,polystable}.

We now prove (2). Assume $t=n+3$, $a\neq 0$, and the curve $D$ does not has a type $A_{n+3}$-singularity. This corresponding to $t=n+3$ and $m\leq n+3$ in Lemma \ref{section 5:a is nonzero Lemma}.
So $(W,\frac{n+5}{2n+6}D)$ is log canonical. Assume $t=n+3$, $a=0$ and $a_0\neq 0$. This corresponding to $t=n+3$ in Lemma \ref{section5:a zero a_0 nonzero}. So $(W,\frac{n+5}{2n+6}D)$ is log canonical. Assume $\frac{n+3}{2}< t\leq n+2$, $\{a,a_0\}$ are not all zero, by Lemma \ref{section 5:a is nonzero Lemma} and Lemma \ref{section5:a zero a_0 nonzero}, the pair $(W,\frac{n+5}{2n+6}D)$ is log canonical. Assume $\frac{n+3}{2}< t\leq n+3$ and 
$a=a_0=0$ and $0\leq m_0\leq\frac{n+1}{2}$, by Lemma \ref{section5:a a_0 zeros}, the log pair $\left(W, \frac{n+5}{2n+6}D\right)$ is log canonical. Therefore, for each curve $D$ satisfying the condition in (2), we have $(W,\frac{n+5}{2n+6}D)$ is K-semistable.

On the other hand,
under the $\bG_m$ action $\lambda[x:y:z]=[x:\lambda^{-1}y:\lambda^{\frac{1-t}{2}}z],$ the curve $D$ degenerates to $D_t:z^2y+y^{2n+6-2t}y^t=0$ as $\lambda \rightarrow 0.$
By Proposition \ref{evenpoly+polyanyn},
the pair $(W, w_{n+3-t}D_t)$ is K-polystable. Then $(W,w_{n+3-t}D)$ is K-semistable by Theorem \ref{openness}. Therefore, by the first claim of Lemma \ref{theta,w,interpolation,polystable}, the pair $(W,wD)$ is K-semistable iff $w\in[w_{n+3-t},\frac{n+5}{2n+6}]$. We have rank $\mathrm{Aut}(W,D_t)=1$ since the pair $(W,D_t)$ has complexity one. By Proposition \ref{evenpoly+polyanyn}, the pair $(W,\frac{n+5}{2n+6}D_t)$ is not K-semistable. Then we have $(W,wD)$ is K-polystable iff $w\in(w_{n+3-t},\frac{n+5}{2n+6})$ by (2) in Lemma \ref{theta,w,interpolation,polystable}.

We now prove (3). Under the $\bG_m$ action $\lambda[x:y:z]\mapsto[\lambda^{-1}x: y:\lambda^{-(n+3-(m_0+1))}z]$, the curve $D$ degenerates to $D_{m_0+1}$ given by $z^2y+a_{m_0+1}x^{2n+6-2(m_0+1)}y^{m_0+1}=0$. On the other hand, under the $\bG_m$ action $\lambda[x:y:z]=[x:\lambda^{-1}y:\lambda^{\frac{1-t}{2}}z],$ the curve $D$ degenerates to $D_t:z^2y+y^{2n+6-2t}y^t=0$ as $\lambda \rightarrow 0.$ Proposition \ref{evenpoly+polyanyn} shows that
the pairs $(W, w_{n+3-t}D_t)$ and $(W, w_{n+3-(m_0+1)}D_{m_0+1})$ are K-polystable, so are K-semistable by Theorem \ref{openness}. Therefore, by the first claim of Lemma \ref{theta,w,interpolation,polystable}, the pair $(W,wD)$ is K-semistable iff $w\in[w_{n+3-t},w_{n+3-(m_0+1)}]$. We have rank $\mathrm{Aut}(W,D_t)=1$ since the pair $(W,D_t)$ has complexity one. By Proposition \ref{evenpoly+polyanyn}, the pair $(W,w_{n+3-(m_0+1)}D_t)$ is not K-semistable. Then we have $(W,wD)$ is K-polystable iff $w\in(w_{n+3-t},w_{n+3-(m_0+1)})$ by (2) in Lemma \ref{theta,w,interpolation,polystable}.
\end{proof}

\begin{proof}[\bf Proof of Theorem \ref{isomorphism,GITandK}]
By Proposition \ref{ksemistable,poly-allcases}, we get (2)-(5) as follows.
\begin{enumerate}
\item If $w\in(0,\frac{n+2}{2n+6})$, the K-moduli space $M_w$ is empty. This is implied by Lemma \ref{notK-s.s.}.
     \item If $w=w_0=\frac{n+2}{2n+6}$, then the K-moduli space $M_w$ contains the K-polystable pair $(W,wD)$ such that $a=a_0=\cdots=a_{n+2}=0$ and $a_{n+3}\neq 0$, i.e., $D:z^2y+a_{n+3}y^{n+3}=0$.
     \item If $w=w_{n+3-e}$, where $\frac{n+3}{2}+1\leq e\leq n+2$, then $M_w$ contains all pairs $(W,wD)$ from one of the following two cases:
     \begin{itemize}
         \item $\{a,a_0,\cdots,a_{e-1}\}$ are not all zero, $a_e$ is arbitrary, and $\{a_{e+1},\cdots,a_{n+3}\}$ are not all zero;
         \item $a=a_0=\cdots=a_{e-1}=0$, $a_e\neq0$ and $a_{e+1}=\cdots=a_{n+3}=0$ (K-polystable).
     \end{itemize}
     \item If $w\in(w_{n+3-e},w_{n+3-(e-1)})$, where $\frac{n+3}{2}+2\leq e\leq n+3$, then $M_w$ contains all pairs $(W,wD)$ such that $\{a,a_0,\cdots,a_{e-1}\}$ are not all zero, and $\{a_e,a_{e+1},\cdots,a_{n+3}\}$ are not all zero.
     \item If $c\in(w_{\frac{n+1}{2}},\xi_n)$,
     then $M_w$ contains all pairs $(W,wD)$ such that $\{a,a_0,\cdots,a_{\frac{n+3}{2}}\}$ are not all zero, and $\{a_{\frac{n+3}{2}+1},\cdots,a_{n+3}\}$ are not all zero.
     \end{enumerate}
     We replace $n+3-e$ by $i$. Then we know $\Phi_{w_0}$ maps the GIT moduli space $M^{GIT}_{w_0}$ described in Theorem \ref{describe,GITmoduli,odd}(2) onto the open subset of $M_{w_0}$ described in (2). Note that $i=n+3-e$. For $1\leq i\leq \frac{n+1}{2}$, the morphism $\Phi_{w_i}$ maps the GIT moduli space $M^{GIT}_{w_i}$ described in Theorem \ref{describe,GITmoduli,odd}(3) onto the open subset of $M_{w_i}$ described in (3). If $w\in(w_i,w_{i+1})$, where $0\leq i\leq \frac{n-1}{2}$, the morphism $\Phi_{w}$ maps the GIT moduli space $M^{GIT}_{w}$ described in Theorem \ref{describe,GITmoduli,odd}(4) onto the open subset of $M_{w}$ described in (4). When $w\in(w_{\frac{n+1}{2}},\xi_n)$, the morphism $\Phi_{w}$ maps the GIT moduli space $M^{GIT}_{w}$ described in Theorem \ref{describe,GITmoduli,odd}(4) onto the open subset of $M_{w}$ described in (5).

     We have $\Phi_w$ is a birational morphism since it is also injective. Since GIT moduli space is compact, the image of $\Phi_w$ is also compact, so equals the whole K-moduli space. Thus $\Phi_w$ gives the isomorphism.
\end{proof}

Then we descibe the K-moduli spaces $M_w$ in Theorem \ref{intro,K-moduliparameterize,i}.
Example \ref{examplen=3} gives the description when $n=3$. 

\begin{proof}[{\bf Proof of Theorem \ref{intro,K-moduliparameterize,i}(1)-(5)}]
Theorem \ref{intro,K-moduliparameterize,i}(1)-(5) are consequences of Theorem \ref{isomorphism,GITandK}.
\end{proof}

Now we can describe the K-moduli space $\mathcal{F}_n$ 
parameterizing these K-stable quasi-smooth hypersurfaces $H_{2(n+3)}$ in $\mathbb{P}(1,2,n+2,n+3)_{t,x,y,z}$ and their K-polystable limits. By the double cover construction in Lemma \ref{reduce1/2} and its generalization to families, we have $\mathcal{F}_n$ is isomorphic to $M_{\frac{1}{2}}$. Since $w_{\frac{n+3}{2}-1}<\frac{1}{2}< \xi_n$, then by Theorem \ref{intro,K-moduliparameterize,i} (5), we get the following corollary.

\begin{corollary}[=Theorem \ref{thm:main}(1)]\label{moduli,hypersurface}
Let $n$ be a positive odd integer. 
Then the K-polystable members in the K-moduli space $\mathcal{F}_n$ are precisely all the hypersurfaces $H_{2(n+3)}$ in $\mathbb{P}(1,2,n+2,n+3)_{t,x,y,z}$ given by 
\begin{equation*}
  \tilde{f}= t^2 + z^2y + azx^{n+4} + a_0x^{2n+6} + a_1x^{2n+4}y + \cdots + a_{n+3}y^{n+3}=0
\end{equation*}
    such that $\{a,a_0,a_1,\cdots,a_{\frac{n+3}{2}}\}$ are not all zero and $\{a_{\frac{n+3}{2}+1},\cdots,a_{n+3}\}$ are not all zero.
\end{corollary}

We descibe the wall crossing of K-moduli space $M_w$ for $w_{0}=\frac{n+2}{2n+6}\leq w< \xi_n$.

\begin{figure}
    \centering
    \caption{Wall crossing for K-moduli spaces at $w_{n,i}$}
    \label{Wall crossing at $c_{n,e}$}
\[
\xymatrix{
M_{w_{n,i}-\epsilon} \ar[dr]^{\psi^-_i} \ar@{-->}[rr] & & M_{w_{n,i}+\epsilon} \ar[dl]^{\psi^+_i} \\
& M_{w_{n,i}} & & 
} 
\]
\end{figure}
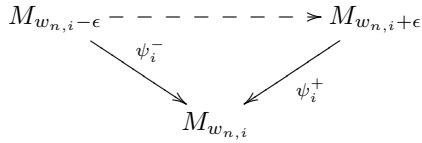

\begin{theorem}\label{wall crossings}\
    When $w\in(w_0,\xi_n)$, the surface of each pair in $M_w$ is isomorphic to $W$. The wall crossing of moduli spaces happens at $w_i$ for $1\leq i \leq \frac{n+3}{2}-1$. (Figure \ref{Wall crossing at $c_{n,e}$}).
    \begin{enumerate}
        \item We look at the first wall $w_{0}=\frac{n+2}{2n+6}$. Let $\Sigma_{\frac{n+2}{2n+6}+\epsilon}$ be the locus parameterizing the pairs $(W,(\frac{n+2}{2n+6}+\epsilon)D)$ such that $a_{n+3}\neq 0$ and  $\{a,a_0,\cdots,a_{n+2}\}$ are not all zero. The morphism $\psi^+_{0}: M_{\frac{n+2}{2n+6}+\epsilon} \rightarrow M_{\frac{n+2}{2n+6}}$ contracts $\Sigma_{\frac{n+2}{2n+6}+\epsilon}$ to a K-polystable point $(W,\frac{n+2}{2n+6}\{z^2y+y^{n+3}=0\})$.

    \item For $1\leq i \leq \frac{n+3}{2}-1$, the birational morphism $\psi^-_i: M_{w_{i}-\epsilon} \rightarrow M_{w_{i}}$ is isomorphic away from the locus $\Sigma_{w_{i}-\epsilon}$ parameterizing the pairs $(W,(w_{i}-\epsilon)D)$ such that $a=a_0=\cdots=a_{n+2-i}=0$, $a_{n+3-i}\neq0$, and $\{a_{n+4-i},\cdots,a_{n+3}\}$ are not all zero. Moreover, the morphism $\psi^-_i$ contracts $\Sigma_{w_{i}-\epsilon}$ to a point $(W,w_{i}\{z^2y+x^{2i}y^{n+3-i}=0\})$.
    \item For $1\leq i \leq \frac{n+3}{2}-1$, the birational morphism $\psi^+_i: M_{w_{i}+\epsilon} \rightarrow M_{w_{i}}$ is isomorphic away from the locus $\Sigma_{w_{i}+\epsilon}$
    parameterizing the pairs $(W,(w_{i}+\epsilon)D)$ such that $\{a,a_0,\cdots,a_{n+2-i}\}$ are not all zero, $a_{n+3-i}\neq0$ and $a_{n+4-i}=\cdots=a_{n+3}=0$. Moreover, the morphism $\psi^+_i$ contracts $\Sigma_{w_{i}+\epsilon}$ to a point $(W,w_{i}\{z^2y+x^{2i}y^{n+3-i}=0\})$.
\end{enumerate} 
\end{theorem}

\begin{proof}
Then (1) is implied by Remark \ref{degenrates to K-poly} and Theorem \ref{intro,K-moduliparameterize,i}.
And (2) and (3) are implied by Remark \ref{degenrates to K-poly} and Theorem \ref{intro,K-moduliparameterize,i}. 
\end{proof}

\medskip
\begin{example}\label{examplen=3}
We consider the case of $n=3$. Note that $W=\mathbb{P}(1,2,5)$, $D$ is given by
\begin{equation}\label{Dequation,n=3}
f=z^2y+azx^7+a_0x^{12}+a_1x^{10}y+a_2x^8y^2+a_3x^6y^3+a_4x^4y^4+a_5x^2y^5+a_6y^6,
\end{equation}
and $\xi_3=\frac{121}{204}$. Before $\frac{1}{2}$, all the walls are $w_{0},w_{1}, w_{2}$. And wall $\xi_3$ and $w_{3}$ are behind $\frac{1}{2}$. We have $$w_{0}=\frac{5}{12}<w_{1}=\frac{3}{7}<w_{2}=\frac{11}{24}<\frac{1}{2}<\xi_3=\frac{121}{204}<\frac{2}{3}=w_{3}.$$

We have the following description of Moduli spaces by Theorem \ref{intro,K-moduliparameterize,i}. 
When $\frac{5}{12}\leq w\leq\xi_n$, each pair in $M_w$ has surface isomorphic to $W$ (Table \ref{K-modulin=3}). When $\xi_3\leq w<\frac{2}{3}$, some pairs are replaced by new pairs of which the surface are not isomorphic to $W$ and with worse singularity than $W$; see Theorem \ref{intro,K-moduliparameterize,i}(6) or Section \ref{moduliafterxi_n}.

\begin{center}
    \begin{longtable}{|l|l|l|l|}
 			\caption{Explicit description of $M_w$ for $w<\xi_3$ when $n=3$}\label{K-modulin=3}\\
		\hline
		$M_w$  & coefficients   \\
				\hline

 $w=w_{0}=\frac{5}{12}$  &  $a=a_0=\cdots=a_5=0$, $a_6\neq 0$   \\
\hline

$\frac{5}{12}<w<\frac{3}{7}$  &  $\{a,a_0,\cdots,a_5\}$ not all zero, $a_6\neq 0$  \\
\hline

$w=w_{1}=\frac{3}{7}$  &   $\{a,a_0,\cdots,a_4\}$ not all zero, $a_5$ arbitrary, $a_6\neq 0$. \\
&   or $a=a_0=\cdots=a_4=0$, $a_5\neq 0$, $a_6=0$.  \\
\hline

$\frac{3}{7}<w<\frac{11}{24}$  &   $\{a,a_0,\cdots,a_4\}$ not all zero, $\{a_5,a_6\}$ not all zero \\
\hline

$w=w_{2}=\frac{11}{24}$  &   $\{a,a_0,\cdots,a_3\}$ not all zero, $a_4$ arbitrary, $\{a_5,a_6\}$ not all zero  \\
&   or $a=a_0=\cdots=a_3=0$, $a_4\neq 0$, $a_5=a_6=0$  \\
\hline

$\frac{11}{24}<w<\xi_3$  &  $\{a,a_0,\cdots,a_3\}$ not all zero, $\{a_4,a_5,a_6\}$ not all zero  \\
\hline

\end{longtable}
\end{center}
\end{example}

\section{K-moduli space when $n$ is even}\label{section,even}

Let $S\subset \PP(1, 2, n+2, n+3)_{x,y,z,t}$ be a singular del Pezzo hypersurface of degree $2n+6$ defined by a quasi-homogeneous polynomial
\begin{equation*}
    t^2 + z^2y + azx^{n+4} + a_0x^{2n+6} + a_1x^{2n+4}y + \cdots + a_{n+3}y^{n+3} = 0,
\end{equation*}
where $n$ is an even integer and $a$, $a_0,\ldots , a_{n+3}$ are constants, and let $T\subset \PP(1, 1, l+1, n+3)_{u,y,z,v}$ be the hypersurface defined by the quasi-homogeneous polynomial
\begin{equation*}
    v + z^2y + azu^{l+2} + a_0u^{2l+3} + a_1u^{2l+2}y + \cdots + a_{2l+3}y^{2l+3}  = 0
\end{equation*}
with $2l = n$. Then there is the double covering $\phi\colon S\to T$ defined by $[x:y:z:t]\mapsto  [x^2:y:z:t^2]$ with the ramification divisors $H_r$ and $D_r$ defined by $x=0$ and $t=0$, respectively. We have
\begin{equation*}
    K_S \equiv \phi^*(K_T) + H_r + D_r \equiv \phi^*\left(K_T + \frac{1}{2}H_b + \frac{1}{2}D_b\right),
\end{equation*}
where $H_b$ and $D_b$ the branch divisors defined by $u=0$ and $v=0$ in $T$, respectively.

By the isomorphism $\psi\colon\PP(1,1,l+1)_{u, y, z}\to T$ defined by $[u:y:z]\mapsto [u:y:z:z^2y + azu^{l+2} + a_0u^{2l+3} + a_1u^{2l+2}y + \cdots + a_{2l+3}y^{2l+3}]$ we can see that the log Fano pair $(T, \frac{1}{2}(H_b + D_b))$ is isomorphic to the log Fano pair $(\PP(1,1,l+1), \frac{1}{2}(H+ D))$ where $H$ and $D$ are the divisors defined by $u=0$ and $z^2y + azu^{l+2} + a_0u^{2l+3} + a_1u^{2l+2}y + \cdots + a_{2l+3}y^{2l+3} = 0$ in $\PP(1,1,l+1)$, respectively. Then by \cite[Theorem 1.2]{LiuZhu}, we have the following result.
\begin{lemma}\label{reduce1/2,even}
    For $n=2l$, where $l$ is a positive integer. The singular del Pezzo hypersurface $S$ is K-semistable if and only if the log Fano pair $(\PP(1,1,l+1), \frac{1}{2}(H+ D))$ is K-semistable.
\end{lemma}

 \begin{definition}\label{defineK-moduli,even}
Recall that $D$ is a curve in $W'=\mathbb{P}(1,1,l+1)$ given by 
\begin{equation}\label{defineequationD,even}
f=z^2y + azu^{l+2} + a_0u^{2l+3} + a_1u^{2l+2}y + \cdots + a_{2l+3}y^{2l+3} = 0,
\end{equation}
where $l\geq 1$. Let $r=\frac{4l+6}{2l+5}$. We have $D=\mathcal{O}(2l+3)$ and $\frac{1}{r}D+\frac{1}{2}H=K_{W'}$. We use the set-ups in section \ref{backgroudmultiple boundaries}. We fix a fano variety $W'$ and two boundary divisors $D_1=D$ and $D_2=H$. Let $d=2$, $v=(-K_{W'})^2=\frac{(l+3)^2}{l+1}$, $I=\{1\}$ and $\vec{w}=(\frac{1}{2},w)$.
We have $M_{2,v,I,\vec{w}}$ is the K-moduli space parametrizing all K-polystable log pairs $(W', \frac{1}{2}H+wD)$ for $0\leq w<\frac{2l+5}{4l+6}$.
Denote $M_w$ to be the closure of the locus in $M_{2,v,I,\vec{w}}$ consisting of pairs of which the surfaces are isomorphic to $W'$. We know $M_w$ is a closed subscheme of $M_{2,v,I,\vec{w}}$.
\end{definition}

\begin{definition}\label{definew_l,i}
Let $l\geq 1$ be an integer, for any integer $0\leq i \leq l+1$, we denote $$w_{i}:=\frac{(l+3-\frac{1}{2})(2l+2+i)-3(l+1)(1+i)}{(2l+3)(2l+2+i)-6i(l+1)},$$
and $\xi_{2l}:=\frac{2l^2+8l+3}{4l^2+12l+6}$.
Note that $w_{0}=\frac{l+1}{2l+3}$, $w_{l+1}=\frac{1}{2}$, and $\frac{l+1}{2l+3}\leq w_{i}\leq\frac{1}{2}<\xi_{2l}<\frac{2l+5}{4l+6}$ for any integer $0\leq i \leq l+1$.
\end{definition}

Now we relate the GIT moduli spaces $M^{GIT}_w$ defined in section \ref{VGIT} with $M_w$. Theorem \ref{isomorphism,GITandK,even} implies that they are isomorphic as long as $w<\xi_{2l}$. 

\begin{theorem}\label{isomorphism,GITandK,even}
For $w\in (0,\xi_n)$, the GIT moduli spaces $M^{GIT}_w$ are isomorphic to the K-moduli spaces $M_w$ by the morphism 
\[
\begin{split}
\Phi_w: M^{GIT}_w &\longrightarrow M_w \\(a,a_0,\cdots,a_{2l+3}) & \mapsto (W',\frac{1}{2}H+wD),
\end{split}
\]
where $D$ is given by the equation
\begin{equation*}
     z^2y + azu^{l+2} + a_0u^{2l+3} + a_1u^{2l+2}y + \cdots + a_{2l+3}y^{2l+3} = 0.
\end{equation*}
\end{theorem}

In order to prove Theorem \ref{isomorphism,GITandK,even}, we need several Lemmas and Propositions as follows to describe an open subset of $M_w$.

We first state a relatively easy direction, that is, K-stability implies GIT as long as the surface $W'$ together with the curve $H$ does not change.

\begin{theorem}\label{thm:K-imply-GIT-even}
Let $D$ be a curve in $W'$ given by the equation \eqref{defineequationD,even}. Let $w\in (0, \frac{2l+5}{4l+6})\cap \bQ$. If $(W',\frac{1}{2}H + wD)$ is K-semistable (resp. K-polystable), then $(a, a_0, \cdots, a_{2l+3})$ is $w$-GIT semistable (resp. $w$-GIT polystable).
\end{theorem}

\begin{proof}
When $w> \frac{2n+1}{4n+6}$, this follows directly from the Paul--Tian criterion \cite{PT06} (see \cite[Theorem 2.22]{ADL19} for a precise version) applied to the universal family over $\cP(1, 2^{n+4})$ and Proposition \ref{prop:projGIT} as the linearization $\Lambda_w$ is the CM $\bQ$-line bundle for the universal family of $(W', \frac{1}{2}H+wD)$. When $w\leq \frac{2n+1}{4n+6}$,  we know that $\lambda_{(0,-1)}$ induces a test configuration degenerating $(W',\frac{1}{2}H + wD)$ to $(W',\frac{1}{2}H + w D_0)$ whose Futaki invariant is equal to $\beta^w_{(0,-1)}<0$, thus $(W',\frac{1}{2}H + wD)$ is always K-unstable. When $w = \frac{2n+1}{4n+6}$, the same argument implies that either $(W',\frac{1}{2}H+wD)$ is K-unstable or it is K-semistable and $\lambda_{(0,-1)}$ has Futaki invariant  $\beta^w_{(0,-1)}=0$. Since $(W',\frac{1}{2}H+wD)$ is K-semistable, we are in the latter case. Then by \cite[Lemma 3.1]{LWX21} we know that $(W',\frac{1}{2}H+wD_0)$ is K-semistable which implies that $\beta_{(d,b)}^w = 0$ for any $(d,b)\in \bZ^2$. This is absurd as $1-n + 2nw \neq 0$. Thus the proof is finished.
\end{proof}

The pairs in Proposition \ref{evenpoly+polyanyn,even} with a $\bG_m$ action have complexity one.

\begin{proposition}\label{evenpoly+polyanyn,even}
For any integer $l+2\leq e \leq 2l+3$, where $l\geq 1$,
the pair $(W',\frac{1}{2}H+wD_e)$ is K-polystable if and only if $w=w_{2l+3-e}$, where $D_e$ is given by $z^2y+u^{2l+3-e}y^e=0$, $H:=\{u=0\}$, 
and $\frac{l+1}{2l+3}\leq w_{2l+3-e}\leq \frac{1}{2}<\frac{2l+5}{4l+6}$. Moreover, when $w\neq w_{2l+3-e}$, $(W',\frac{1}{2}H+wD_e)$ is K-unstable.
\end{proposition}
\begin{proof}
We can check that $$w_{2l+3-e}\leq\frac{1}{2} \text{ iff } l+2\leq e, \text{ and } \frac{l+1}{2l+3}\leq w_{2l+3-e} \text{ iff } e \leq 2l+3.$$
The pair $(W', \frac{1}{2}H+wD_e)$ is a complexity one $\mathbb{T}$-pair with an action by $\mathbb{T}=\mathbb{G}_m$: $\lambda[u,y,z]=[\lambda^2 u, y, \lambda^{2l+3-e} z]$ for $\lambda \in \mathbb{G}_m$. In the affine open set $U_y$ defined by $y=1$, we have $D_e$ is given by $z^2+u^{2l+3-e}=0$.
We use Theorem \ref{Tone} and Remark \ref{remarkTone} to check K-polysablity. 

The monomial valuation $v_{(2,2l+3-e)}$ such that $v(u)=2$ and $v(z)=2l+3-e$ is induced by the $\mathbb{G}_m$-action. 
We first compute the $\beta$-invariant for $v=v_{(2,2l+3-e)}$. 
Then by Lemma \ref{S-inv.}, we get the S-invariant for $v$: $$S_{\mathcal{O}(1)}(v)=\frac{1}{3}(2+\frac{2l+3-e}{l+1})=\frac{4l+5-e}{3(l+1)}.$$ We also have $-K_W'-(\frac{1}{2}H+wD_e)=\mathcal{O}(l+3-\frac{1}{2}-w(2l+3))$ and $S_{(W',\frac{1}{2}H+wD_e)}(v)=\mathrm{deg}(-K_{W'}-(\frac{1}{2}H+wD_e))\cdot S_{\mathcal{O}(1)}(v)=(l+3-\frac{1}{2}-w(2l+3))\frac{4l+5-e}{3(l+1)}.$
We also have the log discrepancies $A_{W'}(v)=2l+5-e$ and 
\begin{equation*}
    A_{(W',\frac{1}{2}H+wD_e)}(v)=A_{W'}(v)-v(\frac{1}{2}H+wD_e)=2l+4-e-2w(2l+3-e).
\end{equation*}
Then the $\beta$-invariant 
\begin{equation}\label{horizontal-c-D_e,even}
\beta_{(W',\frac{1}{2}H+wD_e)}(v)=A_{(W',\frac{1}{2}H+wD_e)}(v)-S_{(W',\frac{1}{2}H+wD_e)}(v)=0 \text{ if and only if } w=w_{2l+3-e}.
\end{equation}

By Remark \ref{remarkTone}, there is no horizontal $\mathbb{T}$-invariant prime divisor.
If $e$ is even, We have vertical
$\mathbb{T}$-invariant prime divisors 
$\{[u;0;z]\}, \{[0;y;z]\}, \{[u;y;0]\}, \{z^2+u^{2l+3-e}y^{e-1}=0\}, \{z^2+au^{2l+3-e}y^{e-1}=0,a\neq 1\}$ under the $\bG_m$ action $\lambda[u,y,z]=[\lambda^2 u, y, \lambda^{2l+3-e} z]$.
If $e$ is odd, the vertical
$\mathbb{T}$-invariant prime divisors are
$\{[u;0;z]\}, \{[0;y;z]\}, \{[u;y;0]\}, \{z+ au^{\frac{2l+3-e}{2}}y^{\frac{e-1}{2}}=0,a=\pm i\}, \{z+au^{\frac{2l+3-e}{2}}y^{\frac{e-1}{2}}=0,a\neq \pm i\}$, where $i^2=-1$.
Next, we check $\beta(E)> 0$ for every vertical $\mathbb{T}$-invariant prime divisor $E$. 

Since $-K_{W'}-\frac{1}{2}H-wD_0=\big(l+3-(\frac{1}{2}+w(2l+3))\big)\mathcal{O}(1)$, by Lemma \ref{Sinv,k}, we have $$S_{(W',\frac{1}{2}H+wD_0)}(F)=\frac{l+3}{3k}-\frac{\frac{1}{2}+w(2l+3)}{3k}$$ for $F\in|\mathcal{O}(k)|$.

If $E$ is $\{[u;0;z]\}$, then $E\in|\mathcal{O}(1)|$. We have $S_{(W',\frac{1}{2}H+wD_e)}(E)=\frac{l+3}{3}-\frac{1}{3}(\frac{1}{2}+w(2l+3))$. Since $D_e$ is given by $y(z^2+u^{2l+3-e}y^{e-1})=0$, so $A_{(W',\frac{1}{2}H+wD_e)}(E)=1-\mathrm{Coeff}_E(\frac{1}{2}H+wD_e)=1-w$. Then $\beta_{(W',\frac{1}{2}H+wD_e)}(E)=A_{(W',\frac{1}{2}H+wD_e)}(E)-S_{(W',\frac{1}{2}H+wD_e)}(E)>0$ iff $w> \frac{l-\frac{1}{2}}{2l}$. Since $w_{2l+3-e}>\frac{l+1}{2l+3}>\frac{l-\frac{1}{2}}{2l}$, the constant $w_{2l+3-e}$ satisfies this inequality.

If $E$ is $\{[0;y;z]\}$, then $E\in|\mathcal{O}(1)|$. We have $S_{(W',\frac{1}{2}H+wD_e)}(E)=\frac{l+3}{3}-\frac{1}{3}(\frac{1}{2}+w(2l+3))$. We have $A_{(W',\frac{1}{2}H+wD_e)}(E)=1-\mathrm{Coeff}_E(\frac{1}{2}H+wD_e)=1-\frac{1}{2}$. Then $\beta_{(W',\frac{1}{2}H+wD_e)}(E)>0$ iff $w> \frac{l+1}{2l+3}$. So $w_{2l+3-e}$ satisfies this inequality.

If $E$ is $\{[u;y;0]\}$, then $E\in|\mathcal{O}(l+1)|$. We have $S_{(W',\frac{1}{2}H+wD_e)}(E)=\frac{l+3}{3(l+1)}-\frac{1}{3(l+1)}(\frac{1}{2}+w(2l+3))$. We have $A_{(W',\frac{1}{2}H+wD_e)}(E)=1-\mathrm{Coeff}_E(\frac{1}{2}H+wD_e)=1$. Then we get $\beta_{(W',\frac{1}{2}H+wD_e)}(E)>0$ iff $(2l+3)w>-2l-\frac{1}{2}$ which is satisfied by any $w>0$, so is the constant $w_{2l+3-e}$.

When $e$ is even, we consider $E:z^2+au^{2l+3-e}y^{e-1}=0$. We have $E\in|\mathcal{O}(2l+2)|$ and  $S_{(W',\frac{1}{2}H+wD_e)}(E)=\frac{l+3}{6(l+1)}-\frac{1}{6(l+1)}(\frac{1}{2}+w(2l+3))$. If $a=1$, then $A_{(W',\frac{1}{2}H+wD_e)}(E)=1-\mathrm{Coeff}_E(\frac{1}{2}H+wD_e)=1-w$. 
Then $\beta_{(W',\frac{1}{2}H+wD_e)}(E)>0$ iff $w<\frac{5l+3+\frac{1}{2}}{4l+3}$ which is satisfied by $w_{2l+3-e}$ since $w_{2l+3-e}<\frac{2l+5}{4l+6}<\frac{5l+3+\frac{1}{2}}{4l+3}$.
If $a\neq 1$, then $A_{(W',\frac{1}{2}H+wD_e)}(E)=1-\mathrm{Coeff}_E(\frac{1}{2}H+wD_e)=1$. Then $\beta_{(W',\frac{1}{2}H+wD_e)}(E)>0$ iff $(2l+3)w>-5l-\frac{1}{2}-3$ which is satisfied by any $w>0$, so is the constant $w_{2l+3-e}$.

When $e$ is odd, we consider $E:z+au^{\frac{2l+3-e}{2}}y^{\frac{e-1}{2}}=0$.
We have $E\in|\mathcal{O}(l+1)|$
and $S_{(W',\frac{1}{2}H+wD_e)}(E)=\frac{l+3}{3(l+1)}-\frac{1}{3(l+1)}(\frac{1}{2}+w(2l+3))$. If $a=\pm i$, then $A_{(W',\frac{1}{2}H+wD_e)}(E)=1-\mathrm{Coeff}_E(\frac{1}{2}H+wD_e)=1-w$. 
Then $\beta_{(W',\frac{1}{2}H+wD_e)}(E)>0$ iff $w<\frac{2l+\frac{1}{2}}{l}$ which is satisfied by $w_{2l+3-e}$ since $w_{2l+3-e}<\frac{2l+5}{4l+6}<\frac{2l+\frac{1}{2}}{l}$.
If $a\neq \pm i$, then $A_{(W',\frac{1}{2}H+wD_e)}(E)=1-\mathrm{Coeff}_E(\frac{1}{2}H+wD_e)=1$. Then $\beta_{(W',\frac{1}{2}H+wD_e)}(E)>0$ iff $(2l+3)w>-2l-\frac{1}{2}$ which is satisfied by any $w>0$, so is the constant $w_{2l+3-e}$.

Thus $(W',\frac{1}{2}H+w_{2l+3-e}D_e)$ is K-polystable by Theorem \ref{Tone}.

Now we show $(W',\frac{1}{2}H+wD_e)$ is K-semistable only when $w=w_{2l+3-e}$. By Theorem \ref{Tone}, if $(W',\frac{1}{2}H+wD_e)$ is K-semistable, we have $\beta_{(W',\frac{1}{2}H+wD_e)}(v)=0$. By \eqref{horizontal-c-D_e,even}, $\beta_{(W',\frac{1}{2}H+wD_e)}(v)=0$ is equivalent to $w=w_{2l+3-e}$. Hence the conclusion holds.

\end{proof}

\begin{remark}\label{degenrates to K-poly,even}\
We will use the following degenerations to show K-semistability.
\begin{enumerate}
    \item Let $l+2\leq e\leq 2l+3$. Assume $D$ is defined by Equation \ref{defineequationD,even} such that $a_{e}\neq 0$, $a_{e+1}=\cdots=a_{2l+3}=0$ and $\{a,a_0,\cdots,a_{e-1}\}$ are not all zero. Then the pairs $(W',\frac{1}{2}H+D)$ degenerate to $(W',\frac{1}{2}H+\{z^2y+u^{2l+3-e}y^{e}=0\})$ under the $\bG_m$ action defined by $\lambda[u:y:z]=[u:\lambda^{-1}y:\lambda^{\frac{1-t}{2}}z]$, as $\lambda \rightarrow 0$.
    \item Let $l+1\leq e\leq 2l+2$. Assume $D$ is defined by Equation \ref{defineequationD,even} such that $a=a_0=\cdots=a_{e-1}=0$, $a_{e}\neq 0$ and $\{a_{e+1},\cdots,a_{2l+3}\}$ are not all zero. Then the pairs $(W',\frac{1}{2}H+D)$ degenerate to $(W',\frac{1}{2}H+\{z^2y+u^{2l+3-e}y^{e}=0\})$ under the $\bG_m$ action defined by $\lambda[u:y:z]\mapsto[\lambda^{-1}u:y:\lambda^{\frac{-(2l+3-e)}{2}}z]$, as $\lambda \rightarrow 0$.
\end{enumerate}
\end{remark}

We set $W' = \PP(1,1,l+1)_{u,y,z}$, where $l\geq 1$. Let $H$ and $D$ be the effective divisors on $W'$ defined by $u=0$ and 
\begin{equation}\label{eveneuation2}
    z^2y + azu^{l+2} + a_0u^{2l+3} + a_1u^{2l+2}y + \cdots + a_{2l+3}y^{2l+3} = 0,
\end{equation} 
respectively, where $a, a_0,\ldots , a_{2l+3}$ are constants. Finding the log canonical threshold $\lct(W', \frac{1}{2}H ; D)$ is similar to the odd case. When $a_0,\ldots , a_{2l+3}$ are not all zero, let $t=\max\{i|a_i\neq 0\}$. 
When $a=a_0=0$ and $a_1,\ldots , a_{2l+3}$ are not all zero, we denote $m_0$ such that $a_1,\ldots , a_{m_0}$ are zero and $a_{m_0+1}$ is not zero. That is to say
$m_0+1=\min\{i|a_i\neq 0\}$.
Then the polynomial for $D$ is \begin{equation*}
    z^2y + a_{m_0+1}u^{2l+3-(m_0+1)}y^{m_0+1} + \cdots + a_{t}u^{2l+3-t}y^{t} = 0.
\end{equation*}

\begin{proposition}\label{t<=(n+3)/2,even}
Assume $a$ is any integer.
\begin{enumerate}
    \item If $a_0=\cdots=a_{2l+3}=0$, then $(W',\frac{1}{2}H+wD)$ is not K-semistable for $\frac{l+1}{2l+3}\leq w\leq \frac{2l+5}{4l+6}$.
    \item If $0\leq t\leq l+1$, then $(W',\frac{1}{2}H+wD)$ is not K-semistable for $\frac{l+1}{2l+3}\leq w \leq \frac{2l+5}{4l+6}$.
\end{enumerate}
\end{proposition}
\begin{proof}
    The conclusion is implied by Lemma \ref{alla,a_izero,notK-s.s,even}, Lemma \ref{all,a_izero,notK-s.s,even} and Lemma \ref{0<=t<=(n+3)/2notk-s.s.,even}.
\end{proof}

\begin{lemma}\label{notK-s.s.even}
The log Fano pair $(W', \frac{1}{2}H+wD)$ is not K-semistable if $w<\frac{l+1}{2l+3}$.  
\end{lemma}
\begin{proof}
Since $-K_{W'}-\frac{1}{2}H-wD=\big(l+3-(\frac{1}{2}+w(2l+3))\big)\mathcal{O}(1)$, by Lemma \ref{Sinv,k}, we have $S_{(W',\frac{1}{2}H+wD)}(H)=\frac{1}{3}(l+3-\frac{1}{2}-w(2l+3))$. Thus the $\beta$-invariant 
    \begin{equation*}
        \beta_{(W',\frac{1}{2}H+wD)}(H) = A_{(W',\frac{1}{2}H+wD)}(H)-S_{(W',\frac{1}{2}H+wD)}(H)= \frac{1}{2} - \frac{1}{3}\left(l+3-\frac{1}{2}-w(2l+3)\right).
    \end{equation*}
    We have $\beta_{(W',\frac{1}{2}H+wD)}(H)\geq 0$ iff $w\geq \frac{l+1}{2l+3}$.
\end{proof}

We recall that $\xi_{2l}:=\frac{2l^2+8l+3}{4l^2+12l+6}$.
Assume $a_{2l+3}$ is not zero, i.e., $t=2l+3$. By Lemma \ref{section5 even:a nonzero}, when $a$ is not zero, then the log Calabi-Yau pair $(W', \frac{1}{2}H + \frac{2l+5}{4l+6}D)$ is log canonical if and only if $m\leq 2l+3$, where $m$ is the number defined in 
(\ref{section 5 even:maximal m}). The curve $D$ has a type $A_{2l+3}$-singularity iff $m = 2l+4$.

\begin{proposition}\label{ksemistable,allcases,even}
The K-stability of $(W',\frac{1}{2}H+wD)$ depends on whether the coefficients in the polynomial \eqref{eveneuation2} are zero or not. We list all the cases as follows (Table \ref{table: oddcoefficicentlist,even}):    
\begin{enumerate}
    \item Assume $a_{2l+3}\neq 0$ and $a\neq 0$. if the curve $D$ has a type $A_{2l+3}$-singularity, then the log pair $(W',\frac{1}{2}H+ wD)$ is K-semistable if and only if $w\in[\frac{l+1}{2l+3},\xi_{2l}]$; is K-polystable if and only if $w\in(\frac{l+1}{2l+3},\xi_{2l})$.
    \item Assuem $a_t\neq 0$, $a_{t+1}=\cdots=a_{2l+3}=0$ and $\{a,a_0,\cdots,a_{m_0+1}\}$ are not all zero, where $l+2\leq t\leq 2l+3$ and $0\leq m_0\leq l$. When $t=2l+3$ and $a\neq 0$, we assume that the curve $D$ does not has a type $A_{2l+3}$-singularity. Then $(W',\frac{1}{2}H+ wD)$ is K-semistable if and only if $w\in[w_{2l+3-e},\frac{2l+5}{4l+6}]$; is K-polystable if and only if $w\in(w_{2l+3-e},\frac{2l+5}{4l+6})$.
    \item Assume $a=a_0=\cdots=a_{m_0}=0$, $a_{m_0+1}\neq0$, $a_t\neq$ and $a_{t+1}=\cdots=a_{2l+3}=0$, where $l+2\leq t\leq 2l+3$ and $l+1\leq m_0\leq t-1$. Then $(W',\frac{1}{2}H+wD)$ is K-semistable for $w\in[w_{2l+3-t},w_{2l+3-(m_0+1)}]$; is K-polystable for $w\in(w_{2l+3-t},w_{2l+3-(m_0+1)})$.
\end{enumerate}
\end{proposition}

 \begin{center}
    \begin{longtable}{|l|l|l|l|l|l|l|l|}
 			\caption{List of cases of coefficients}\label{table: oddcoefficicentlist,even}\\
		\hline
		Case & $t$ & $a$ & $a_0$ & $m_0$ & $w$ & Proof  \\
				\hline
(a.) & $2l+3$  & $\neq 0$ &  & &  $(\frac{l+1}{2l+3}, \xi_{2l})$ & Proposition \ref{ksemistable,allcases,even}(1) \\
\hline
(b.) & $2l+3$ & $\neq 0$ & &  &  $(\frac{l+1}{2l+3}, \frac{2l+5}{4l+6})$ & Proposition \ref{ksemistable,allcases,even}(2)\\
		\hline
(c.) & $2l+3$ &  $0$ & $\neq 0$ &  & $(\frac{l+1}{2l+3}, \frac{2l+5}{4l+6})$ & Proposition \ref{ksemistable,allcases,even}(2) \\
\hline
(d.) & $2l+3$ & $0$ & $0$ & $[0,l]$  &$(\frac{l+1}{2l+3}, \frac{2l+5}{4l+6})$ & Proposition \ref{ksemistable,allcases,even}(2)\\
		\hline	
(e.) & $2l+3$ & $0$ & $0$ & $[l+1,2l+2]$  &$(\frac{l+1}{2l+3}, w_{2l+3-(m_0+1)})$ & Proposition \ref{ksemistable,allcases,even}(3)\\
		\hline	
(f.) & $[l+2, 2l+2]$ & $\neq 0$ &  &   &$(w_{2l+3-t},\frac{2l+5}{4l+6})$ & Proposition \ref{ksemistable,allcases,even}(2)\\
		\hline	
(g.) & $[l+2, 2l+2]$ & $0$ & $\neq 0$  &  &$(w_{2l+3-t}, \frac{2l+5}{4l+6})$ & Proposition \ref{ksemistable,allcases,even}(2)\\
		\hline	
(h.) & $[l+2, 2l+2]$ & $0$ & $0$  & $[0,l]$  &$(w_{2l+3-t}, \frac{2l+5}{4l+6})$ & Proposition \ref{ksemistable,allcases,even}(2)\\
		\hline
(i.) & $[l+2, 2l+2]$ & $0$ & $0$  & $[l+1,t-1]$  &$(w_{2l+3-t}, w_{2l+3-(m_0+1)})$ & Proposition \ref{ksemistable,allcases,even}(3)\\
		\hline	
(j.) & $[l+2, 2l+3]$ & $0$ & $0$  & $t-1$  &$w_{2l+3-t}$ & Proposition \ref{evenpoly+polyanyn,even}\\
		\hline	
\end{longtable}
\end{center}

\begin{proof}
    $\lambda[u:y:z]=[u:\lambda^{-1}y:\lambda^{\frac{1-(2l+3)}{2}}z],$ the curve $D$ degenerates to $D_0$ as $\lambda \rightarrow 0$, where $D_0$ is defined by $z^2y + y^{2l+3} = 0$.
    By Proposition \ref{ksemistable,allcases,even}, the log pair $(W',\frac{1}{2}H+\frac{l+1}{2l+3}D_0)$ is K-polystable. Then $(W',\frac{1}{2}H + \frac{l+1}{2l+3}D)$ is K-semistable by Theorem \ref{openness}.

    We first show (1). By Proposition \ref{section5:main lemma}, the log pair $(W', \frac{1}{2}H+\xi_{2l} D)$ is K-semistable.  Therefore by the first claim of Lemma \ref{theta,w,interpolation,polystable}, the pair $(W', \frac{1}{2}H+wD)$ is K-semistable iff $w\in[\frac{l+1}{2l+3},\xi_{2l}].$ We have rank $\mathrm{Aut}(W',\frac{1}{2}H+D_0)=1$ since the pair $(W',\frac{1}{2}H+D_0)$ has complexity one. By Proposition \ref{ksemistable,allcases,even}, the pair $(W',\frac{1}{2}H+\xi_{2l}D_0)$ is not K-semistable. Then we have $(W',\frac{1}{2}H+wD)$ is K-polystable iff $w\in(\frac{l+1}{2l+3},\xi_{2l})$ by (2) in Lemma \ref{theta,w,interpolation,polystable}. 

    We now prove (2). Assume $t=2l+3$, $a\neq 0$ and the curve $D$ does not has a type $A_{2l+3}$-singularity. This corresponding to $t=2l+3$ and $m\leq 2l+3$ in Lemma \ref{section5 even:a nonzero}.
So $(W', \frac{1}{2}H+\frac{2l+5}{4l+6}D)$ is log canonical. Assume $t=2l+3$, $a=0$ and $a_0\neq 0$.
This corresponding to $t=2l+3$ in Lemma \ref{section5 even:a zero a_0 nonzero}. So $(W', \frac{1}{2}H+\frac{2l+5}{4l+6}D)$ is log canonical.
Assume $l+2\leq t\leq 2l+2$ and $\{a,a_0\}$ are not all zero, by Lemma \ref{section5 even:a nonzero} and Lemma \ref{section5 even:a zero a_0 nonzero}, the pair $(W',\frac{1}{2}H+\frac{2l+5}{4l+6}D)$ is log canonical. Assume $l+2\leq t\leq 2l+3$, $a=a_0=0$ and $0\leq m_0\leq l$,  by Lemma \ref{section5 even:a and a_0 zero}, the log pair $\left(W',\frac{1}{2}H+\frac{2l+5}{4l+6}D\right)$ is log canonical. Therefore, for each curve $D$ satisfying the condition in (2), we have 
$(W',\frac{1}{2}H+\frac{2l+5}{4l+6}D)$ is K-semistable.

On the other hand, under the $\bG_m$ action $\lambda[u:y:z]=[u:\lambda^{-1}y:\lambda^{\frac{1-t}{2}}z],$ the curve $D$ degenerates to $D_t:z^2y+u^{2l+3-t}y^t$ as $\lambda \rightarrow 0.$
By Proposition \ref{evenpoly+polyanyn,even},
the pair $(W',\frac{1}{2}H+w_{2l+3-t}D_t)$ is K-polystable. Then $(W',\frac{1}{2}H+w_{2l+3-t}D)$ is K-semistable by Theorem \ref{openness}. Therefore, by the first claim of Lemma \ref{theta,w,interpolation,polystable}, the pair $(W', \frac{1}{2}H+wD)$ is K-semistable iff $w\in[w_{2l+3-t},\frac{2l+5}{4l+6}].$ We have rank $\mathrm{Aut}(W',\frac{1}{2}H+D_t)=1$ since the pair $(W',\frac{1}{2}H+D_t)$ has complexity one. By Proposition \ref{evenpoly+polyanyn,even}, the pair $(W',\frac{1}{2}H+\frac{2l+5}{4l+6}D_t)$ is not K-semistable. Then we have $(W',\frac{1}{2}H+wD)$ is K-polystable iff $w\in(w_{2l+3-t},\frac{2l+5}{4l+6})$ by (2) in Lemma \ref{theta,w,interpolation,polystable}. 

We now prove (3). Under the $\bG_m$ action $\lambda[u:y:z]=[\lambda^{-1}u: y:\lambda^{\frac{-(2l+3-(m_0+1))}{2}}z]$, the curve $D$ degenerates to $D_{m_0+1}$ given by $z^2y+a_{m_0+1}x^{2l+3-(m_0+1)}y^{m_0+1}=0$ as $\lambda \rightarrow 0.$ On the other hand, under the $\bG_m$ action $\lambda[u:y:z]=[u:\lambda^{-1}y:\lambda^{\frac{1-t}{2}}z],$ the curve $D$ degenerates to $D_t:z^2y+u^{2l+3-t}y^t=0$ as $\lambda \rightarrow 0.$ Proposition \ref{evenpoly+polyanyn,even} shows that the pairs $(W',\frac{1}{2}H+w_{2l+3-t}D_t) $ and $(W',\frac{1}{2}H+w_{2l+3-(m_0+1)}D_{m_0+1})$ are K-polystable, so are K-semistable by Theorem \ref{openness}. Therefore, by the first claim of Lemma \ref{theta,w,interpolation,polystable}, the pair $(W', \frac{1}{2}H+wD)$ is K-semistable iff $w\in[w_{2l+3-t},w_{2l+3-(m_0+1)}].$ We have rank $\mathrm{Aut}(W',\frac{1}{2}H+D_t)=1$ since the pair $(W',\frac{1}{2}H+D_t)$ has complexity one. By Proposition \ref{evenpoly+polyanyn,even}, the pair $(W',\frac{1}{2}H+w_{2l+3-(m_0+1)}D_t)$ is not K-semistable. Then we have $(W',\frac{1}{2}H+wD)$ is K-polystable iff $w\in(w_{2l+3-t},w_{2l+3-(m_0+1)})$ by (2) in Lemma \ref{theta,w,interpolation,polystable}. 
\end{proof}

\begin{proof}[\bf Proof of Theorem \ref{isomorphism,GITandK,even}]
    By Proposition \ref{ksemistable,allcases,even}, we get (2)-(5) as follows.
    \begin{enumerate}
        \item If $w\in (0,\frac{l+1}{2l+3})$, the K-moduli space $M_w$ is empty. This is implied by Lemma \ref{notK-s.s.even}.
\item If $w=w_0=\frac{l+1}{2l+3}$, then $M_w$ parameterizing the K-polystable pair $(W',\frac{1}{2}H+wD)$ such that $a=a_0=\cdots=a_{2l+2}=0$ and $a_{2l+3}\neq 0$, i.e., $D:z^2y+a_{2l+3}y^{2l+3}=0$.
     \item If $w=w_{2l+3-e}$, where $l+2\leq e\leq 2l+2$, then $M_w$ contains all pairs $(W',\frac{1}{2}H+wD)$ such that 
     \begin{itemize}
         \item $\{a,a_0,\cdots,a_{e-1}\}$ are not all zero, $a_e$ is arbitrary, and $\{a_{e+1},\cdots,a_{2l+3}\}$ are not all zero;
         \item $a=a_0=\cdots=a_{e-1}=0$, $a_e\neq0$ and $a_{e+1}=\cdots=a_{2l+3}=0$ (K-polystable).
     \end{itemize}
     \item If $w\in(w_{2l+3-e},w_{2l+3-(e-1)})$, where $l+3\leq e\leq 2l+3$, then $M_w$ contains all pairs $(W',\frac{1}{2}H+wD)$ such that $\{a,a_0,\cdots,a_{e-1}\}$ are not all zero, $\{a_e,a_{e+1},\cdots,a_{2l+3}\}$ are not all zero.
     \item If $w\in(w_{l+1},\xi_{2l})$,
     then $M_w$ contains all pairs $(W',\frac{1}{2}H+wD)$ such that $\{a,a_0,\cdots,a_{l+1}\}$ are not all zero, $\{a_{l+2},\cdots,a_{2l+3}\}$ are not all zero.
\end{enumerate}
We replace $2l+3-e$ by $i$. Then we know $\Phi_{w_0}$ maps the GIT moduli space $M^{GIT}_{w_0}$ described in Theorem \ref{describe,GITmoduli,even}(2) onto the open subset of $M_{w_0}$ described in (2). Note that $i=2l+3-e$. For $1\leq i\leq l+1$, the morphism $\Phi_{w_i}$ maps the GIT moduli space $M^{GIT}_{w_i}$ described in Theorem \ref{describe,GITmoduli,even}(3) onto the open subset of $M_{w_i}$ described in (3). If $w\in (w_i,w_{i+1})$, where $0\leq i\leq l$, the morphism $\Phi_{w}$ maps the GIT moduli space $M^{GIT}_{w}$ described in Theorem \ref{describe,GITmoduli,even}(4) onto the open subset of $M_{w}$ described in (4). 
When $w\in(w_{l+1},\xi_{2l})$, the morphism $\Phi_{w}$ maps the GIT moduli space $M^{GIT}_{w}$ described in Theorem \ref{describe,GITmoduli,even}(5) onto the open subset of $M_{w}$ described in (5).

We have $\Phi_{w}$ is a birational morphism since it is also injective. Since GIT moduli space is compact, the image of $\Phi_w$ is also compact, so equals the whole K-moduli space. Thus $\Phi_w$ gives the isomorphism.
\end{proof}

Now we describe the K-moduli space $M_w$ in Theorem \ref{intro,K-moduliparameterize,even,i}. Example \ref{l=1,example} gives the description when $l=1$.

\begin{proof}[\bf Proof of Theorem \ref{intro,K-moduliparameterize,even,i}(1)-(5)]
 Theorem \ref{intro,K-moduliparameterize,even,i}(1)-(5) are consequences of Theorem \ref{isomorphism,GITandK,even}.  
\end{proof}

Now for any even positive integer $n=2l$, 
we can describe the K-moduli space $\mathcal{F}_{2l}$ 
parameterizing these K-stable quasi-smooth hypersurfaces $H_{2(n+3)}$ in $\mathbb{P}(1,2,n+2,n+3)_{t,x,y,z}$ and their K-polystable limits. By the double cover construction in Lemma \ref{reduce1/2,even} and its generalization to families, we have $\mathcal{F}_{2l}$ is isomorphic to $M_{\frac{1}{2}}$. Since $w_{l+1}=\frac{1}{2}< \xi_{2l}$, then by Theorem \ref{intro,K-moduliparameterize,even,i} (3), we get the following corollary.

\begin{corollary}[=Theorem \ref{thm:main}(2)]\label{moduli,hypersurface,even}
Let $n=2l$ be an even positive integer.
Then the K-polystable members in the K-moduli space $\mathcal{F}_n$ are precisely all the hypersurfaces $H_{2(n+3)}$ in $\mathbb{P}(1,2,n+2,n+3)_{t,x,y,z}$ given by 
\begin{equation*}
  t^2 + z^2y + azx^{n+4} + a_0x^{2n+6} + a_1x^{2n+4}y + \cdots + a_{n+3}y^{n+3}=0
\end{equation*}
such that one of the following holds\begin{itemize}
         \item $\{a,a_0,\cdots,a_{l+1}\}$ are not all zero, $a_{l+2}$ is arbitrary, and $\{a_{l+3},\cdots,a_{2l+3}\}$ are not all zero;
         \item $a=a_0=\cdots=a_{l+1}=0$, $a_{l+2}\neq0$ and $a_{l+3}=\cdots=a_{2l+3}=0$.
     \end{itemize}
\end{corollary}

We descibe the wall crossing of K-moduli space $M_w$ for $w_{0}=\frac{l+1}{2l+3}\leq w< \xi_{2l}$.

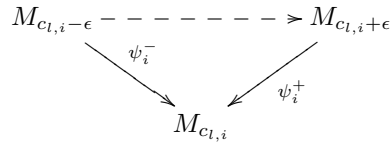
\begin{figure}
    \centering
    \caption{Wall crossing for K-moduli spaces at $c_{l,e}$}
    \label{Wall crossing at $c_{l,e}$,even}
\[
\xymatrix{
M_{c_{l,i}-\epsilon} \ar[dr]^{\psi^-_i} \ar@{-->}[rr] & & M_{c_{l,i}+\epsilon} \ar[dl]^{\psi^+_i} \\
& M_{c_{l,i}} & & 
} 
\]
\end{figure}

\begin{theorem}\label{wall crossings,even}\
When $w\in(w_0,\xi_{2l})$, the surface of each pair in $M_w$ is isomorphic to $W'$. 
    The wall crossing happens at $w_{i}$ for $1\leq 
    i\leq l+1$ (Figure \ref{Wall crossing at $c_{l,e}$,even}).
    \begin{enumerate}
        \item We look at the first wall $w_{0}=\frac{l+1}{2l+3}$. Let $\Sigma_{\frac{l+1}{2l+3}+\epsilon}$ be the locus parameterizing the pairs $(W',\frac{1}{2}H+(\frac{l+1}{2l+3}+\epsilon)D)$ such that $a_{2l+3}\neq 0$ and  $\{a,a_0,\cdots,a_{2l+2}\}$ are not all zero. The morphism $\psi^+_{0}: M_{\frac{l+1}{2l+3}+\epsilon} \rightarrow M_{\frac{l+1}{2l+3}}$ contracts $\Sigma_{\frac{l+1}{2l+3}+\epsilon}$ to a K-polystable point $(W',\frac{1}{2}H+\frac{l+1}{2l+3}\{z^2y+y^{2l+3}=0\})$.
    \item The birational morphism $\psi^-_i: M_{w_{i}-\epsilon} \rightarrow M_{w_{i}}$ is isomorphic away from the locus $\Sigma_{w_{i}-\epsilon}$ parameterizing the pairs $(W',\frac{1}{2}H+(w_{i}-\epsilon)D)$ such that $a=a_0=\cdots=a_{2l+2-i}=0$, $a_{2l+3-i}\neq0$, $\{a_{2l+4-i},\cdots,a_{2l+3}\}$ are not all zero. Moreover, the morphism $\psi^-_i$ contracts $\Sigma_{w_{i}-\epsilon}$ to a K-polystable point $(W',\frac{1}{2}H+w_{i}\{z^2y+u^iy^{2l+3-i}=0\})$.
    \item The birational morphism $\psi^+_i: M_{w_{i}+\epsilon} \rightarrow M_{w_{i}}$ is isomorphic away from the locus $\Sigma_{w_{i}+\epsilon}$
    parameterizing the pairs $(W',\frac{1}{2}H+(w_{i}+\epsilon)D)$ such that $\{a,a_0,\cdots,a_{2l+2-i}\}$ are not all zero, $a_{2l+3-i}\neq0$ and $a_{2l+4-i}=\cdots=a_{2l+3}=0$. Moreover, the morphism $\psi^+_i$ contracts $\Sigma_{w_{i}+\epsilon}$ to a K-polystable point $W',\frac{1}{2}H+w_{i}\{z^2y+u^{i}y^{2l+3-i}=0\}$.
    \end{enumerate}
\end{theorem}

\begin{proof}
We have $e=2l+3-i$. Then (1) is implied by Remark \ref{degenrates to K-poly,even} and Theorem \ref{intro,K-moduliparameterize,even,i}. And (2) and (3) are implied by Remark \ref{degenrates to K-poly,even} and Theorem \ref{intro,K-moduliparameterize,even,i}.
\end{proof}

\begin{example}\label{l=1,example}
    We consider the case $l=1$. Note that $W'=\mathbb{P}(1,1,l+1)$ and $D$ is a curve in $W'$ given by $$f=z^2y + azu^3 + a_0u^5 + a_1u^4y + a_2u^3y^2+a_3u^2y^3+a_4uy^4 + a_5y^5 = 0.$$ 
    We have $\frac{2l+5}{4l+6}=\frac{7}{10}$ and $\xi_{2l}=\frac{13}{22}$.
    Before $\xi_{2l}$, all the walls are $w_{0}$, $w_{1}$ and $w_{2}$. We have
  $$\frac{2}{5}=w_{0}<w_{1}=\frac{11}{26}<w_{2}=\frac{1}{2}<\xi_1<\frac{7}{10}.$$

We have the following description of Moduli spaces by Theorem \ref{intro,K-moduliparameterize,even,i}. 
When $\frac{2}{5}\leq c<\xi_{2l}$, each pair in $M_w$ has surface isomorphic to $W'$ (Table \ref{K-modulil=3}). When $\xi_{2l}\leq w<\frac{7}{10}$, we still have the same surface $W'$ but different curve equations; see Theorem \ref{intro,K-moduliparameterize,even,i}(6) or Section \ref{moduliafterxi_n}.

  \begin{center}
    \begin{longtable}{|l|l|l|l|}
 			\caption{Explicit description of $M_w$ for $w<\xi_{2l}$ when $l=1$}\label{K-modulil=3}\\
		\hline
		$M_w$   &coefficients   \\
				\hline

 $w=w_{0}=\frac{2}{5}$   & $a=a_0=\cdots=a_4=0$, $a_5\neq0$  \\
\hline

$\frac{2}{5}<w<\frac{11}{26}$   & $\{a,a_0,\cdots,a_4\}$ not all zero, $a_5\neq0$ \\
\hline

$w=w_{1}=\frac{11}{26}$  &  $\{a,a_0,\cdots,a_3\}$ not all zero, $a_4$ arbitrary, $a_5\neq0$ \\
&  or $a=a_0=\cdots=a_3=0$, $a_4\neq 0$ $a_5=0$.  \\
\hline

$\frac{11}{26}<w<\frac{1}{2}$  &  $\{a,a_0,\cdots,a_3\}$ not all zero, $\{a_4, a_5\}$ not all zero\\
\hline

$w=w_{2}=\frac{1}{2}$  &  $\{a,a_0,a_1,a_2\}$ not all zero, $a_3$ arbitrary, $\{a_4,a_5\}$ not all zero  \\
&  or $a=a_0=a_1=a_2=0$, $a_3\neq0$, $a_4=a_5=0$. \\
\hline

$\frac{1}{2}<w<\xi_2$  &  $\{a,a_0,a_1,a_2\}$ not all zero, $\{a_3,a_4,a_5\}$ not all zero \\
\hline
\end{longtable}
\end{center} 

\end{example}

\section{The last wall crossing for K-moduli spaces
}\label{moduliafterxi_n}

In this section we will describe the moduli spaces $M_w$ for $w\geq \xi_n$ where the last wall crossing happens.

\subsection{$n$ is odd}    
Let $(W,D)$ be a pair such that the Equation \eqref{defineK-moduli,odd} defining $D$ satisfies $a_{n+3}\neq 0$, $a\neq 0$ and curve $D$ has a type $A_{n+3}$-singularity. By Lemma \ref{section 5:a is nonzero Lemma}, such a curve $D$ is projectively equivalent to the curve $D_{\ss}$ as in Definition \ref{def:Dss}. 
By Proposition \ref{section5:main lemma}, the pairs $(W,\xi_n D_{\ss})$ is K-semistable, but it is not K-polystable. As we shall see in Proposition \ref{prop:xi_n-wall}, the K-polystable degeneration of $(W, \xi_n D_{\ss})$ is $(W_0, \xi_n D_0)$ where $W_0=\PP(1,n+2,
\frac{(n+3)^2}{2})_{x_0,x_1,x_2}$ and $D_0: x^2_2-x_0x^{n+4}_1=0$.

\begin{proposition}\label{prop:Dss-degenerate}
$W$ admits a $\bQ$-Gorenstein isotrivial degeneration to $W_0:=\PP(1,n+2,
\frac{(n+3)^2}{2})$.
\end{proposition}

\begin{proof}
  By \cite{AK16}, there exists a one-step mutation from $W$ to $W_0$ by taking $(\lambda_0, \lambda_1, \lambda_2):=(2,1,n+2)$. Here we use the fact that the cyclic quotient singularity of type $\frac{1}{2}(1, n+2)$ is an $A_1$-singularity and hence a $T$-singularity. Then by \cite[Theorem 4.2]{Por19} we know that $W$ admits a $\bQ$-Gorenstein isotrivial degeneration to $W_0$. 
\end{proof}

 Recall that $$\xi_n=\frac{n^3 + 11n^2 + 31n + 23}{2n^3 + 18n^2 + 50n + 42}.$$

\begin{lemma}\label{polystable pair at xi}
Let $n>0$ be an odd integer, then $(W_0,wD_0)$ is K-polystable if and only if $w=\xi_n$, where $W_0=\PP(1,n+2,
\frac{(n+3)^2}{2})_{x_0,x_1,x_2}$ and $D_0$ is defined by $x^2_2-x_0x^{n+4}_1=0$. Moreover, when  $w\neq \xi_n$, $(W_0,wD_0)$ is K-unstable.
\end{lemma}

\begin{proof}
We use Theorem \ref{Tone} to check K-polystability and look at curve $D_0$ in the affine chart $x_0\neq 0$ in $W_0$. The pair $(W_0, \xi_nD_0)$ is a complexity one $\mathbb{T}$-pair with an action by  $\mathbb{T}=\mathbb{G}_m$: $\lambda[x_0,x_1,x_2]=[x_0,\lambda^2x_1,\lambda^{n+4}x_2]$ for $\lambda \in \mathbb{G}_m$. 

The monomial valuation $v_{(2,n+4)}$ such that $v(x_1)=2$ and $v(x_2)=n+2$ is induced by the $\mathbb{T}$-action.
We first compute the $\beta$-invariant for $v_{(2,n+4)}$. 
Then by Lemma \ref{S-inv.}, we get the S-invariant for $v$:
\begin{equation*}
    S_{\mathcal{O}(1)}(v)=\frac{1}{3}(\frac{2}{n+2}+\frac{2(n+2)}{(n+3)^2})=\frac{2(2n^2+12n+17)}{3(n+3)^2(n+2)}.
\end{equation*}
We have $-K_{W_0}-wD_0=\mathcal{O}(n+3+\frac{(n+3)^2}{2}-w(n+3)^3)=r\mathcal{O}(1)$ and $S_{(W_0,wD_0)}(v)=\mathrm{deg}(-K_{W_0}-wD_0)\cdot S_{\mathcal{O}(1)}(v)=rS_{\mathcal{O}(1)}(v)$, where $r=\frac{(n+3)(n+5)}{2}-w(n+3)^2$. Also, the log discrepancy $A_{W_0}(v)=2+(n+4)=n+6$ and $A_{(W_0,wD_0)}(v)=A_{W_0}(v)-wv(D_0)=n+6-c(2n+8)$. Then we have $\beta$-invariant 
\begin{equation}\label{xi,betaD_0}
    \beta_{(W_0,wD_0)}(v)=A_{(W_0,wD_0)}(v)-S_{(W_0,wD_0)}(v)=0
\end{equation}
if and only if $w=\xi_n$.

There is no horizontal $\mathbb{T}$-invariant prime divisor by Remark \ref{remarkTone}. Next, we check $\beta_{(W_0,\xi_nD_0)}(E)> 0$ for every vertical $\mathbb{T}$-invariant prime divisor $E$. All the vertical
$\mathbb{T}$-invariant prime divisors are 
$\{[0;x_1;x_2]\},\{[x_0;0;x_2]\}, \{[x_0;x_1;0]\}, \{x^2_2-x_0x^{n+4}_1=0\}, \{x^2_2-ax_0x^{n+4}_1=0,a\neq 1\}$ under the $\bG_m$ action $\lambda[x_0,x_1,x_2]=[x_0,\lambda^2x_1,\lambda^{n+4}x_2]$.

Since $-K_{W_0}-wD_0=r\mathcal{O}(1)$, by Lemma \ref{Sinv,k}, we have $S_{(W_0,wD_0)}(\mathcal{O}(k))=\frac{r}{3k}$.

If $E$ is $\{[0;x_1;x_2]\}$, then $E=\mathcal{O}(1)$. Thus $S_{(W_0, wD_0)}(E)=\frac{r}{3}$. Note that $A_{(W_0,wD_0)}(E)=1-\mathrm{Coeff}_E(wD_0)=1-0=1$. Thus $\beta_{(W_0,wD_0)}(E)=A_{(W,wD_0)}(E)-S_{(W,wD_0)}(E)>0$ iff $w>\frac{n^2+8n+9}{2(n^2+6n+9)}$ which is satisfied by $w=\xi_n$.

If $E$ is $\{[x_0;0;x_2]\}$, then $E=\mathcal{O}(n+2)$.
Thus $S_{(W_0, wD_0)}(E)=rS_{\mathcal{O}(1)}(E)=\frac{r}{3(n+2)}$. Note that $A_{(W_0,wD_0)}(E)=1-\mathrm{Coeff}_E(wD_0)=1-0=1$. Thus $\beta_{(W_0,wD_0)}(E)=A_{(W,wD_0)}(E)-S_{(W,wD_0)}(E)>0$ iff $w>\frac{n^2+2n-15}{2(n^2+6n+9)}$ which is satisfied by $w=\xi_n$.

If $E$ is $\{[x_0;x_1;0]\}$, then $E=\mathcal{O}(\frac{(n+3)^2}{2})$. Thus $S_{(W_0, wD_0)}(E)=rS_{\mathcal{O}(1)}(E)=\frac{2r}{3(n+3)^2}=\frac{n+5}{3(n+3)}-\frac{2w}{3}$. Note that $A_{(W_0,wD_0)}(E)=1-\mathrm{Coeff}_E(wD_0)=1-0=1$. Thus $\beta_{(W_0,wD_0)}(E)=A_{(W,wD_0)}(E)-S_{(W,wD_0)}(E)>0$ when $w>0$ which is satisfied by $w=\xi_n$.

If $E$ is $\{x^2_2-ax_0x^{n+4}_1=0\}$, then $E=\mathcal{O}((n+3)^2)$. Thus $S_{(W_0, wD_0)}(E)=rS_{\mathcal{O}(1)}(E)=\frac{r}{3(n+3)^2}=\frac{n+5}{6(n+3)}-\frac{w}{3}$. If $a\neq 1$, we have $A_{(W_0,wD_0)}(E)=1-\mathrm{Coeff}_E(wD_0)=1-0=1$. Thus $\beta_{(W_0,wD_0)}(E)>0$ when $c>0$ which is satisfied by $c=\xi_n$. If $a=1$, we have $A_{(W_0,wD_0)}(E)=1-\mathrm{Coeff}_E(wD_0)=1-c$. Thus $\beta_{(W,wD_0)}(E)=A_{(W,wD_0)}(E)-S_{(W,wD_0)}(E)>0$ iff $w<\frac{5n+13}{4n+12}$ which is satisfied by $w=\xi_n$.

Hence $(W_0, \xi_nD_0)$ is K-polystable by Theorem \ref{Tone}.

Now we show $(W_0, wD_0)$ is K-semistable only when $w=\xi_n$. By Theorem \ref{Tone}, if $(W_0, wD_0)$ is K-semistable, we have $\beta_{(W_0, wD_0)}(v)=0$. By \eqref{xi,betaD_0}, $\beta_{(W_0, wD_0)}(v)=0$ is equivalent to $w=\xi_n$. Hence the conclusion holds.
\end{proof}

We aim to find all the new K-polystable pairs parametrized in $M_w$ for $w\in (\xi_n, \frac{n+5}{2n+6})$. Let $D_1$ be the curve in $W_0=\PP(1,n+2,
\frac{(n+3)^2}{2})_{x_0,x_1,x_2}$ defined by
\begin{equation}\label{defineD_1afterxi}
x^2_2-x_0x_1^{n+4}+b_{n+2}x_0^{2n+5}x_1^{n+2}+\cdots+b_{1}x_0^{(n+3)(n+2)+1}x_1+b_{0}x_0^{(n+4)(n+2)+1},
\end{equation}
where coefficients $b_0,\cdots,b_{n+2}$ are not all zero.

\begin{lemma}\label{lem:W0-general-curve}
A general curve in $|\cO_{W_0}((n+3)^2)|$ is  projectively equivalent to some $D_1$ from \eqref{defineD_1afterxi}. 
\end{lemma}

\begin{proof}
A general curve $C\in |\cO_{W_0}((n+3)^2)|$ has equation $C: c_0 x_2^2 + g(x_0,x_1) x_2 + h(x_0,x_1)=0$ where $g$ and $h$ are homogeneous of degree $\frac{(n+3)^2}{2}$ and $(n+3)^2$ respectively. Since $C$ is general, we have $c_0\neq 0$ and by rescaling we may assume $c_0 = 1$. Then after a change of coordinate $x_2 \mapsto x_2 + \frac{g}{2}$ we eliminate the $g x_2$-term, so we may assume $g = 0$. Then we can write $C:x_2^2 + \sum_{i=0}^{n+4} b_i x_0^{(n+3)^2 - (n+2)i} x_1^i = 0$. Since $C$ is general, we have $b_{n+4}\neq 0$ and we may assume $b_{n+4} = -1$ after a rescaling of $x_1$. Then after a change of coordinate $x_1\mapsto x_1 - \frac{b_{n+3}}{n+4} x_0^{n+2}$, we eliminates the $x_0^{n+3}x_1^{n+3}$-term. So we may assume $b_{n+3} = 0$. Thus $C$ has the form of $D_1$ as $b_0,\cdots, b_{n+2}$ are not all zero by $C$ being general.
\end{proof}

\begin{lemma}\label{lem:curves-W0}
Let $n>0$ be an odd integer. The pair $(W_0,wD_1)$ is K-polystable for $w\in(\xi_n,\frac{n+5}{2n+6})$.
\end{lemma}
\begin{proof}
In the affine chart $x_0=1$, the curve $D_1$ is given by
\begin{equation*}
x^2_2-x_1^{n+4}+b_{n+2}x_1^{n+2}+b_{n+1}x_1^{n+1}\cdots+b_{1}x_1+b_{0}.
\end{equation*}
Since $b_0,\cdots,b_{n+2}$ are not all zero, the curve
$D_1$ is either smooth or has the singularity of type $A_{m}$ for $0<m\leq n+2$. Thus $\mathrm{lct}(W_0,D_1)=\mathrm{min}\{\frac{m+3}{2m+2},1\}\geq \frac{n+5}{2n+6}$. This implies $(W_0,\frac{n+5}{2n+6}D_1)$ is log canonical, so is K-semistable.

Under the $\bG_m$ action $\lambda[x_0,x_1,x_2]=[\lambda^{n+4} x_0,\lambda^{-1} x_1,x_2],$ the curve $D_1$ degenerates to $D_0$ as $\lambda \rightarrow 0,$ where $D_0$ is defined by $x^2_2-x_0x^{n+4}_1=0$.
    By Lemma \ref{polystable pair at xi}, the pair $(W_0,\xi_nD_0)$ is K-polystable. 
    Then $(W_0,\xi_nD_1)$ is K-semistable by Theorem \ref{openness}. 
    We have rank $\Aut(W_0,D_0)=1$ since the pair $(W_0,D_0)$ has complexity one. Also we have $(W_0,\frac{n+5}{2n+6}D_0)$ is not K-semistable by Lemma \ref{polystable pair at xi}.
    Therefore, by Lemma \ref{theta,w,interpolation,polystable}, the pair $(W_0,wD_1)$ is K-stable for $w\in(\xi_n,\frac{n+5}{2n+6})$, so is K-polystable. 
\end{proof}

We now show that for $w>\xi_n$, the K-moduli space $M_w$ is isomorphic to a weighted blow up of GIT moduli space $M^{GIT}_{w}$ with the exceptional divisor parametrizing the pairs $(W_0,wD_1)$, where $D_1$ is defined by Equation \eqref{defineD_1afterxi}.

\begin{proposition}\label{prop:xi_n-wall}
There is a wall crossing diagram
\[
M_{\xi_n -\epsilon }\xrightarrow[\cong]{\phi^-} M_{\xi_n}\xleftarrow{\phi^+}M_{\xi_n +\epsilon}
\]
for $0<\epsilon \ll 1$ such that the following hold.
\begin{enumerate}
    \item $\phi^-$ is an isomorphism that only replaces $(W, D_{\ss})$ by $(W_0, D_0)$.
    \item $\phi^+$ is a divisorial contraction with only one exceptional divisor $E^+$. Moreover, we have $\phi^+(E^+)= \{[(W_0, D_0)]\}$, and $E^+$ parameterizes pairs $(W_0, D_1)$ as in Lemma \ref{lem:curves-W0}.
\end{enumerate}
\end{proposition}

\begin{proof}
(1) By Theorem \ref{isomorphism,GITandK} we know that $M_{\xi_n -\epsilon} \cong M_{\xi_n -\epsilon}^{GIT}$. Moreover, by Lemmas \ref{section 5:a is nonzero Lemma}, \ref{section5:a zero a_0 nonzero}, and \ref{section5:a a_0 zeros}  we know that a K-polystable pair $(X, D)\in M_{\xi_n -\epsilon}^{GIT}$ satisfies that $\lct(X,D)\geq \frac{n+5}{2n+6}$ if and only if $(X,D)$ is not isomorphic to $(W, D_{\ss})$. Therefore, by Lemma \ref{theta,w,interpolation,polystable} we know that $\phi^-$ is an isomorphism away from the point $[(W, D_{\ss})]$. By Lemma \ref{polystable pair at xi} we know that the pair $(W_0, \xi_n D_0)$ is K-polystable. We can compute that $h^0(W_0, \cO_{W_0}((n+3)^2)) =\frac{3n+17}{2}= h^0(W, \cO_W(2n+6))$ since $n$ is odd. Hence $(W_0, \xi_n D_0)$ is a K-polystable pair and a limit of pairs $(W, \xi_n D_t)$ in the K-moduli space $M_{\xi_n}$ by Proposition \ref{prop:Dss-degenerate} and Lemma \ref{polystable pair at xi}. Since $\phi^-$ only replaces $(W, D_{\ss})$, the K-polystable replacement can only be $(W_0, \xi_n D_0)$.

(2) From the above argument we know that $M_{\xi_n -\epsilon} \setminus \{[(W, D_{\ss})]\}$  admits an open immersion to $M_{\xi_n +\epsilon}$. It suffices to show that its complement $E^+$ is an irreducible divisor  that parameterizes pairs $(W_0, D_1)$. By Lemma \ref{lem:curves-W0} we know that  $\{(W_0, D_1)\}$ are parameterized by a weighted projective space $E_1^+$ of dimension $n+2$ such that $E^+_1\subset E^+$. Note that $(W_0, D_0)$ is in $M_{\xi_n}$ and $(W,D_{\ss})$ is replaced by $\phi^-$ which is an isomorphism away from the point $(W,D_{\ss})$. Hence $W_0$ is a $\bQ$-Gorenstein degeneration of $W$. Then we know that the locus $E_1^+$ of $\{(W_0, D_1)\}$ is  a prime divisor of $M_{\xi_n+\epsilon}$ contained in $E^+$ as the K-moduli space has dimension $n+3$. Thus it suffices to show that every pair $[(X,B)]\in (\phi^+)^{-1}([(W_0, D_0)])$ satisfies $[(X,B)]\in E_1^+$. 

We have a chain of isotrivial degenerations $W\rightsquigarrow X\rightsquigarrow W_0$. Since the locus $E_1^+$ of $\{(W_0, D_1)\}$ is a divisor in the K-moduli space $M_{\xi_n+\epsilon}$, we have $\dim E_1^++\dim \Aut(W_0)$ equals $h^0(W_0, \cO_{W_0}((n+3)^2))-1$. Also the $\dim \big( M_{\xi_n -\epsilon} \setminus \{[(W, D_{\ss})]\}\big)+\dim\Aut(W)$ equals $h^0(W, \cO_W(2n+6))-1$. Since $\dim E_1^++1=\dim M_{\xi_n +\epsilon}=\dim \big( M_{\xi_n -\epsilon}\setminus \{[(W, D_{\ss})]\}\big)$,
we have that $\dim\Aut(W_0) = \dim \Aut(W)+1$. Thus either $X\cong W$ or $X\cong W_0$. If $X\cong W$ (and for simplicity we assume $X=W$), then by Theorem \ref{thm:K-imply-GIT} we know that $(W,B) \in M_{\xi_n+\epsilon}^{GIT}$ which implies that $(W, B) \in M_{\xi_n -\epsilon}^{GIT} $ as $\xi_n$ is not a VGIT wall. Hence $(W, (\xi_n -\epsilon)B)$ is K-polystable by Theorem \ref{isomorphism,GITandK}. This is a contradiction to Lemma \ref{theta,w,interpolation,polystable} as $(W, \xi_n B)$ is strictly K-semistable. Thus we have $X\cong W_0$. 
Since a general curve in $|\cO_{W_0}((n+3)^2)|$ is projectively equivalent to some $D_1$ by Lemma \ref{lem:W0-general-curve}, we can find a (not necessarily isotrivial) one parameter degeneration $(W_0, D_{1,t})\rightsquigarrow (X, B)$ where $[(W_0,D_{1,t})]\in E_1^+$. By properness of $E_1^+$ we know that $[(X,B)]\in E_1^+$.
This finishes the proof of (2). 
\end{proof}

\begin{proposition}
    For any $w\in (\xi_n, \frac{n+5}{2n+6})$, we have $M_w\cong M_{\xi_n+\epsilon}$ with the same universal family.
\end{proposition}

\begin{proof}
By Lemma \ref{theta,w,interpolation,polystable} it suffices to show that any pair $[(X,B)]\in M_{\xi_n+\epsilon}$ satisfies $\lct(X,B)\geq \frac{n+5}{2n+6}$. By Proposition \ref{prop:xi_n-wall}, there are two cases: either $(X,(\xi_n-\epsilon)B)$ is K-polystable (which implies $(X,B)\cong (W,D)$ and $(X,B)\not\cong (W,D_{\ss})$), or $(X,B) \cong (W_0, D_1)$. In the former case, this follows from Lemmas \ref{section 5:a is nonzero Lemma}, \ref{section5:a zero a_0 nonzero}, and \ref{section5:a a_0 zeros}. In the latter case,  this follows from the fact that $D_1$ has at worst $A_{n+2}$-singularities as $b_0, \cdots, b_{n+2}$ are not all zero (see also Lemma \ref{lem:curves-W0}). 
\end{proof}

\subsection{$n$ is even}
Let $(W',\frac{1}{2}H+D)$ be a pair such that the Equation \eqref{eveneuation2} defining $D$ satisfies $a_{2l+3}\neq 0$, $a\neq 0$ and $D$ has a type $A_{2l+3}$-singularity. Similar to the $n$ odd case, by Lemma \ref{section5 even:a nonzero} the divisor $\frac{1}{2}H+D$ is projectively equivalent to $\frac{1}{2}H + D_{\ss}$ where $D_{\ss}$ is given by Definition \ref{def:Dss-even}.  By Proposition \ref{section5:main lemma}, the pair $(W',\frac{1}{2}H+\xi_{2l}D_{\ss})$ is K-semistable, but not K-polystable. We can show the K-polystable degeneration of this pair is $(W',\frac{1}{2}H_0+\xi_{2l}D_0)$, where $W'=\PP(1,1,l+1)_{u,y,z}$ and $H_0:y=0$ and $D_0:z^2y=zu^{l+2}$.

\begin{lemma}\label{polystable pair at xi-even}
Let $l\geq 1$ be an integer. Let $W':=\mathbb{P}(1,1,l+1)_{u,y,z}$, $H:=\{u=0\}$, and $D_{\ss}$ satisfies $a_{2l+3}\neq 0$, $a\neq 0$, and has a type $A_{2l+3}$-singularity.
Then the pair $(W',\frac{1}{2}H_0+\xi_{2l}D_0)$ is the K-polystable degeneration of $(W', \frac{1}{2}H + \xi_{2l} D_{\ss})$, where  $H_0:=\{y=0\}$ and $D_0:z^2y=zu^{l+2}$. Moreover, the pair $(W',\frac{1}{2}H_0+wD_0)$ is K-unstable when $w\neq \xi_{2l}$.
\end{lemma}
\begin{proof}
We first show that $(W',\frac{1}{2}H_0+\xi_{2l}D_0)$ is a special degeneration of $(W', \frac{1}{2}H + \xi_{2l} D_{\ss})$. In the affine coordinate $y=1$, by Section \ref{NonCYcase} we can write $D_{\ss}: (z-u^{l+2})^2 - (u-1)^{2l+4} =0$. After pullback under a change of coordinates $(u,z)\mapsto (u+1, 2z - u^{l+2} + (u+1)^{l+2})$, we have $(W', \frac{1}{2}H + \xi_{2l} D_{\ss}) \cong (W', \frac{1}{2}H_1 + \xi_{2l} D_0)$ where $H_1 = \{u+y=0\}$.  Thus under the $\bG_m$-action $t\cdot [u,y,z] = [tu, y, t^{l+2}z]$, $H_1$ degenerates to $H_0$, and $D_0$ is $\bG_m$-invariant. Thus the special degeneration part is proved.

Next, we use Theorem \ref{Tone} to check K-polystability of $(W',\frac{1}{2}H_0+wD_0)$ and look at curve $D_0$ in the affine chart $z\neq 0$ in $W'$. The pair $(W',\frac{1}{2}H_0+wD_0)$ is a complexity one $\mathbb{T}$-pair with an action by  $\mathbb{T}=\mathbb{G}_m$: $\lambda[u,y,z]=[\lambda u,\lambda^{l+2} y,z]$ for $\lambda \in \mathbb{G}_m$. 

The monomial valuation $v_{(1,l+2)}$ such that $v(u)=1$ and $v(y)=l+2$ is induced by the $\mathbb{T}$-action.
 We first compute the $\beta$-invariant for $v_{(1,l+2)}$. 
Then by Lemma \ref{S-inv.}, we get the S-invariant for $v$:
\begin{equation*}
    S_{\mathcal{O}(1)}(v)=\frac{1}{3}(1+l+2)=\frac{l+3}{3}.
\end{equation*}
We have $-K_{W'}-\frac{1}{2}H_0-wD_0=r\mathcal{O}(1)$ and $S_{(W',\frac{1}{2}H_0+wD_0)}(v)=\mathrm{deg}(-K_{W'}-\frac{1}{2}H_0-wD_0)\cdot S_{\mathcal{O}(1)}(v)=r\frac{l+3}{3}$, where $r=l+3-\frac{1}{2}-w(2l+3)$. Also, the log discrepancy $A_{W'}(v)=l+3$ and $A_{(W',\frac{1}{2}H_0+wD_0)}(v)=A_{W'}(v)-v(\frac{1}{2}H_0+wD_0)=l+3-\frac{1}{2}(l+2)-w(2l+3)$ since $v(H_0)=l+2$. Then we have $\beta$-invariant 
\begin{equation}\label{xi,betaH_0D_0,even}
    \beta_{(W',\frac{1}{2}H_0+wD_0)}(v)=A_{(W',\frac{1}{2}H_0+wD_0)}(v)-S_{(W',\frac{1}{2}H_0+wD_0)}(v)=0.
\end{equation}
if and only if $w=\xi_{2l}$.

There is no horizontal $\mathbb{T}$-invariant prime divisor by Remark \ref{remarkTone}. Next, we check $\beta_{(W',\frac{1}{2}H_0+wD_0)}(E)$ is positive for every vertical $\mathbb{T}$-invariant prime divisor $E$. All the vertical $\mathbb{T}$-invariant prime divisors are $\{0;y;z\}$, $\{[u;0;z]\}$, $\{u;y;0\}$, $\{zy-u^{l+2}=0\}$ and $\{zy-au^{l+2}=0, a\neq 1\}$ under the $\mathbb{G}_m$ action $\lambda[u,y,z]=[\lambda u,\lambda^{l+2} y,z]$ for $\lambda \in \mathbb{G}_m$. 

Since $-K_{W'}-\frac{1}{2}H_0-wD_0=r\mathcal{O}(1)$, by Lemma \ref{Sinv,k}, we have $S_{(W',\frac{1}{2}H_0+wD_0)}(\mathcal{O}(k))=\frac{r}{3k}$.

If $E$ is $\{0;y;z\}$, then $E=\mathcal{O}(1)$ and have $S_{(W',\frac{1}{2}H_0+wD_0)}(E)=\frac{r}{3}$. We have $A_{(W',\frac{1}{2}H_0+wD_0)}(E)=1-\mathrm{Coeff}_E(W',\frac{1}{2}H_0+wD_0)=1-0=1$. Then $\beta_{(W',\frac{1}{2}H_0+wD_0)}(E)>0$ iff $w>\frac{2l-1}{4l+6}$ which is satisfied by $w=\xi_{2l}$.

If $E$ is $\{u;0;z\}$, then $E=\mathcal{O}(1)$ and have $S_{(W',\frac{1}{2}H_0+wD_0)}(E)=\frac{r}{3}$. We have $A_{(W',\frac{1}{2}H_0+wD_0)}(E)=1-\mathrm{Coeff}_E(W',\frac{1}{2}H_0+wD_0)=1-\frac{1}{2}$. Then $\beta_{(W',\frac{1}{2}H_0+wD_0)}(E)>0$ iff $w>\frac{l+1}{2l+3}$ which is satisfied by $w=\xi_{2l}$.

If $E$ is $\{u;y;0\}$, then $E=\mathcal{O}(l+1)$ and have $S_{(W',\frac{1}{2}H_0+wD_0)}(E)=\frac{r}{3(l+1)}$. Since $D_0$ is defined by $z(zy-u^{l+2})=0$, we have $A_{(W',\frac{1}{2}H_0+wD_0)}(E)=1-w$. Then $\beta_{(W',\frac{1}{2}H_0+wD_0)}(E)>0$ iff $w<\frac{4l+1}{2l}$ which is satisfied by $w=\xi_{2l}$.

If $E$ is $\{zy-au^{l+2}=0\}$, then $E=\mathcal{O}(l+2)$ and have $S_{(W',\frac{1}{2}H_0+wD_0)}(E)=\frac{r}{3(l+2)}$. If $a=1$, we have $A_{(W',\frac{1}{2}H_0+wD_0)}(E)=1-w$. Then $\beta_{(W',\frac{1}{2}H_0+wD_0)}(E)>0$ when $w>0$ which is satisfied by $w=\xi_{2l}$.
If $a\neq 1$, $A_{(W',\frac{1}{2}H_0+wD_0)}(E)=1$. Then $\beta_{(W',\frac{1}{2}H_0+wD_0)}(E)>0$ iff $w<\frac{4l+7}{2l+6}$ which is satisfied by $w=\xi_{2l}$.

Hence $(W',\frac{1}{2}H_0+\xi_{2l}D_0)$ is K-polystable by Theorem \ref{Tone}.

Now we show $(W',\frac{1}{2}H_0+wD_0)$ is K-semistable only when $w=\xi_{2l}$. By Theorem \ref{Tone}, if $(W',\frac{1}{2}H_0+wD_0)$ is K-semistable, we have $\beta_{(W',\frac{1}{2}H_0+wD_0)}(v)=0$. By \eqref{xi,betaH_0D_0,even}, $\beta_{(W',\frac{1}{2}H_0+wD_0)}(v)=0$ is equivalent to $w=\xi_{2l}$. Hence the conclusion holds.
\end{proof}

We aim to find all the new K-polystable pairs parametrized in $M_w$. Let $D_1$ be the curve in $W'$ defined by
\begin{equation}\label{defineD_1afterxi-even}
z^2 y - z u^{l+2} + b_1 u^{2l+2} y + b_2 u^{2l+1} y^2 + \cdots + b_{2l+3} y^{2l+3} = 0,
\end{equation}
where coefficients $b_1,\cdots,b_{2l+3}$ are not all zero.

\begin{lemma}\label{lem:curves-W0,even}
Let $n=2l>0$ be an even integer. The pair $(W', \frac{1}{2}H_0 + wD_1)$ is K-polystable for $w\in(\xi_{2l},\frac{2l+5}{4l+6})$.
\end{lemma}

\begin{proof}
In the affine chart $y=1$, the curve $D_1$ is given by
\begin{equation*}
z^2  - z u^{l+2} + b_1 u^{2l+2}  + \cdots + b_{2l+3} = 0
\end{equation*}
Since $b_1,\cdots,b_{2l+3}$ are not all zero, the curve 
$D_1$ is either smooth or has the singularity of type $A_{m}$ for $0<m\leq 2l+2$. Thus $\mathrm{lct}_p(W',\frac{1}{2}H_0,D_1)=\mathrm{min}\{\frac{m+3}{2m+2},1\}\geq \frac{2l+5}{4l+6}$ for $p\not\in H_0$. If $p\in H_0$, then one can easily check that $(W',\frac{1}{2}H_0+\frac{2l+5}{4l+6}D_1)$ is lc at $p$. This implies $(W',\frac{1}{2}H_0+\frac{2l+5}{4l+6}D_1)$ is log canonical, so is K-semistable.

Under the $\bG_m$ action $\lambda[u,y,z]=[\lambda^{-1} u,  y, \lambda^{-l-2}z]$, the curve $D_1$ degenerates to $D_0$ as $\lambda \rightarrow 0,$ where $D_0$ is defined by $z^2 y - z u^{l+2}=0$.
    By Lemma \ref{polystable pair at xi-even}, the pair $(W',\frac{1}{2}H_0+ \xi_{2l} D_0)$ is K-polystable. 
    Then $(W',\frac{1}{2}H_0+ \xi_{2l} D_1)$ is K-semistable by Theorem \ref{openness}. 
    We have rank $\Aut(W',\frac{1}{2}H_0+ \xi_{2l} D_0)=1$ since the pair $(W',H_0+  D_0)$ has complexity one. Also we have $(W',\frac{1}{2}H_0+ \frac{2l+5}{4l+6} D_0)$ is not K-semistable by Lemma \ref{polystable pair at xi-even}.
    Therefore, by Lemma \ref{theta,w,interpolation,polystable}, the pair $(W',\frac{1}{2}H_0+wD_1)$ is K-stable for $w\in(\xi_{2l},\frac{2l+5}{4l+6})$, so is K-polystable.
\end{proof}

We now show that for $w>\xi_{2l}$, the K-moduli space $M_w$ is isomorphic to a weighted blow up of GIT moduli space $M^{GIT}_{w}$ with the exceptional divisor parametrizing the pairs $(W',\frac{1}{2}H_0+wD_1)$, where $D_1$ is defined by Equation \eqref{defineD_1afterxi-even}.

\begin{proposition}\label{prop:xi_n-wall,even}
There is a wall crossing diagram
\[
M_{\xi_{2l} -\epsilon }\xrightarrow[\cong]{\phi^-} M_{\xi_{2l}}\xleftarrow{\phi^+}M_{\xi_{2l} +\epsilon}
\]
for $0<\epsilon \ll 1$ such that the following hold.
\begin{enumerate}
    \item $\phi^-$ is an isomorphism that only replaces $(W', \frac{1}{2}H+D_{\ss})$ by $(W', \frac{1}{2}H_0+D_0)$.
    \item $\phi^+$ is a divisorial contraction with only one exceptional divisor $E^+$. Moreover, we have $\phi^+(E^+)= \{[(W', \frac{1}{2}H_0+D_0)]\}$, and $E^+$ parameterizes pairs $(W', \frac{1}{2}H_0+D_1)$ as in Lemma \ref{lem:curves-W0,even}.
\end{enumerate}
\end{proposition}

\begin{proof}
    (1) By Theorem \ref{isomorphism,GITandK,even} we know that $M_{\xi_{2l} -\epsilon} \cong M_{\xi_{2l} -\epsilon}^{GIT}$. Moreover, by Lemmas \ref{section5 even:a nonzero}, \ref{section5 even:a zero a_0 nonzero}, and \ref{section5 even:a and a_0 zero}  we know that a K-polystable pair $(W', \frac{1}{2}H+D)\in M_{\xi_{2l} -\epsilon}^{GIT}$ satisfies that $\lct(W',\frac{1}{2}H;D)\geq \frac{2l+5}{4l+6}$ if and only if $(W',\frac{1}{2}H+D)$ is not isomorphic to $(W', \frac{1}{2}H+D_{\ss})$. Therefore, by Lemma \ref{theta,w,interpolation,polystable} we know that $\phi^-$ is an isomorphism away from the point $[(W', \frac{1}{2}H+D_{\ss})]$. Thus (1) follows from Lemma \ref{polystable pair at xi-even}.

    (2) From the above argument we know that $M_{\xi_{2l} -\epsilon} \setminus \{[(W', \frac{1}{2}H+D_{\ss})]\}$  admits an open immersion to $M_{\xi_{2l} +\epsilon}$. It suffices to show that its complement $E^+$ is an irreducible divisor that parameterizes pairs $(W', \frac{1}{2}H_0+D_1)$. By Lemma \ref{lem:curves-W0,even} we know that  $\{(W', \frac{1}{2}H_0+D_1)\}$ are parameterized by a weighted projective space $E_1^+$ of dimension $2l+2$ such that $E_1^+\subset E^+$. 
    We know that the locus $E_1^+$ of $\{(W', \frac{1}{2}H_0+D_1)\}$ is a prime divisor of $M_{\xi_{2l}+\epsilon}$ contained in $E^+$ as the K-moduli space has dimension $2l+3$. Thus it suffices to show that every pair $[(X,\frac{1}{2}L+B)]\in (\phi^+)^{-1}([(W', \frac{1}{2}H_0+D_0)])$ satisfies $[(X,\frac{1}{2}L+B)]\in E_1^+$.

   We have a chain of isotrivial degenerations $W'\rightsquigarrow X\rightsquigarrow W'$. Hence $X\cong W'$, and for simplicity we assume $X= W'$. By the wall crossing morphism, there is a special degeneration $(W', \frac{1}{2}L+B)\rightsquigarrow (W', \frac{1}{2}H_0 + D_{0})$. In particular, we have $L\in |\cO_{W'}(1)|$ and $B\in |\cO_{W'}(2l+3)|$. Since $B$ specially degenerates to $D_0$, we know that $B$ contains monomials divisible by $z^2$. Under a suitable change of projective coordinates, we may assume that $B: z^2y + zg(u,y) + h(u,y) =0 $ and $L: \ell(u,y)=0$ where $\ell,g,h$ have degree $1, l+2, 2l+3$, respectively. Below we split the discussion into two cases. 
   
   Case 1: $\ell$ contains a $u$-term. Then under a change of projective coordinates $u\mapsto \ell(u,y)$ we may assume that $L= H = (u=0)$. Then we see that $B$ has the form of $D$ after a further change of coordinate $z\mapsto z + \frac{1}{2}g_1(u,y)$ with $\deg(g_1) = l+1$ such that $g-g_1 y$ only contains $u^{l+2}$-term. As a result, we have $(W', \frac{1}{2}L + B)\cong (W', \frac{1}{2}H + D)$ for some $D$ of the form \eqref{defineequationD,even}. By Theorem \ref{thm:K-imply-GIT-even} we know that $(W', \frac{1}{2}H + D)\in M_{\xi_{2l}+\epsilon}^{GIT}$ which implies that $(W', \frac{1}{2}H + D)\in M_{\xi_{2l}-\epsilon}^{GIT}$ as $\xi_{2l}$ is not a VGIT wall. Hence $(W', \frac{1}{2}H+(\xi_{2l}-\epsilon)D)$ is K-polystable. This is a contradiction to Lemma \ref{theta,w,interpolation,polystable} as $(W', \frac{1}{2}H+\xi_{2l}D)$ is strictly K-semistable.

   Case 2: $\ell$ only contains a $y$-term. Then we have $\ell$ divides the coefficient of the $z^2$-term in $B$. Let us consider a more general situation without change of coordinate.
   We look at the locus $\cE$ of $([\ell'],[f'])\in \cP:=|\cO_{W'}(1)|\times |\cO_{W'}(2l+3)|$ where $\ell'$ divides the coefficient of $z^2$-term in $f'$. Denote by $L':=(\ell'=0)$ and $B':=(f'=0)$. Let $\cE_y$ be the closed subset of $\cE$ where $[\ell']=[y]$ and $f'$ has no $z^2u$-term.  Clearly $\cE_y$ is isomorphic to a projective space, hence a smooth irreducible closed subvariety of $\cP$, and that $\cE = \GL_2 \cdot \cE_y$ where the $\GL_2$ action is the standard linear action on $(u,y)$. Thus $\cE$ is an irreducible closed subvariety (indeed a prime divisor) of $\cP$. Since every surface in $M_{\xi_{2l}+\epsilon}$ is isomorphic to $W'$, we know that the open subset $\cP^{+}\subset \cP$ parameterizing K-semistable pairs $(W', \frac{1}{2}L'+(\xi_{2l}+\epsilon)B')$ admits a surjective morphism $\sigma: \cP^{+} \to M_{\xi_{2l}+\epsilon}$. Let $\cE^{+}:=\cE\cap \cP^{+}$. From the discussion of part (1) and Case 1, we know that $\sigma(\cP^{+}\setminus \cE^{+}) \subset M_{\xi_{2l}+\epsilon}\setminus E^+$ and $\sigma(\cE^{+})\subset E^+$. Thus we have $\sigma(\cE^{+}) = E^+$ by surjectivity of $\sigma$, which implies that $E^+$ is irreducible by irreducibility of $\cE$. Since $E^+$ contains an irreducible divisor $E_1^+$ of $M_{\xi_n +\epsilon}$, we conclude that $E^+ = E_1^+$. This finishes the proof of (2). 
   \end{proof}

\begin{proposition}
    For any $w\in (\xi_{2l}, \frac{2l+5}{4l+6})$, we have $M_w\cong M_{\xi_{2l}+\epsilon}$ with the same universal family.
\end{proposition}

\begin{proof}
By Lemma \ref{theta,w,interpolation,polystable} it suffices to show that any pair $[(X,\frac{1}{2}L+B)]\in M_{\xi_{2l}+\epsilon}$ satisfies $\lct(X,\frac{1}{2}L;B)$ $\geq \frac{2l+5}{4l+6}$, where $L\in|\mathcal{O}(1)|$. By Proposition \ref{prop:xi_n-wall,even}, there are two cases: either $(X,\frac{1}{2}L+(\xi_{2l}-\epsilon)B)$ is K-polystable (which implies $(X,\frac{1}{2}L+B) \cong (W',\frac{1}{2}H + D)$ and $(X,\frac{1}{2}L+B)\not\cong (W',\frac{1}{2}H+D_{\ss})$), or $(X,\frac{1}{2}L+B) \cong (W', \frac{1}{2}H_0+D_1)$. In the former case, this follows from Lemmas \ref{section5 even:a nonzero}, \ref{section5 even:a zero a_0 nonzero}, and \ref{section5 even:a and a_0 zero}. In the latter case,  this follows from the fact that $D_1$ has at worst $A_{2l+2}$-singularities as $b_1, \cdots, b_{2l+3}$ are not all zero (see also Lemma \ref{lem:curves-W0,even}). 
\end{proof}

\subsection{Relation with moduli of marked hyperelliptic curves}\label{sec:hyperelliptic}

Here we discuss the relation of our moduli spaces with certain moduli spaces of marked hyperelliptic curves.

\subsubsection{$n$ is odd}
Let $D\subset W=\bP(1,2,n+2)$ be a general curve of the form \eqref{equation of D}. By analyzing local equations, it is easy to see that a general $D$ is smooth. Consider the projection $\pr: D\to \bP^1$ given by $\pr([x,y,z]) := [x^2, y]$. In the affine chart $x=1$, the map $\pr$ is simply the projection of
\[
D: z^2 y + az + a_{0} + a_1 y + \cdots + a_{n+3} y^{n+3} = 0
\]
onto the $y$-axis. Thus we see that $\pr: D\to \bP^1$ has degree $2$, which implies that $D$ is a hyperelliptic curve. Moreover, there are precisely $n+5$ branch points, which implies that $D$ has genus $g = \frac{n+3}{2}$. Now, taking intersection of $D$ with the line $H_x:(x=0)$, we obtain two marked points $p_1 = [0,0,1]$ and $p_2 = [0, 1, \sqrt{-a_{n+3}}]$ on $D$ where $p_1$ is the only singularity of $W$ contained in $D$. It is also clear that $\pr(p_1) = [1,0]$ and $\pr(p_2) = [0,1]$ and hence $p_2$ is  a Weierstra\ss{} point of $D$. Conversely, given a general hyperelliptic curve of genus $g=\frac{n+3}{2}$ and two marked point $p_1$ and $p_2$ where $p_2$ is a Weierstra\ss{} point, we can embed $D$ inside $W$ using the $\bQ$-divisor $L:=\frac{1}{n+2}p_1 + p_2$, where $x,y,z$ corresponds to a generating set of the graded algebra $\oplus_{m= 0}^{\infty} H^0(D, \cO_D(\lfloor mL\rfloor))$.
Therefore, we see that the K-moduli space $M_w$ for $w\in (\frac{n+2}{2n+6}, \frac{n+5}{2n+6})$ is birational to the divisor $H_{g,2}^{\rm W}$ in $H_{g,2}$ parameterizing  marked hyperelliptic curves $(D,p_1+p_2)$ where the second marked point $p_2$ is Weierstra\ss{}, where $H_{g,k}$ denotes the coarse moduli space of smooth hyperelliptic curve of genus $g$ with $k$ distinct marked points.

After the last K-moduli wall, i.e. $w\in (\xi_n, \frac{n+5}{2n+6})$, we have new pairs $(W_0,D_1)$ in the K-moduli space $M_w$ where $W_0=\bP(1, n+2, \frac{(n+3)^2}{2})$ and $D_1$ is given by \eqref{defineD_1afterxi}. The projection $\pr_0: D_1\to \bP^1$ defined by $[x_0, x_1, x_2]\mapsto [x_0^{n+2}, x_1]$ is clearly a degree $2$ map. Thus $D_1$ is also a hyperelliptic curve. Then the line $H_x$ on $W$ degenerate to the divisor $(x_0^{\frac{n+3}{2}} = 0)$. Thus both $p_1$ and $p_2$ on $D$ degenerates to the point $p_0=[0,1,0]$ on $D_1$. Since this point is a Weierstra\ss{} point on $D_1$, we call the marked curve $(D_1, 2p_0)$ a \emph{Weierstra\ss{} double marked hyperelliptic curve}. Althought this marked curve is not Deligne-Mumford stable, it has a unique stable replacement by gluing a rational tail with two marked points to $D_1$ at $p_0$. These stable replacements form a divisor in the closure $\overline{H}_{g,2}^{\rm W}$ of $H_{g,2}^{\rm W}$ in the Deligne-Mumford compactification $\overline{H}_{g,2}$, which coincides with the K-moduli wall crossing description from Proposition \ref{prop:xi_n-wall}.

It is an interesting problem to relate our K-moduli wall crossing results to some form of Hassett-Keel program for the moduli space  $\overline{H}_{g,2}^{\rm W}$.

\subsubsection{$n$ is even}
Let $n = 2l$. 
Let $D\subset W'=\bP(1,1,l+1)$ be a general curve of the form \eqref{defineequationD,even}. By analyzing local equations, it is easy to see that a general $D$ is smooth. Consider the projection $\pr: D\to \bP^1$ given by $\pr([u,y,z]) := [u, y]$. In the affine chart $u=1$, the map $\pr$ is simply the projection of
\[
D: z^2 y + az + a_{0} + a_1 y + \cdots + a_{2l+3} y^{2l+3} = 0
\]
onto the $y$-axis. Thus we see that $\pr: D\to \bP^1$ has degree $2$, which implies that $D$ is a hyperelliptic curve. Moreover, there are precisely $2l+4$ branch points, which implies that $D$ has genus $g = l+1$. Now, taking intersection of $D$ with the line $H:(u=0)$, we obtain three marked points $p_1 = [0,0,1]$, $p_2 = [0, 1, \sqrt{-a_{2l+3}}]$, and $p_3 = [0, 1, -\sqrt{-a_{2l+3}}]$ on $D$ where $p_1$ is the only singularity of $W'$ contained in $D$. It is also clear that $\pr(p_1) = [1,0]$ and $\pr(p_2) =\pr(p_3) = [0,1]$ and hence $p_2$ and $p_3$ are conjugate points of $D$. Conversely, given a general hyperelliptic curve of genus $g=l+1$ and three marked point $p_1,p_2,p_3$ where $p_2$ and $p_3$ are conjugate points, we can embed $D$ inside $W'$ using the $\bQ$-divisor $L:=\frac{1}{l+1}p_1 + p_2 + p_3$, where $u,y,z$ corresponds to a generating set of the graded algebra $\oplus_{m= 0}^{\infty} H^0(D, \cO_D(\lfloor mL\rfloor))$ and $(u=0)=p_2+p_3$.
Therefore, we see that the K-moduli space $M_w$ for $w\in (\frac{l+1}{2l+3}, \frac{2l+5}{4l+6})$ is birational to the divisor $H_{g,3}^{\rm C}/\bmu_2$ in $H_{g,3}/\bmu_2$ parameterizing  marked hyperelliptic curves $(D,p_1+p_2+p_3)$ where the second and third marked points $p_2,p_3$ are conjugate, and the $\bmu_2$-action swaps $p_2$ and $p_3$.

After the last K-moduli wall, i.e. $w\in (\xi_{2l}, \frac{2l+5}{4l+6})$, we have new pairs $(W',\frac{1}{2}H_0 + D_1)$ in the K-moduli space $M_w$ where $H_0:(y=0)$ and $D_1$ is given by \eqref{defineD_1afterxi-even}. The projection $\pr: D_1\to \bP^1$ defined by $[u,y,z]\mapsto [u,y]$ is clearly a degree $2$ map. Thus $D_1$ is also a hyperelliptic curve. Then the line $H$ on $W'$ degenerates to the line $H_0$. Thus after possibly swapping $p_2$ and $p_3$, we have that both $p_1$ and $p_2$ on $D$ degenerate to the point $p_{1,0}=[0,0,1]$ on $D_1$, and $p_3$ degenerates to the point $p_{3,0} = [1,0,0]$ such that $p_{1,0}$ and $p_{3,0}$ are conjugate.  Althought this marked curve $(D_1, 2p_{1,0}+p_{3,0})$ is not Deligne-Mumford stable, it has a unique stable replacement by gluing a rational tail with two marked points to $D_1$ at $p_{1,0}$. These stable replacements form a divisor in the closure $\overline{H}_{g,3}^{\rm C}/\bmu_2$ of $H_{g,3}^{\rm C}/\bmu_2$ in the quotient of the Deligne-Mumford compactification $\overline{H}_{g,3}/\bmu_2$, which coincides with the K-moduli wall crossing description from Proposition \ref{prop:xi_n-wall}.

It is an interesting problem to relate our K-moduli wall crossing results to some form of Hassett-Keel program for the moduli space  $\overline{H}_{g,3}^{\rm C}/\bmu_2$.

\section{Computations using Abban-Zhuang Method}\label{NonCYcase}
Let $W = \PP(1,2,n+2)_{x,y,z}$ and $W'=\PP(1,1,l+1)_{u,y,z}$ be the weighted projective planes where $n>0$ is an odd integer and $l$ is a positive integer. 
Subsequently, we use $X$ to represent the weighted projective planes $W$ and $W'$ when properties or conditions are satisfied by both the weighted projective planes.

Recall that (Definition \ref{def:Dss} and Definition \ref{def:Dss-even}) in affine chart $U_y=(y=1)$, the curve $D_{\ss}$ is given by 
\begin{align*}
    (z - x^{n+4})^2 - (x^2 - 1)^{n+4} = 0 & \qquad \text{on $W$}\\
    (z - u^{l+2})^2 - (u - 1)^{2l+4} = 0 & \qquad \text{on $W'$},
\end{align*}
respectively. 
Then these curves are only singular at the point $\mathfrak{P} = (1:1:1)$. We set 
\begin{equation}\label{definexi}
    \xi = 
    \begin{dcases}
        \xi_n=\frac{n^3 + 11n^2 + 31n + 23}{2n^3 + 18n^2 + 50n + 42} & \text{for $W$},\\
        \xi_{2l}=\frac{2l^2 + 8l + 3}{4l^2 + 12l + 6} & \text{for $W'$}.
    \end{dcases}
\end{equation}

We use $\Delta_{w}$ to represent the effective divisors $wD_{\ss}$ and $\frac{1}{2}H_u + wD_{\ss}$, where $w$ is a positive rational number. We set
\begin{equation*}
    c =
    \begin{dcases}
        n+5 - w(2n+6) & \text{for $W$},\\
        l+\frac{5}{2} - w(2l+3) & \text{for $W'$}.
    \end{dcases}
\end{equation*}
Then $-(K_X + \Delta_w)\sim_{\QQ} \mathcal{O}_X(c)$.

\begin{proposition}\label{section5:main lemma}
    If $\frac{1}{2}<w\leq \xi$ then the log pair $(X, \Delta_w)$ is K-semistable. Moreover, if $w > \xi$ then the log pair $(X, \Delta_w)$ is not K-semistable.
\end{proposition}

We need several lemmas. In particular,
we estimate the $\delta$-invariant of the log pair $(X, \Delta_w)$ at the point $\mathfrak{P}$ by the following lemma.
\begin{lemma}\label{specialpoint}
    If $\frac{1}{2} <w\leq \xi$ then $\delta_{\mathfrak{P}}(X, \Delta_w) \geq 1$. Moreover, if $w> \xi$ then the log pair $(X, \Delta_w)$ is not K-semistable.
\end{lemma}

\begin{proof}
By the coordinate changes $(\tilde{x}, \tilde{z}) =(x-1, z-x^{n+4})$ on $W$ and $(\tilde{u}, \tilde{z}) =(u-1, z-u^{l+2})$ on $W'$, we obtain the following:
\begin{align*}
    \tilde{z}^2 - \tilde{x}^{n+4}(\tilde{x}+2)^{n+4} = 0 & \qquad \text{on $W$}\\
    \tilde{z}^2 - \tilde{u}^{2l+4} = 0 & \qquad \text{on $W'$}.
\end{align*}
Let $\phi\colon Y\to X$ be the weighted blow-up at the point $(\tilde{x}, \tilde{z}) = (0, 0)$ with weights $(\wt(\tilde{x}), \wt(\tilde{z})) = (2, n+4)$ for $W$ and $(\wt(\tilde{u}), \wt(\tilde{z})) = (1, l+2)$ for $W'$. And let $M\in \mathcal{O}_W(\wt(z))$ be the effective divisor defined by
\begin{align*}
    z+\sum_{i=1}^{\frac{n+3}{2}}\binom{\frac{n+3}{2}}{i}(-1)^i x^{n+4-2i} y^{i-1} = 0 & \qquad \text{on $W$}\\
    z-\sum_{i=1}^{l+2}\binom{l+2}{i}(u-y)^{l+2-i}y^{i-1} = 0 & \qquad \text{on $W'$}.
\end{align*}
On the affine open set $U_y$ defined by $y=1$ we have
\begin{multline*}
    z+\sum_{i=1}^{\frac{n+3}{2}}\binom{\frac{n+3}{2}}{i}(-1)^i x^{n+4-2i} \\
    = (z - x^{n+4}) + x(x^2 - 1)^{\frac{n+3}{2}} 
    = \tilde{z} + (\tilde{x} + 1)\tilde{x}^{\frac{n+3}{2}}(\tilde{x}+2)^{\frac{n+3}{2}} \qquad \text{on $W$},
\end{multline*}
\begin{multline*}
    z-\sum_{i=1}^{l+2}\binom{l+2}{i}(u-1)^{l+2-i} \\
    = z-\sum_{i=1}^{l+2}\binom{l+2}{i}\tilde{u}^{l+2-i} 
    = \tilde{z} + u^{l+2}-\sum_{i=1}^{l+2}\binom{l+2}{i}\tilde{u}^{l+2-i} = \tilde{z} + \tilde{u}^{l+2} \qquad \text{on $W'$}.
\end{multline*}
From these equations we obtain the following:
\begin{equation*}
    \phi^*(M) = \widetilde{M} + mE
\end{equation*}
where $m$ are $n+3$ for $W$ and $l+2$ for $W'$, $\widetilde{M}$ is the strict transform of $M$ and $E$ is the $\phi$-exceptional divisor. Since $\widetilde{M}^2 < 0$ the divisor $\phi^*(M) - tE$ is pseudoeffective if and only if $t\in [0, m]$. For $W$ the positive and the negative parts of the Zariski decomposition of $\phi^*(M) - tE$ are
\begin{equation*}
    P_n(t) = 
    \begin{dcases}
        \widetilde{M} + (n+3 - t)E & \text{for~} t\leq \frac{(n+2)(n+4)}{n+3},\\
        ((n+3)\widetilde{M} + E)(n+3 - t) & \text{for~} \frac{(n+2)(n+4)}{n+3}\leq t\leq n+3,
    \end{dcases}
\end{equation*}
\begin{equation*}
    N_n(t) = 
    \begin{dcases}
        0 & \text{for~} t\leq \frac{(n+2)(n+4)}{n+3},\\
        (1 - (n+3)(n+3 - t))\widetilde{M} & \text{for~} \frac{(n+2)(n+4)}{n+3}\leq t\leq n+3,
    \end{dcases}
\end{equation*}
respectively. The volume of $\phi^*(M) - tE$ is
\begin{equation*}
    \vol(\phi^*(M) - tE) = 
    \begin{dcases}
        \frac{n+2}{2} - \frac{t^2}{2n+8} & \text{for~} t\leq \frac{(n+2)(n+4)}{n+3},\\
        (n+3-t)^2\frac{n+2}{2} & \text{for~} \frac{(n+2)(n+4)}{n+3}\leq t\leq n+3.
    \end{dcases}
\end{equation*}
Then the S-invariant in (\ref{pre:S-invariant}) is
\begin{equation*}
    S_{W, \Delta_w}(E) = \frac{c(2n^2 + 12n + 17)}{3(n+2)(n+3)}.
\end{equation*}
Similarly, for $W'$ the positive and the negative parts of the Zariski decomposition of $\phi^*(M) - tE$ are
\begin{equation*}
    P_l(t) = 
    \begin{dcases}
        \widetilde{M} + (l+2 - t)E & \textrm{for~} t\leq l+1,\\
        (l+2 - t)(\widetilde{M} + E) & \textrm{for~} l+1 \leq t\leq l+2,\\
    \end{dcases}
\end{equation*}

\begin{equation*}
    N_l(t) =
    \begin{dcases}
        0 & \textrm{for~} t\leq l+1,\\
        (1-(l+2-t))\widetilde{M} & \textrm{for~} l+1 \leq t\leq l+2,\\
    \end{dcases}
\end{equation*}
respectively.
Then the volume of $\phi^*(M) - tE$ is
\begin{equation*}
    \vol(\phi^*(M) - tE) = 
    \begin{dcases}
        l+1 - \frac{t^2}{l+2} & \textrm{for~} t\leq l+1,\\
        (l+2-t)^2\frac{l+1}{l+2} & \textrm{for~} l+1 \leq t\leq l+2.\\
    \end{dcases}
\end{equation*}
Then the S-invariant in (\ref{pre:S-invariant}) is
\begin{equation*}
    S_{W', \Delta_w}(E) = \frac{c(2l+3)}{3l+3}.
\end{equation*}
The log discrepancy of the log pair $(X, \Delta_w)$ is
\begin{equation*}
    A_{X,\Delta_w}(E) = 
    \begin{dcases}
        n+6 - w(2n+8) & \text{on $W$},\\
        l+3 - w(2l+4) & \text{on $W'$}.
    \end{dcases}
\end{equation*}
We can see that $\frac{A_{X,\Delta_w}(E)}{S_{X,\Delta_w}(E)}\geq 1$ if and only if
\begin{equation*}
    w\leq \xi = 
    \begin{dcases}
        \frac{n^3 + 11n^2 + 31n + 23}{2n^3 + 18n^2 + 50n + 42} & \text{for $W$},\\
        \frac{2l^2 + 8l + 3}{4l^2 + 12l + 6} & \text{for $W'$}.
    \end{dcases}
\end{equation*}

Meanwhile, to obtain the function in (\ref{pre:function h(u)}) we need the following:
\begin{equation*}
    P_n(t)\cdot E =
    \begin{dcases}
        \frac{t}{2n+8}& \text{for~} t\leq \frac{(n+2)(n+4)}{n+3},\\
        \frac{n+2}{2}(n+3-t) & \text{for~} \frac{(n+2)(n+4)}{n+3}\leq t\leq n+3,
    \end{dcases}
\end{equation*}
\begin{equation*}
    P_l(t)\cdot E =
    \begin{dcases}
        \frac{t}{l+2} & \textrm{for~} t\leq l+1,\\
        \frac{l+1}{l+2}(l+2 - t)& \textrm{for~} l+1 \leq t\leq l+2.\\
    \end{dcases} 
\end{equation*}
We first consider a point $\msq\in E$ satisfying $\msq\not\in \widetilde{M}$. It implies that $\ord_{\msq}(N(t)|_E) = 0$. Then we have
\begin{equation*}
    S(X_{\bullet, \bullet}^E;\msq) = \frac{2c}{\wt(z)\vol(M)}\int_0^{\infty} \frac{1}{2}(P(u)\cdot E)^2\dd u =
    \begin{dcases}
        \frac{c}{6(n+3)} & \text{on $W$},\\
        \frac{c}{3(l+2)} & \text{on $W'$}.
    \end{dcases}
\end{equation*}

We next consider a point $\msq\in E\cap \widetilde{M}$. Since $\ord_{\msq}(N(t)|_E) = (N(t)\cdot E)_{\msq}$ we have the following:
\begin{equation*}
    (N_n(t)\cdot E)_{\msq} = 
    \begin{dcases}
        0 & \text{for~} t\leq \frac{(n+2)(n+4)}{n+3},\\
        (1 - (n+3)(n+3 - t))\frac{n+3}{2n+8} & \text{for~} \frac{(n+2)(n+4)}{n+3}\leq t\leq n+3,
    \end{dcases}
\end{equation*}
\begin{equation*}
    (N_l(t)\cdot E)_{\msq} = 
    \begin{dcases}
        0 & \textrm{for~} t\leq l+1,\\
        (1-(l+2-t)) & \textrm{for~} l+1 \leq t\leq l+2,\\
    \end{dcases}
\end{equation*}
We have
\begin{multline*}
    \frac{2c}{\wt(z)\vol(M)}\int_0^{\infty}(P_n(u)\cdot E)(\ord_{\msq}(N_n(u)|_E)) \dd u \\= \frac{4c}{(n+2)^2}\int_{\frac{(n+2)(n+4)}{n+3}}^{n+3} \frac{(n+2)(n+3)}{4(n+4)}(n+3-u)(1 - (n+3)(n+3 - u)) \dd u \\= \frac{c}{6(n+2)(n+3)(n+4)}
\end{multline*}
on $W$ and
\begin{multline*}
    \frac{2c}{\wt(z)\vol(M)}\int_0^{\infty}(P_l(u)\cdot E)(\ord_{\msq}(N_l(u)|_E)) \dd u \\
    = \frac{2c}{(l+1)^2}\int_{l+1}^{l+2} \frac{l+1}{l+2}(l+2 - u)(1-(l+2-u))\dd u = \frac{c}{3(l+1)(l+2)}
\end{multline*}
on $W'$. Thus we can obtain the following:
\begin{equation*}
    S(X_{\bullet, \bullet}^E;\msq) = 
    \begin{dcases}
        \frac{c(n+3)}{6(n+2)(n+4)} & \text{on $W$},\\
        \frac{c}{3(l+1)} & \text{on $W'$}.
    \end{dcases}
\end{equation*}
If $\msq$ is a quotient singular point of type $\frac{1}{r}(a,b)$ then the log discrepancy of the log pair $(E, \Phi)$ along the divisor $\msq$ is
\begin{equation*}
    A_{E, \Phi}(\msq) = \frac{1}{r} - \ord_{\msq}(\Phi) = \frac{1}{r} - (\Delta_w \cdot E)_{\msq}
\end{equation*}
where $\Phi$ is the divisor such that $(K_Y + \Delta_w + E)|_E = K_E + \Phi$. The log discrepancies for the log pairs $(E, \Phi)$ for the cases $(W, \Delta_w)$ and $(W', \Delta_w)$ are
\begin{equation*}
    A_{E, \Phi}(\msq) = 
    \begin{dcases}
        \frac{1}{2} &\textrm{if~} \msq \textrm{~is the singular point of type~} \frac{1}{2}(1,1),\\
        1 - w &\textrm{if~} \msq = E\cap \widetilde{D}_{\ss},\\
        \frac{1}{n+4} &\textrm{if~} \msq = E\cap \widetilde{M},\\
        1 & \textrm{otherwise},
    \end{dcases}
\end{equation*}
and
\begin{equation*}
    A_{E, \Phi}(\msq) =
    \begin{dcases}
        1 - w&\text{if $\msq\in \widetilde{M}\cap E$},\\
        \frac{1}{l+2}&\text{if $\msq\in \widetilde{H}_{\tilde{u}}\cap E$},\\
        1 &\text{otherwise},
    \end{dcases}
\end{equation*}
respectively. We can see that, if $w\in (\frac{1}{2}, 1)$ then 
\begin{equation*}
    \frac{A_{X, \Delta_w}(E)}{S_{X, \Delta_w}(E)}\leq \frac{A_{E, \Phi}(\msq)}{S(X_{\bullet, \bullet}^E;\msq)}.
\end{equation*}
It implies that $\delta_{\msp}(X, \Delta_w) = \frac{A_{X, \Delta_w}(E)}{S_{X, \Delta_w}(E)}\geq 1$ if $w\leq \xi$.

\end{proof}

\begin{lemma}\label{delta-px}
    If $w>\frac{1}{2}$ then $\delta_{\msp_z}(X, \Delta_w) > 1$.
\end{lemma}

\begin{proof}
    On the affine open set $U_{\msp_z}$ defined by $W\setminus H_z$, we can choose $x$ and $y$ as orbifold coordinates. On the affine open set $U_{\msp_z}$ defined by $W'\setminus H_z$, we can choose $u$ and $y$ as orbifold coordinates. Let $\phi\colon Y\to X$ be the weighted blow-up at the point $\msp_z$ with weights $(x,y) = (1,2)$ for $W$ and $(u,y) = (1,1)$ for $W'$.

    Let $H_1$ be the hyperplane section cutting out by $x=0$ in $W$ and by $u=0$ in $W'$. Then we have
    \begin{align*}
        \phi^*(H_1) = \widetilde{H}_1 + \frac{1}{n+2}E & \qquad \text{for $W$,}\\
        \phi^*(H_1) = \widetilde{H}_1 + \frac{1}{l+1}E & \qquad \text{for $W'$}
    \end{align*}
    Where $E$ is the $\phi$-exceptional divisor and $\widetilde{H}_1$ is the strict transform of $H_1$.
    
    We consider an effective divisor $\phi^*(H_1) - tE$ where $t$ is a positive constant. Since $\widetilde{H}_1^2 = 0$ the pseudoeffective threshold $\tau(E)$ with respect to $-(K_X + \Delta_w)$ is
    \begin{equation*}
        \tau(E) = 
        \begin{dcases}
            \frac{c}{n+2} & \text{for $W$},\\
            \frac{c}{l+1} & \text{for $W'$}.
        \end{dcases}
    \end{equation*}
    We write that the positive and negative parts of the Zariski decomposition of $\phi^*(-(K_X + \Delta_w)) - tE$ are
    \begin{equation*}
        P(t) = \phi^*(cH_1) - tE, \qquad N(t) = 0
    \end{equation*}
    for $t\in (0, \tau(E))$. The S-invariant of $E$ in (\ref{pre:S-invariant}) is given by
    \begin{equation*}
        S_{X, \Delta_w}(E) = \frac{1}{c^2H_1^2}\int_0^{\tau(E)}c^2H_1^2 + u^2E^2 \dd u =
        \begin{dcases}
            \frac{2c}{3n+6} & \text{on $W$},\\
            \frac{2c}{3l+3} & \text{on $W'$}.
        \end{dcases}
    \end{equation*}
    The function in (\ref{pre:function h(u)}) is given by
    \begin{equation*}
        h(u) = \frac{1}{2}(uE^2)^2 = 
        \begin{dcases}
            \frac{(n+2)^2u^2}{8} & \text{on $W$},\\
            \frac{(l+1)^2u^2}{2} & \text{on $W'$}.
        \end{dcases}
    \end{equation*}
    Then the function in (\ref{pre:restricted volume part of AZ}) is given by
    \begin{equation*}
        S(X_{\bullet,\bullet}^E;\msq) = 
        \begin{dcases}
            \frac{c}{6} & \text{on $W$},\\
            \frac{c}{3} & \text{on $W'$}
        \end{dcases}
    \end{equation*}
    for a point $\msq\in E$. Meanwhile, the log discrepancy of $E$ is
    \begin{equation*}
        A_{X, \Delta_w}(E) =
        \begin{dcases}
            \frac{3-2w}{n+2} & \text{on $W$},\\
            \frac{3-2w}{2l+2} & \text{on $W'$}.
        \end{dcases}
    \end{equation*}
    We also have the following:
    \begin{equation*}
        A_{E,\Phi}(\msq) = 
        \begin{dcases}
            1 & \textrm{if~}\msq\in E\setminus \widetilde{D}_{\ss}\cup \widetilde{H}_1\\
            \frac{1}{2} & \textrm{if~}\msq \in E\cap \widetilde{H}_1\\
            1 - w  & \textrm{if~}\msq \in E\cap \widetilde{D}_{\ss}
        \end{dcases}
    \end{equation*}
    where $\Phi$ is the different such that $(K_Y + \widetilde{\Delta}_w + E)|_E = K_E + \Phi$ and $\widetilde{\Delta}_w$ is the strict transform of $\Delta_w$. Thus we can see that $\frac{A_{X, \Delta_w}(E)}{S_{X, \Delta_w}(E)}>1$ and $\frac{A_{E, \Phi}(\msq)}{S(X_{\bullet,\bullet}^E;\msq)}>1$. By the Abban-Zhuang theorem we see that $\delta_{\msp_z}(X, \Delta_w)>1$.
\end{proof}

We use $H_1$ to represent the hyperplane section cutting out by $x=0$ in $W$ and by $u=0$ in $W'$. 

\begin{lemma}\label{Hx}
    For every point $\msp\in H_1\setminus\{\msp_z\}$, if $w>\frac{1}{2}$ then $\delta_{\msp}(X, \Delta_w)\geq 1$.
\end{lemma}

\begin{proof}
    To obtain a contradiction we assume that $\delta_{\msp}(X, \Delta_w)< 1$ for a point $\msp\in H_1\setminus\{\msp_z\}$. Then there is a $m$-basis type divisor $M$ such that the log pair $(X, \Delta_w + M)$ is not log canonical at the point $\msp$. We write
    \begin{align*}
        c_0D_{\ss} + M &= c_0D_{\ss} + \alpha H_1 + \Xi \qquad \text{on $W$},\\
        c_0D_{\ss} + \frac{1}{2}H_1 + M &= c_0D_{\ss} + \left(\frac{1}{2} + \alpha'\right) H_1 + \Xi \qquad \text{on $W'$}
    \end{align*}
    where $\Xi$ is an effective $\QQ$-divisor such that the support of $\Xi$ does not contain $H_1$.
    Since $-(K_X + \Delta_w) - tH_1 \equiv (c-t)H_1$ for a positive constant $t$, the pseudoeffective threshold $\tau(H_1)$ of $H_1$ with respect to $-(K_X + \Delta_w)$ is $c$. From the function in (\ref{pre:S-invariant}) the S-invariant of $H_1$ is given by
    \begin{equation*}
        S_{X, \Delta_w}(H_1) = \frac{1}{c^2H_1^2}\int_0^{\tau(H_1)} (c-u)^2H_1^2 \dd u = \frac{c}{3}.
    \end{equation*}
    It implies that $\alpha < 1$ and $\frac{1}{2}+\alpha' < 1$. By the inversion of adjunction formula we have
    \begin{align*}
        w + \frac{c}{2n + 4} - \frac{\alpha}{2n + 4}
        \geq (wD_{\ss}\cdot H_1)_{\msp} + \Xi\cdot H_1 
        \geq \mult_{\msp}(wD_{\ss} + \Xi)|_{H_1} > 1 
        & \text{~on $W$, where $\msp\neq \msp_y$},\\
        \frac{c}{2n+4} - \frac{\alpha}{2n+4}
        \geq (wD_{\ss}\cdot H_1)_{\msp} + \Xi\cdot H_1 
        \geq \mult_{\msp}(wD_{\ss} + \Xi)|_{H_1} > \frac{1}{2} 
        & \text{~on $W$, where $\msp = \msp_y$},\\
        w + \frac{c}{l+1} - \frac{\alpha'}{l+1}
        \geq (wD_{\ss}\cdot H_1)_{\msp} + \Xi\cdot H_1 
        \geq \mult_{\msp}(wD_{\ss} + \Xi)|_{H_1} > 1 
        & \text{~on $W'$}.
    \end{align*}
   They imply that $w<\frac{1}{2}$. It is a contradiction. 
\end{proof}

\begin{lemma}\label{outsideHx}
    For every point $\msp\in X\setminus H_1\cup \{\mathsf{P}\}$ if $w>\frac{1}{2}$ then $\delta_{\msp}(X, \Delta_w)>1$.
\end{lemma}

\begin{proof}
    By a suitable coordinate change we can assume that $\msp = (1:0:0)$. We use $\mathcal{L}\subset |\mathcal{O}_X(\wt(z))|$ to represent the sublinear system defined by $ax^{\wt(z)-\wt(y)}y + bz = 0$ on $W$ and the sublinear system defined by $au^{\wt(z)-\wt(y)}y + bz = 0$ on $W'$, where $(a:b)\in \PP^1$. There is an effective divisor $M\in \mathcal{L}$ such that $(D_{\ss}\cdot M)_{\msp} \leq 1$. 
    Since $-(K_X + \Delta) \equiv \frac{c}{\wt(z)}M$, the function in (\ref{pre:S-invariant}) is given by
    \begin{equation*}
        S_{X, \Delta_w}(M) = \frac{\wt(z)^2}{c^2M^2}\int_0^{\frac{c}{\wt(z)}} \left(\frac{c}{\wt(z)} - t\right)^2M^2 \dd t = \frac{c}{3\wt(z)} =
        \begin{dcases}
            \frac{c}{3n+6} & \text{on $W$},\\
            \frac{c}{3l+3} & \text{on $W'$}.
        \end{dcases}
    \end{equation*}
    The function in (\ref{pre:function h(u)}) is given by
    \begin{equation*}
        h(u) = \frac{1}{2}\left(\left(\frac{c}{\wt(z)} - u\right)M^2\right)^2.
    \end{equation*}
    Then the function in (\ref{pre:restricted volume part of AZ}) is given by
    \begin{equation*}
        S(X_{\bullet,\bullet}^C;\msp) = \frac{2\wt(z)^2}{c^2M^2}\int_0^{\frac{c}{\wt(z)}}\frac{1}{2}\left(\left(\frac{c}{\wt(z)} - u\right)M^2\right)^2 \dd u = 
        \begin{dcases}
            \frac{c}{6} & \text{on $W$},\\
            \frac{c}{3} & \text{on $W'$}.
        \end{dcases}
    \end{equation*}
    Meanwhile, we have $A_{X, \Delta_w}(M) = 1$ and $A_{X, \Phi}(\msp) = 1 - (wD_{\ss}\cdot M)_{\msp}\geq 1 - w$ where $\Phi$ is an effective divisor on $M$ such that $(K_X + \Delta_w + M)|_M = K_M + \Phi$. Thus we can see that $\frac{A_{X, \Delta_w}(M)}{S_{X, \Delta_w}(M)}>1$ and $\frac{A_{M, \Phi}(\msp)}{S(X_{\bullet,\bullet}^C;\msp)}>1$. By Theorem \ref{AbbanZ} we see that $\delta_{\msp}(X, \Delta_w)>1$.
\end{proof}

\begin{proof}[\bf Proof of Proposition \ref{section5:main lemma}]
By Lemma \ref{specialpoint}, Lemma \ref{delta-px}, Lemma \ref{Hx} and Lemma \ref{outsideHx}.
For every point $\msp\in W$ we see that $\delta_{\msp}(X, \Delta_w)\geq 1$ if $\frac{1}{2}<w\leq \xi$. By Lemma \ref{specialpoint}, if $w> \xi$, the log pair $(X, \Delta_w)$ is not K-semistable. Therefore, we obtain the result.

\end{proof}

Now we prove Proposition \ref{t<=(n+3)/2}. We need Lemma \ref{alla,a_izero,notK-s.s}, Lemma \ref{all,a_izero,notK-s.s} and Lemma \ref{0<=t<=(n+3)/2notk-s.s.}.
Recall that $W = \PP(1,2,n+2)_{x,y,z}$ where $n$ is a positive odd number.  
\begin{lemma}\label{alla,a_izero,notK-s.s}
Assume $D$ is a curve in $W$ defined by $z^2y = 0.$
Then the log pair $(W, wD)$ is not K-semistable for any $\frac{n+2}{2n+6}\leq w\leq \frac{n+5}{2n+6}$. 
\end{lemma}
\begin{proof}
    We have $A_{W, wD}(H_z) = 1 - 2w$ and $S_{H_x}(H_x)=\frac{1}{3}$ by Lemma \ref{Sinv,k}.
Then $S_{W, wD}(H_z) = \frac{n+5 - w(2n+6)}{3(n+2)}$. Therefore, the $\beta$-invariant
\begin{equation*}
    \beta_{W, wD}(H_z) = 1 - 2w - \frac{n+5 - w(2n+6)}{3(n+2)} = \frac{2n+1}{3(n+2)} - \frac{4n+6}{3(n+2)}w.
\end{equation*}
We have $\beta_{W, wD}(H_z))\geq 0$ if and only if $\frac{2n+1}{4n+6}\geq w$. Since $\frac{n+2}{2n+6} >\frac{2n+1}{4n+6}$, the log pair $(W, wD)$ is not K-semistable for any $\frac{n+2}{2n+6}\leq w\leq \frac{n+5}{2n+6}$.
\end{proof}

\begin{lemma}\label{all,a_izero,notK-s.s}
Assume $D$ is a curve in $W$ defined by $z^2y+azx^{n+4}= 0,$ where $a\neq 0$.
Then the log pair $(W, wD)$ is not K-semistable for any $\frac{n+2}{2n+6}\leq w\leq \frac{n+5}{2n+6}$. 
\end{lemma}

\begin{proof}
Let $v_{(1,n+4)}$ be a monomial valuation such that $v(x)=1$ and $v(z)=n+4$. By Lemma \ref{S-inv.}, we get the S-invariant for $v$:
\begin{equation*}
    S_{\mathcal{O}(1)}(v)=\frac{1}{3}(1+\frac{n+4}{n+2})=\frac{2n+6}{3(n+2)}.
\end{equation*}
Since $-K_W-wD=\mathcal{O}(r)$, where $r=n+5-w(2n+6)$, we have $S_{W,wD}(v)=\frac{2n+6}{3(n+2)}r$. The log discrepancy 
$A_{W,wD}(v)=A_W(v)-wv(D)=n+5-2w(n+4)$.
Then 
\begin{equation*}
    \beta_{W, wD}(E) = A_{W, wD}(E) - S_{W, wD}(E) = \frac{n(n+5)}{6(n+2)} - w\frac{n^2 + 6n + 6}{3(n+2)}
\end{equation*}
Note that $\beta(E)\geq 0$ if and only if
\begin{equation*}
    \frac{n^2 + 5n}{2n^2 + 12n + 12} \geq w.
\end{equation*}
Since $\frac{n+2}{2n+6}>\frac{n^2 + 5n}{2n^2 + 12n + 12}$, the log pair $(W, wD)$ is not K-semistable for any $\frac{n+2}{2n+6}\leq w\leq \frac{n+5}{2n+6}$.
\end{proof}

\begin{lemma}\label{0<=t<=(n+3)/2notk-s.s.}
    Let the curve $D$ in $W$ is defined by
\begin{equation*}
    z^2y + azx^{n+4}+a_0x^{2n+6}+\cdots +a_t x^{2n+6-2t}y^t = 0,
\end{equation*} where $0\leq t\leq\frac{n+3}{2}$, $a_t\neq0$ and $a,a_0,\cdots,a_{t-1}$ are arbitrary constants. Then the log pair $(W, wD)$ is not K-semistable for $\frac{n+2}{2n+6}\leq w<\frac{n+5}{2n+6}$. 
\end{lemma}
\begin{proof}
Let $v=v_{(1,n+3-t)}$ be a monomial valuation such that $v(x)=1$ and $v(z)=n+3-t$. By Lemma \ref{S-inv.}, we get the S-invariant for $v$:
\begin{equation*}
    S_{\mathcal{O}(1)}(v)=\frac{1}{3}(1+\frac{n+3-t}{n+2})=\frac{2n+5-t}{3(n+2)}.
\end{equation*}
Since $-K_W-wD=\mathcal{O}(r),$ where $r=n+5-w(2n+6)$, we get $S_{W,wD}(v)=\frac{2n+5-t}{3(n+2)}r$.
Since $v(D)=v(z^2y)=2(n+3-t)$, the log discrepancy
\begin{equation*}
    A_{W, wD}(v) =A_{W}(v)-wv(D)= n+4-t-2w(n+3-t).
\end{equation*}
Let $w_1=\frac{n+4-t}{2(n+3-t)}$. Since $0\leq t\leq \frac{n+3}{2}$, we have $\frac{n+2}{2n+6}<w_1\leq \frac{n+5}{2n+6}$. The log discrepancy $A_{W, wD}(v)\geq 0$ iff $w_1\geq w$. Then if $w>w_1$, the log pair $(W, wD)$ is not log canonical, thus not K-semistable.

We first consider the case $t=0$. Then 
\begin{equation*}
    \beta_{W, wD}(v) = n+4-2w(n+3)-\frac{2n+5}{3(n+2)}r= \frac{n^2 + 3n -1}{6(n+2)} - \frac{(n+3)(n+1)}{3(n+2)}w.
\end{equation*}
We have $\beta_{W, wD}(v)\geq 0$ iff $w\leq \frac{n^2 + 3n -1}{2(n+3)(n+1)}$. Since $\frac{n^2 + 3n -1}{2(n+3)(n+1)} < \frac{n+2}{2n+6}$, thus if $t=0$, the log pair $(W, wD)$ is not K-semistable for $\frac{n+2}{2n+6}\leq w\leq \frac{n+5}{2n+6}$.

We next consider the case that $t\geq 1$. Then
\begin{equation*}
    \beta_{W, wD}(v) = \frac{n^2 - 2nt + 3n -t -1}{6(n+2)} - \frac{n^2 - 2nt + 4n -3t +3}{3(n+2)}w.
\end{equation*}
We set $a(w) = \frac{n+4-t}{2} - w(n+3-t)$ and $s(w) = \frac{2n-t+5}{6(n+2)}(n+5 - w(2n+6))$. 
Let $t_1=\frac{n^2+3n-1}{2n+1}$ and $t_2=\frac{n^2+4n+3}{2n+3}.$ We have $n^2-2nt+3n-t-1\geq0$ iff $t\leq t_1, $
                and  $n^2-2nt+4n-3t+3\geq0$ iff $t\leq t_2.$
                Note that $t_1<t_2<(n+3)/2.$

When $t_1<t\leq t_2$, we get $\beta_{W, wD}(v)<0$ for all $w\geq 0$. Hence $(W, wD)$ is not K-semistable for $\frac{n+2}{2n+6}\leq w\leq \frac{n+5}{2n+6}$.

When $t_2<t\leq \frac{n+3}{2}$, we have $\beta_{W, wD}(v)$ increase as $w$ increase.
We assume $t_2<t< \frac{n+3}{2}$. Recall, if $t<\frac{n+3}{2}$, we have $s(w_1)<0$ and $\frac{n+2}{2n+6}<w_1<\frac{n+5}{2n+6}$. So $\beta_{W, wD}(v)\leq \beta_{W, w_1D}(v)=a(w_1)-s(w_1)<0$ for $w\leq w_1$.
Moreover, if $w>w_1$, the log discrepancy $A_{W, wD}(v)< 0$. So the log pair $(W, wD)$ is not log canonical. Thus $(W, wD)$ is not K-semistable for $\frac{n+2}{2n+6}\leq w< \frac{n+5}{2n+6}$.

When $t=\frac{n+3}{2}$, we have $w_1=\frac{n+5}{2n+6}$ and $s(w_1)=0$. So $\beta_{W, wD}(v)< \beta_{W, w_1D}(v)=a(w_1)-s(w_1)=0$ for $w<w_1$.
Thus $(W, wD)$ is not K-semistable for $\frac{n+2}{2n+6}\leq w< \frac{n+5}{2n+6}$.

 When $1\leq t \leq t_1$, we have $a(0)>s(0)$. Thus $\beta_{W, wD}(v)\geq 0$ only holds for
\begin{equation*}
    w\in \left[0, \frac{n^2 - 2nt + 3n -t -1}{2(n^2 - 2nt + 4n -3t +3)}\right].
\end{equation*} 
We have 
\begin{equation*}
    \frac{n^2 - 2nt + 3n -t -1}{2(n^2 - 2nt + 4n -3t +3)} < \frac{n+2}{2n+6} \qquad\text{iff} \qquad 0<\frac{3n-3t+9}{2(n+3)(n^2 - 2nt + 4n -3t +3)},
\end{equation*} which is true since $1\leq t \leq t_1<t_2$.
This implies $w<\frac{n+2}{2n+6}$. Therefore the log pair $(W, wD)$ is not K-semistable for $\frac{n+2}{2n+6}\leq w\leq \frac{n+5}{2n+6}$.

\end{proof}

Now we prove Proposition \ref{t<=(n+3)/2,even}. We need Lemma \ref{alla,a_izero,notK-s.s,even}, Lemma \ref{all,a_izero,notK-s.s,even} and Lemma \ref{0<=t<=(n+3)/2notk-s.s.,even}.
Recall that $W' = \PP(1,1,l+1)_{u,y,z}$ and $H$ is the hyperplane section defined by $u=0$, where $l$ is a positive integer.  

\begin{lemma}\label{alla,a_izero,notK-s.s,even}
Assume $D$ is a curve in $W'$ defined by $z^2y = 0.$
Then the log pair $(W', \frac{1}{2}H + wD)$ is not K-semistable for $\frac{l+1}{2l+3}\leq w\leq \frac{2l+5}{4l+6}$. 
\end{lemma}
\begin{proof}
  By Lemma \ref{Sinv,k}, we have
\begin{equation*}
    S_{W',\frac{1}{2}H + wD}(H_z) = \frac{1}{3(l+1)}\left(l+\frac{5}{2} - w(2l+3)\right)
\end{equation*}
Meanwhile, the log discrepancy of $H_z$ is $A_{W',\frac{1}{2}H + wD}(H_z) = 1 - 2w$. Then we have
\begin{equation*}
    \beta_{W',\frac{1}{2}H + wD}(H_z) = \frac{4l+1}{6l+6} -\frac{4l+3}{3l+3}w\geq0
\end{equation*}
if and only if $w\leq \frac{4l+1}{8l+6}$ holds. The conclusion holds since $\frac{4l+1}{8l+6}<\frac{l+1}{2l+3}$.
\end{proof}

\begin{lemma}\label{all,a_izero,notK-s.s,even}
Assume $D$ is a curve in $W'$ defined by $z^2y+azu^{l+2}= 0,$ where $a\neq 0$.
Then the log pair $(W', \frac{1}{2}H + wD)$ is not K-semistable for any $\frac{l+1}{2l+3}\leq w\leq \frac{2l+5}{4l+6}$. 
\end{lemma}
\begin{proof}
Let $v_{(1,l+2)}$ be a monomial valuation such that $v(u)=1$ and $v(z)=l+2$. By Lemma \ref{S-inv.}, we get the S-invariant for $v$:
\begin{equation*}
    S_{\mathcal{O}(1)}(v)=\frac{1}{3}(1+\frac{l+2}{l+1})=\frac{2l+3}{3(l+1)}.
\end{equation*}
Since $-K_{W'}-(\frac{1}{2}H+wD)=\mathcal{O}(r)$, where $r=l+3-(\frac{1}{2}+w(2l+3))=l+\frac{5}{2}-w(2l+3)$, we have $S_{W',\frac{1}{2}H+wD}(v)=\frac{2l+3}{3(l+1)}r$. Since $v(\frac{1}{2}H+wD)=\frac{1}{2}+w(2l+4)$, the log discrepancy 
$A_{W',\frac{1}{2}H+wD}(v)=A_{W'}(v)-v(\frac{1}{2}H+wD)=l+\frac{5}{2}-w(2l+4)$.
Then
\begin{equation*}
    \beta_{W',\frac{1}{2}H + wD}(v) = A_{W',\frac{1}{2}H + wD}(v) - S_{W',\frac{1}{2}H + wD}(v) = \frac{2l^2+5l}{6l+6} - \frac{(2l^2 + 6l + 3)w}{3l+3}.
\end{equation*}
Note that $\beta(v)\geq 0$ if and only if
\begin{equation*}
    w\leq \frac{2l^2 + 5l}{4l^2 + 12l + 6}.
\end{equation*}
Thus the conclusion holds 
since $\frac{2l^2 + 5l}{4l^2 + 12l + 6} < \frac{l+1}{2l+3}$.

\end{proof}

\begin{lemma}\label{0<=t<=(n+3)/2notk-s.s.,even}
    Let the curve $D$ in $W'$ is defined by
\begin{equation*}
    z^2y + azu^{l+2} + a_0u^{2l+3} + \cdots + a_tu^{2l+3-t}y^t=0,
\end{equation*} where $t\leq l+1$ is a nonnegative integer, $a_t\neq0$ and $a,a_0,\cdots,a_{t-1}$ are arbitrary constants. Then the log pair $(W', \frac{1}{2}H + wD)$ is not K-semistable for $\frac{l+1}{2l+3}\leq w\leq \frac{2l+5}{4l+6}$. 
\end{lemma}
\begin{proof}
Let $v_{(2,2l+3-t)}$ be a monomial valuation such that $v(u)=2$ and $v(z)=2l+3-t$. By Lemma \ref{S-inv.}, we get the S-invariant for $v$:
\begin{equation*}
    S_{\mathcal{O}(1)}(v)=\frac{1}{3}(2+\frac{2l+3-t}{l+1})=\frac{4l+5-t}{3(l+1)}.
\end{equation*}
Since $-K_{W'}-(\frac{1}{2}H+wD)=\mathcal{O}(r)$, where $r=l+3-(\frac{1}{2}+w(2l+3))=l+\frac{5}{2}-w(2l+3)$, we have $S_{W',\frac{1}{2}H+wD}(v)=\frac{4l+5-t}{3(l+1)}r$. Since $v(\frac{1}{2}H + wD)=1+wv(D)=1+2w(2l+3-t)$, the log discrepancy is
\begin{equation*}
    A_{W',\frac{1}{2}H + wD}(v) = 2l+5-t-1-2w(2l+3-t)=2l+4-t-2w(2l+3-t).
\end{equation*}
Then we have the $\beta$-invariant
\begin{equation*}
    \beta_{W',\frac{1}{2}H + wD}(v) = \frac{-8l^2 - 4lt + 6l - t + 11}{6(l + 1)} +\frac{(-4l^2 + 4lt - 8l + 3t - 3)w}{3(l + 1)}.
\end{equation*}
Since $1\leq t\leq l+1$, the inequality
\begin{equation*}
    -4l^2 + 4lt - 8l + 3t - 3<0
\end{equation*}
holds. Let $t_1=\frac{-8l^2+6l+11}{4l+1}$. We have $-8l^2 - 4lt + 6l - t + 11<0$ iff $t_1<t$. We have $t_1<1$ for $1<l$, and $t_1\geq1$ only when $t=1$.
Thus for $1<l$, we have $t_1<1\leq t$. Hence $\beta_{W',\frac{1}{2}H + wD}(v)<0$ for any $w\geq 0$.

If $l=1$, then $t=1$ or $2$. We have $ \beta_{W',\frac{1}{2}H + wD}(v) = \frac{9-5t}{12} + \frac{-15+7t}{6}w<0$ for $\frac{l+1}{2l+3}\leq w\leq \frac{2l+5}{4l+6}$. Therefore when $1\leq t\leq l+1$, the log pair $(W', \frac{1}{2}H + wD)$ is not K-semistable for $\frac{l+1}{2l+3}\leq w\leq \frac{2l+5}{4l+6}$.

We next consider the case that $t=0$. The log discrepancy is $A_{W',\frac{1}{2}H + wD}(v) =2l+4-2w(2l+3).$ And $S_{W',\frac{1}{2}H+wD}(v)=\frac{4l+5}{3(l+1)}(l+\frac{5}{2}-w(2l+3))$.
Thus the $\beta$-invariant $\beta_{W',\frac{1}{2}H + wD}(v)=A_{W',\frac{1}{2}H + wD}(v)-S_{W',\frac{1}{2}H+wD}(v)\geq 0$ 
iff $w\leq \frac{4l^2+6l-1}{8l^2+16l+6}$.
 Since $\frac{4l^2+6l-1}{8l^2+16l+6}<\frac{l+1}{2l+3}$ for $l\geq 1$, we have $\beta_{W',\frac{1}{2}H + wD}(v) < 0$ for $\frac{l+1}{2l+3}\leq w\leq \frac{2l+5}{4l+6}$. Hence we get the conclusion.

\end{proof}

\bibliographystyle{alpha}
\bibliography{ref}

\end{document}